\documentclass{article}
\usepackage{graphicx} 
\usepackage{float}
\usepackage[a4paper,left = 2cm, right = 2cm, top=1.8cm, bottom=1.8cm]{geometry}
\usepackage{indentfirst}
\usepackage{tikz}
\usepackage{amssymb}
\usepackage{amsmath,mathrsfs,amsfonts}
\allowdisplaybreaks[4]
\usepackage{graphics}
\usepackage{amsthm}
\usepackage{subfigure}
\usepackage{array}
\usepackage{setspace}
\usepackage{setspace}
\usepackage[mathscr]{eucal}
\usepackage{hyperref}

\def\m{\overline{m}}
 \def\diss{{\rm diss}}
  \def\mod{{\rm mod~}}
\begin{document}
\newtheorem{corollary}{Corollary}[section]
\newtheorem{lemma}[corollary]{Lemma}
\newtheorem{theorem}[corollary]{Theorem}

\title{The maximum number of maximum dissociation sets  in   potted graphs }
\author{Zejun Huang\footnote{Corresponding author.
Email: mathzejun@gmail.com (Huang), szuxinwei@outlook.com (Zhang)}, Xinwei Zhang}
\date{School of Mathematical Sciences, Shenzhen University, Shenzhen 518060, China}
\maketitle
 \begin{abstract}
     \large{ A potted graph is a   unicyclic graph such that its cycle contains a unique vertex with degree larger than 2.  Given a graph $G$, a subset of $V(G)$ is a dissociation set of $G$ if it induces a subgraph with maximum degree at most one. A maximum dissociation set is a dissociation set with maximum cardinality. In this paper, we determine the maximum number of  maximum dissociation sets in a potted graph of order $n$ which contains a fixed cycle. Extremal potted graphs attaining this maximum number are also characterized. }

     \vspace{12pt}

 \end{abstract}

\section{Introduction}

\large {
Throughout this paper, we consider undirected simple   graphs. We follow the terminology in \cite{BM}. Given a graph $G$, we denote by $V(G)$ its vertex set. An {\it independent set} of a graph $G$ is a set of vertices such that no two of them are adjacent in $G$. A  {\it  maximal independent set} is an independent set not contained in any other independent sets, which is also called a   {\it  maximum independent set} if it has the maximum cardinality. The {\it independence number} of a graph $G$ is the cardinality of a maximum independent set, which is denoted by $\alpha(G)$.

 A {\it dissociation set} in a graph is a subset of vertices which induces a subgraph with vertex degree at most one. Similarly, a {\it maximal dissociation set} of a graph  is a dissociation set not contained in  any other dissociation sets,  and  a {\it maximum dissociation set} is a dissociation set with the maximum cardinality. The {\it dissociation number} of a graph $G$, denoted by ${\rm diss}(G)$, is the cardinality of a maximum dissociation set in $G$. A {\it $k$-path vertex cover} of a graph $G$ is a subset
of vertices $S\subset V(G) $ such that every path on $k$ vertices in $G$ contains a vertex from $S$.
 The dissociation set is a natural generalization of the independent set and it is closely related to the 3-path vertex cover, which has  applications in network security; see \cite{CHZ,ref01,Tujianhua, ref02}.

 Erd\H{o}s and Moser raised the problem of determining the maximum number of  maximal independent sets in a general graph  of order $n$ as well as the graphs attaining this maximum number, which was solved by Erd\H{o}s, and Moon and Moser independently; see \cite{MM}.
 Since then,
determining the maximum number  of maximal or maximum independent sets in special types of graphs has attracted lots of researches.  Along this line, the maximum numbers of maximal independent sets in trees, forests, connected graphs, $k$-connected graphs, unicyclic connected graphs,  connected graphs with at most $r$ cycles, bipartite graphs, and connected triangle-free graphs have been determined; see \cite{CJ,ref16,ref18,JC,ref19,ref21,Liu,Liu.J,ref22,ref23,ref25}. The maximum numbers of maximum independent sets in general trees, trees with a given independence number,   subcubic trees with a given independence number  and connected graphs with a given independence number  have been determined in \cite{ref15, MR, MR.,ref28}.

As a generalization of the above problems, Tu, Zhang and Shi \cite{ref13} initiated the study on the number of maximum dissociation sets in graphs. They determined the maximum number of maximum dissociation sets in trees of order $n$ as well as the trees attaining this maximum number.    The maximum numbers of maximum dissociation sets of trees, forests and subcubic trees with a given dissociation number are determined in \cite{ref11, TZD, ZTX}. The maximum numbers of maximal dissociation sets in general graphs, triangle-free graphs and forests are determined in \cite{ref6, ref12}. The minimum number of maximal dissociation sets in trees is determined in  \cite{ZQH}.

A {\it unicyclic graph} is a connected graph containing exactly one cycle. A {\it potted graph} is a unicyclic graph such that its cycle contains a unique vertex with degree larger than 2. It is clear that a potted graph is  obtained by identifying one vertex of a cycle $C$ and one vertex of a tree $T$, where $C$ and $T$ are vertex-disjoint.

Let $\mathcal{P}(n,k)$ be the set of potted graphs of order $n$ which contain a cycle of order $k$. Denote by $m(G)$ the number of maximum dissociation sets in a graph $G$.
Inspired by the previous results on the  number of  maximum  dissociation sets in special types of graphs,   we  study the following problem.

 {\bf Problem 1.} {\it Let $n, k$ be positive integers such that $n\ge k+1$.  Determine the  maximum value of $m(G)$ with $G\in \mathcal{P}(n,k)$ as well as the extremal potted graphs attaining this maximum value.}

 We solve Problem 1 in this paper. To state our main results, we need the following notations.

   We denote by $P_n$ and $C_n$ the path and the cycle of order $n$, respectively. Given $r$ vertex disjoint trees $T_1,\ldots, T_r$, we denote by $G(T_1,\ldots, T_r)$ a graph obtained from
   $T_1\cup T_2\cup\cdots\cup T_r$ by identifying $r$ leaves $v_1,\ldots,v_r$ with $v_i\in V(T_i)$ for $i=1,\ldots,r$. We also write $G(T_1,\ldots, T_r)$ as $G_w(T_1,\ldots, T_r)$ if we need to emphasize the identified vertex $w$. Suppose $C_k$ is vertex disjoint with $G_w(T_1,\ldots, T_r)$.
 We denote by  $C_k+_e G_w(T_1,\ldots, T_r)$  the graph obtained from  $C_k\cup G_w(T_1,\ldots, T_r)$ by adding an edge between $w$ and a vertex of $C_k$, while $C_k\odot G_w(T_1,\ldots, T_r)$ is the graph obtained from  $C_k\cup G_w(T_1,\ldots, T_r)$ by identifying $w$ and a vertex of $C_k$. Obviously, both $C_k+_e G_w(T_1,\ldots, T_r)$ and $C_k\odot G_w(T_1,\ldots, T_r)$ are potted graphs.

\noindent We define $\mathcal{T}_1,\ldots,\mathcal{T}_5 $, $\mathcal{G}_1,\ldots, \mathcal{G}_{9}$ to be the following families of graphs:
\begin{itemize}
  \item[]$\mathcal{T}_1=\{P_3, G_w(P_3,P_4,\ldots,P_4)\}$;
  \item[] $\mathcal{T}_2=\{G_w(P_2,P_3), G_w(P_2,P_3,T_1,\ldots, T_l): T_i= P_4\text{ or }K_{1,3}\text{ for }i=1,\ldots,l\}$;
 \item[]
   $\mathcal{T}_3=\{G_w(P_3,P_3), G_w(P_3,P_3,T_1,\ldots, T_{l}): T_i= P_4\text{ or }K_{1,3}\text{ for }i=1,\ldots,l\}$;
   \item[]
   $\mathcal{T}_4=\{G_w(P_2,P_2,P_3), G_w(P_2,P_2,P_3,T_1,\ldots, T_{l}): T_i=P_4\text{ or }K_{1,3}\text{ for }i=1,\ldots,l\}$;
   \item[]
   $\mathcal{T}_5 =\{G_w(P_2,P_2, P_2, P_2), G_w(P_2,P_2,P_2,P_2,T_1,\ldots, T_{l}): T_i= P_4\text{ or }K_{1,3}\text{ for }i=1,\ldots,l\}$;
   \item[]
   $\mathcal{G}_j=\{C_k+_e T: T\in\mathcal{T}_j\}$, $j=1,2,3,4,5$;
   \item[]
   $\mathcal{G}_6=\{ C_k\odot G_w(P_3,T_1,\ldots, T_{l}): T_i= P_4\text{ or }K_{1,3}\text{ for }i=1,\ldots,l\}$;
   \item[]
   $\mathcal{G}_7=\{C_k\odot G_w(P_2,P_2,T_1,\ldots, T_{l}): T_i=P_4\text{ or }K_{1,3}\text{ for }i=1,\ldots,l\}$;
   \item[]
    $\mathcal{G}_8=\{C_k\odot G_w(P_4,\ldots, P_4)\}$;
    \item[]
 $\mathcal{G}_9=\{ C_k\odot G_w(P_2,T_1,\ldots, T_{l}): T_i= P_4\text{ or }K_{1,3}\text{ for }i=1,\ldots,l\}$;
\end{itemize}
For $1\le i\le 5$,  $\mathcal{T}_i$ will be also written as $\mathcal{T}_i(w)$ if we need to emphasize the branch vertex $w$. It is clear that all graphs in $\mathcal{G}_{1},\ldots, \mathcal{G}_{9}$ are potted graphs. What follows are the diagrams of the above graphs.

\begin{figure}[H]
    \centering
     \subfigure{
    	\begin{tikzpicture}[scale=0.7]
    		\draw[fill=black](0.75,1.6) circle(0.08)
    		(1.5,1.6) circle(0.08)
    		(0,1.6) circle(0.08)
    		(0.4,2.4) circle(0.08);
    		\draw[fill=black](0.8,3.2) circle(0.08)
    		(-0.8,3.2) circle(0.08)
    		(1.2,4) circle(0.08)
    		(-1.2,4) circle(0.08)
    		(-0.4,2.4) circle(0.08);
    		\draw(1.5,1.6)--(0,1.6)
    		(0,1.6)--(1.2,4)
    		(0,1.6)--(-1.2,4);
    		\node at(0,4){$\cdots$};
    		\node at (0,1.35){$w$};
    		\node at (0,0.6){$\mathcal{T}_1(w)$};
    	\end{tikzpicture}
    }\hspace{0.45cm}
    \subfigure{
    	\begin{tikzpicture}[scale=0.7]
    		\draw[fill=black](0,0) circle(0.08)
    		(-0.5,0.8) circle(0.08)
    		(-0.5,1.6) circle(0.08)
    		(-0.5,2.4) circle(0.08)
    		(-1.5,0.8) circle(0.08)
    		(-1.5,1.6) circle(0.08)
    		(-1.5,2.4) circle(0.08)
    		(0.5,1.6) circle(0.08)
    		(1.5,1.6) circle(0.08)
    		(0.25,2.4) circle(0.08)
    		(0.75,2.4) circle(0.08)
    		(1.25,2.4) circle(0.08)
    		(1.75,2.4) circle(0.08)
    		(-0.75,0) circle(0.08)
    		(0.75,0) circle(0.08)
    		(1.5,0) circle(0.08);
    		\draw(0,0)--(1.5,0)
    		(0,0)--(-0.75,0)
    		(0,0)--(-1.5,0.8)
    		(0,0)--(-0.5,0.8)
    		(0,0)--(0.5,1.6)
    		(0,0)--(1.5,1.6)
    		(0.5,1.6)--(0.25,2.4)
    		(0.5,1.6)--(0.75,2.4)
    		(1.5,1.6)--(1.25,2.4)
    		(1.5,1.6)--(1.75,2.4)
    		(-1.5,0.8)--(-1.5,2.4)
    		(-0.5,2.4)--(-0.5,0.8);
    		\node at(-1,2.4){$\cdots$};
    		\node at(1,1.6){$\cdots$};
    		\node at (0,-0.25){$w$};
    		\node at (0,-1){$\mathcal{T}_2(w)$};
    	\end{tikzpicture}
    }\hspace{0.45cm}
    \subfigure{
    	\begin{tikzpicture}[scale=0.7]
    		\draw[fill=black](0,0) circle(0.08)
    		(-0.5,0.8) circle(0.08)
    		(-0.5,1.6) circle(0.08)
    		(-0.5,2.4) circle(0.08)
    		(-1.5,0.8) circle(0.08)
    		(-1.5,1.6) circle(0.08)
    		(-1.5,2.4) circle(0.08)
    		(0.5,1.6) circle(0.08)
    		(1.5,1.6) circle(0.08)
    		(0.25,2.4) circle(0.08)
    		(0.75,2.4) circle(0.08)
    		(1.25,2.4) circle(0.08)
    		(1.75,2.4) circle(0.08)
    		(-0.75,0) circle(0.08)
    		(0.75,0) circle(0.08)
    		(1.5,0) circle(0.08)
    		(-1.5,0) circle(0.08);
    		\draw(-1.5,0)--(1.5,0)
    		(0,0)--(-1.5,0.8)
    		(0,0)--(-0.5,0.8)
    		(0,0)--(0.5,1.6)
    		(0,0)--(1.5,1.6)
    		(0.5,1.6)--(0.25,2.4)
    		(0.5,1.6)--(0.75,2.4)
    		(1.5,1.6)--(1.25,2.4)
    		(1.5,1.6)--(1.75,2.4)
    		(-1.5,0.8)--(-1.5,2.4)
    		(-0.5,2.4)--(-0.5,0.8);
    		\node at(-1,2.4){$\cdots$};
    		\node at(1,1.6){$\cdots$};
    		\node at (0,-0.25){$w$};
    		\node at (0,-1){$\mathcal{T}_3(w)$};
    	\end{tikzpicture}
    }\hspace{0.45cm}
    \subfigure{
    	\begin{tikzpicture}[scale=0.7]
    		\draw[fill=black](0,0) circle(0.08)
    		(-0.5,0.8) circle(0.08)
    		(-0.5,1.6) circle(0.08)
    		(-0.5,2.4) circle(0.08)
    		(-1.5,0.8) circle(0.08)
    		(-1.5,1.6) circle(0.08)
    		(-1.5,2.4) circle(0.08)
    		(0.5,1.6) circle(0.08)
    		(1.5,1.6) circle(0.08)
    		(0.25,2.4) circle(0.08)
    		(0.75,2.4) circle(0.08)
    		(1.25,2.4) circle(0.08)
    		(1.75,2.4) circle(0.08)
    		(-0.55,-0.4) circle(0.08)
    		(0.75,0) circle(0.08)
    		(1.5,0) circle(0.08)
    		(-0.75,0) circle(0.08);
    		\draw(0,0)--(1.5,0)
    		(0,0)--(-0.55,-0.4)
    		(0,0)--(-1.5,0.8)
    		(0,0)--(-0.5,0.8)
    		(0,0)--(0.5,1.6)
    		(0,0)--(1.5,1.6)
    		(0.5,1.6)--(0.25,2.4)
    		(0.5,1.6)--(0.75,2.4)
    		(1.5,1.6)--(1.25,2.4)
    		(1.5,1.6)--(1.75,2.4)
    		(-1.5,0.8)--(-1.5,2.4)
    		(-0.5,2.4)--(-0.5,0.8)
    		(-0.75,0)--(0,0);
    		\node at(-1,2.4){$\cdots$};
    		\node at(1,1.6){$\cdots$};
    		\node at (0.1,-0.25){$w$};
    		\node at (0,-1){$\mathcal{T}_4(w)$};
    	\end{tikzpicture}
    }\hspace{0.45cm}
    \subfigure{
    	\begin{tikzpicture}[scale=0.7]
    		\draw[fill=black](0,0) circle(0.08)
    		(-0.5,0.8) circle(0.08)
    		(-0.5,1.6) circle(0.08)
    		(-0.5,2.4) circle(0.08)
    		(-1.5,0.8) circle(0.08)
    		(-1.5,1.6) circle(0.08)
    		(-1.5,2.4) circle(0.08)
    		(0.5,1.6) circle(0.08)
    		(1.5,1.6) circle(0.08)
    		(0.25,2.4) circle(0.08)
    		(0.75,2.4) circle(0.08)
    		(1.25,2.4) circle(0.08)
    		(1.75,2.4) circle(0.08)
    		(-0.55,-0.4) circle(0.08)
    		(0.55,-0.4) circle(0.08)
    		(0.75,0) circle(0.08)
    		(-0.75,0) circle(0.08);
    		\draw(0,0)--(0.55,-0.4)
    		(0,0)--(-0.55,-0.4)
    		(0,0)--(-1.5,0.8)
    		(0,0)--(-0.5,0.8)
    		(0,0)--(0.5,1.6)
    		(0,0)--(1.5,1.6)
    		(0.5,1.6)--(0.25,2.4)
    		(0.5,1.6)--(0.75,2.4)
    		(1.5,1.6)--(1.25,2.4)
    		(1.5,1.6)--(1.75,2.4)
    		(-1.5,0.8)--(-1.5,2.4)
    		(-0.5,2.4)--(-0.5,0.8)
    		(-0.75,0)--(0.75,0);
    		\node at(-1,2.4){$\cdots$};
    		\node at(1,1.6){$\cdots$};
    		\node at (0,-0.27){$w$};
    		\node at (0,-1){$\mathcal{T}_5(w)$};
    	\end{tikzpicture}
    }

    \subfigure{
    \begin{tikzpicture}[scale=0.7]
         \draw[fill=black](0,0.8) circle(0.08)
        [yshift=0.8cm](0,0) circle(0.08)
        [yshift=0.8cm](0,0) circle(0.08)
        [yshift=0.8cm,xshift=0.4cm](0,0) circle(0.08);
        \draw[fill=black](0.8,3.2) circle(0.08)
        (0.565685425,0.565685425) circle(0.08)
        (-0.565685425,0.565685425) circle(0.08)
        (-0.8,3.2) circle(0.08)
        (0.6,1.6) circle(0.08)
        (1.2,1.6) circle(0.08)
        (1.2,4) circle(0.08)
        (-1.2,4) circle(0.08)
        (-0.4,2.4) circle(0.08);
        \draw(0,0.8)--(0,1.6)
        (0,1.6)--(1.2,4)
        (0,1.6)--(1.2,1.6)
        (0,1.6)--(-1.2,4);
        \node at(0,4){$\cdots$};
        \draw[dashed](0,0) circle(0.8);
        \draw(0.565685425,0.565685425) arc(45:135:0.8);
        \node at (0,-1.3){$\mathcal{G}_1$};
    \end{tikzpicture}
    }\hspace{0.4cm}
    \subfigure{
    \begin{tikzpicture}[scale=0.7]
         \draw[fill=black](0,0) circle(0.08)
        (-0.5,0.8) circle(0.08)
        (-0.5,1.6) circle(0.08)
        (-0.5,2.4) circle(0.08)
        (-1.5,0.8) circle(0.08)
        (-1.5,1.6) circle(0.08)
        (-1.5,2.4) circle(0.08)
        (0.5,1.6) circle(0.08)
        (1.5,1.6) circle(0.08)
        (0.25,2.4) circle(0.08)
        (0.75,2.4) circle(0.08)
        (1.25,2.4) circle(0.08)
        (1.75,2.4) circle(0.08)
        (-0.8,0) circle(0.08)
        (0.8,0) circle(0.08)
        (1.6,0) circle(0.08);
        \draw(-0.8,0)--(1.6,0)
        (0,0)--(-1.5,0.8)
        (0,0)--(-0.5,0.8)
        (0,0)--(0.5,1.6)
        (0,0)--(1.5,1.6)
        (0.5,1.6)--(0.25,2.4)
        (0.5,1.6)--(0.75,2.4)
        (1.5,1.6)--(1.25,2.4)
        (1.5,1.6)--(1.75,2.4)
        (-1.5,0.8)--(-1.5,2.4)
        (-0.5,2.4)--(-0.5,0.8)
        (0,-0.8)--(0,0);
        \node at(-1,2.4){$\cdots$};
        \node at(1,1.6){$\cdots$};
        \draw[fill=black](0,-0.8) circle(0.08)
        (0.565685425,-1.034314575) circle(0.08)
        (-0.565685425,-1.034314575) circle(0.08);
        \draw[dashed](0,-1.6) circle(0.8);
        \draw(0.565685425,-1.034314575) arc(45:135:0.8);
        \node at (0,-2.9){$\mathcal{G}_2$};
    \end{tikzpicture}
    }\hspace{0.4cm}
    \subfigure{
    \begin{tikzpicture}[scale=0.7]
    \draw[fill=black](0,0) circle(0.08)
        (-0.5,0.8) circle(0.08)
        (-0.5,1.6) circle(0.08)
        (-0.5,2.4) circle(0.08)
        (-1.5,0.8) circle(0.08)
        (-1.5,1.6) circle(0.08)
        (-1.5,2.4) circle(0.08)
        (0.5,1.6) circle(0.08)
        (1.5,1.6) circle(0.08)
        (0.25,2.4) circle(0.08)
        (0.75,2.4) circle(0.08)
        (1.25,2.4) circle(0.08)
        (1.75,2.4) circle(0.08)
        (-0.8,0) circle(0.08)
        (0.8,0) circle(0.08)
        (1.6,0) circle(0.08)
        (-1.6,0) circle(0.08);
        \draw(-1.6,0)--(1.6,0)
        (0,0)--(-1.5,0.8)
        (0,0)--(-0.5,0.8)
        (0,0)--(0.5,1.6)
        (0,0)--(1.5,1.6)
        (0.5,1.6)--(0.25,2.4)
        (0.5,1.6)--(0.75,2.4)
        (1.5,1.6)--(1.25,2.4)
        (1.5,1.6)--(1.75,2.4)
        (-1.5,0.8)--(-1.5,2.4)
        (-0.5,2.4)--(-0.5,0.8)
        (0,-0.8)--(0,0);
        \node at(-1,2.4){$\cdots$};
        \node at(1,1.6){$\cdots$};
        \draw[fill=black](0,-0.8) circle(0.08)
        (0.565685425,-1.034314575) circle(0.08)
        (-0.565685425,-1.034314575) circle(0.08);
        \draw[dashed](0,-1.6) circle(0.8);
        \draw(0.565685425,-1.034314575) arc(45:135:0.8);
        \node at (0,-2.9){$\mathcal{G}_3$};
    \end{tikzpicture}
    }\hspace{0.4cm}
    \subfigure{
    \begin{tikzpicture}[scale=0.7]
         \draw[fill=black](0,0) circle(0.08)
        (-0.5,0.8) circle(0.08)
        (-0.75,-0.4) circle(0.08)
        (-0.5,1.6) circle(0.08)
        (-0.5,2.4) circle(0.08)
        (-1.5,0.8) circle(0.08)
        (-1.5,1.6) circle(0.08)
        (-1.5,2.4) circle(0.08)
        (0.5,1.6) circle(0.08)
        (1.5,1.6) circle(0.08)
        (0.25,2.4) circle(0.08)
        (0.75,2.4) circle(0.08)
        (1.25,2.4) circle(0.08)
        (1.75,2.4) circle(0.08)
        (-0.75,0) circle(0.08)
        (0.8,0) circle(0.08)
        (1.6,0) circle(0.08);
        \draw(-0.75,0)--(1.6,0)
        (0,0)--(-1.5,0.8)
        (0,0)--(-0.5,0.8)
        (0,0)--(0.5,1.6)
        (0,0)--(1.5,1.6)
        (0.5,1.6)--(0.25,2.4)
        (0.5,1.6)--(0.75,2.4)
        (1.5,1.6)--(1.25,2.4)
        (1.5,1.6)--(1.75,2.4)
        (-1.5,0.8)--(-1.5,2.4)
        (-0.5,2.4)--(-0.5,0.8)
        (0,-0.8)--(0,0)
        (-0.75,-0.4)--(0,0);
        \node at(-1,2.4){$\cdots$};
        \node at(1,1.6){$\cdots$};
        \draw[fill=black](0,-0.8) circle(0.08)
        (0.565685425,-1.034314575) circle(0.08)
        (-0.565685425,-1.034314575) circle(0.08);
        \draw[dashed](0,-1.6) circle(0.8);
        \draw(0.565685425,-1.034314575) arc(45:135:0.8);
        \node at (0,-2.9){$\mathcal{G}_4$};
    \end{tikzpicture}
    }\hspace{0.4cm}
    \subfigure{
    \begin{tikzpicture}[scale=0.7]
         \draw[fill=black](0,0) circle(0.08)
        (-0.5,0.8) circle(0.08)
        (-0.75,-0.4) circle(0.08)
        (-0.5,1.6) circle(0.08)
        (-0.5,2.4) circle(0.08)
        (-1.5,0.8) circle(0.08)
        (-1.5,1.6) circle(0.08)
        (-1.5,2.4) circle(0.08)
        (0.5,1.6) circle(0.08)
        (1.5,1.6) circle(0.08)
        (0.25,2.4) circle(0.08)
        (0.75,2.4) circle(0.08)
        (1.25,2.4) circle(0.08)
        (1.75,2.4) circle(0.08)
        (-0.75,0) circle(0.08)
        (0.75,-0.4) circle(0.08)
        (0.75,0) circle(0.08);
        \draw(-0.75,0)--(0.75,0)
        (0,0)--(-1.5,0.8)
        (0,0)--(-0.5,0.8)
        (0,0)--(0.5,1.6)
        (0,0)--(1.5,1.6)
        (0.5,1.6)--(0.25,2.4)
        (0.5,1.6)--(0.75,2.4)
        (1.5,1.6)--(1.25,2.4)
        (1.5,1.6)--(1.75,2.4)
        (-1.5,0.8)--(-1.5,2.4)
        (-0.5,2.4)--(-0.5,0.8)
        (0,-0.8)--(0,0)
        (-0.75,-0.4)--(0,0)
        (0.75,-0.4)--(0,0);
        \node at(-1,2.4){$\cdots$};
        \node at(1,1.6){$\cdots$};
        \draw[fill=black](0,-0.8) circle(0.08)
        (0.565685425,-1.034314575) circle(0.08)
        (-0.565685425,-1.034314575) circle(0.08);
        \draw[dashed](0,-1.6) circle(0.8);
        \draw(0.565685425,-1.034314575) arc(45:135:0.8);
        \node at (0,-2.9){$\mathcal{G}_5$};
    \end{tikzpicture}
    }

    \subfigure{
    \begin{tikzpicture}[scale=0.7]
         \draw[fill=black](0,0) circle(0.08)
        (-0.5,0.8) circle(0.08)
        (-0.5,1.6) circle(0.08)
        (-0.5,2.4) circle(0.08)
        (-1.5,0.8) circle(0.08)
        (-1.5,1.6) circle(0.08)
        (-1.5,2.4) circle(0.08)
        (0.5,1.6) circle(0.08)
        (1.5,1.6) circle(0.08)
        (0.25,2.4) circle(0.08)
        (0.75,2.4) circle(0.08)
        (1.25,2.4) circle(0.08)
        (1.75,2.4) circle(0.08)
        (0.75,0.2) circle(0.08)
        (1.5,0.4) circle(0.08);
        \draw(0,0)--(1.5,0.4)
        (0,0)--(-1.5,0.8)
        (0,0)--(-0.5,0.8)
        (0,0)--(0.5,1.6)
        (0,0)--(1.5,1.6)
        (0.5,1.6)--(0.25,2.4)
        (0.5,1.6)--(0.75,2.4)
        (1.5,1.6)--(1.25,2.4)
        (1.5,1.6)--(1.75,2.4)
        (-1.5,0.8)--(-1.5,2.4)
        (-0.5,2.4)--(-0.5,0.8);
        \node at(-1,2.4){$\cdots$};
        \node at(1,1.6){$\cdots$};
        \draw[dashed](0,-0.8) circle(0.8);
        \draw[fill=black](0.565685425,-0.234314575) circle(0.08)
        (-0.565685425,-0.234314575) circle(0.08);
        \draw(0.565685425,-0.234314575) arc(45:135:0.8);
        \node at (0,-2.1){$\mathcal{G}_6$};
    \end{tikzpicture}
    }\hspace{0.75cm}
     \subfigure{
    \begin{tikzpicture}[scale=0.7]
         \draw[fill=black](0,0) circle(0.08)
        (-0.5,0.8) circle(0.08)
        (-0.5,1.6) circle(0.08)
        (-0.5,2.4) circle(0.08)
        (-1.5,0.8) circle(0.08)
        (-1.5,1.6) circle(0.08)
        (-1.5,2.4) circle(0.08)
        (0.5,1.6) circle(0.08)
        (1.5,1.6) circle(0.08)
        (0.25,2.4) circle(0.08)
        (0.75,2.4) circle(0.08)
        (1.25,2.4) circle(0.08)
        (1.75,2.4) circle(0.08)
        (1.1,0.2) circle(0.08)
        (-1.1,0.2) circle(0.08);
        \draw(0,0)--(1.1,0.2)
        (0,0)--(-1.1,0.2)
        (0,0)--(-1.5,0.8)
        (0,0)--(-0.5,0.8)
        (0,0)--(0.5,1.6)
        (0,0)--(1.5,1.6)
        (0.5,1.6)--(0.25,2.4)
        (0.5,1.6)--(0.75,2.4)
        (1.5,1.6)--(1.25,2.4)
        (1.5,1.6)--(1.75,2.4)
        (-1.5,0.8)--(-1.5,2.4)
        (-0.5,2.4)--(-0.5,0.8);
        \node at(-1,2.4){$\cdots$};
        \node at(1,1.6){$\cdots$};
        \draw[dashed](0,-0.8) circle(0.8);
        \draw[fill=black](0.565685425,-0.234314575) circle(0.08)
        (-0.565685425,-0.234314575) circle(0.08);
        \draw(0.565685425,-0.234314575) arc(45:135:0.8);
        \node at (0,-2.1){$\mathcal{G}_7$};
    \end{tikzpicture}
    }\hspace{0.75cm}
    \subfigure{
    \begin{tikzpicture}[scale=0.7]
         \draw[fill=black](0,0.8) circle(0.08);
        \draw[fill=black](0.8,2.4) circle(0.08)
        (-0.8,2.4) circle(0.08)
        (0.4,1.6) circle(0.08)
        (-0.4,1.6) circle(0.08)
        (1.2,3.2) circle(0.08)
        (-1.2,3.2) circle(0.08)
        (0.565685425,0.565685425) circle(0.08)
        (-0.565685425,0.565685425) circle(0.08);
        \draw(0,0.8)--(1.2,3.2)
        (0,0.8)--(-1.2,3.2);
        \node at(0,3.2){$\cdots$};
        \draw[dashed](0,0) circle(0.8);
        \draw(0.565685425,0.565685425) arc(45:135:0.8);
        \node at (0,-1.3){$\mathcal{G}_8$};
    \end{tikzpicture}
    }\hspace{0.75cm}
    \subfigure{
    \begin{tikzpicture}[scale=0.7]
         \draw[fill=black](0,0) circle(0.08)
        (-0.5,0.8) circle(0.08)
        (-0.5,1.6) circle(0.08)
        (-0.5,2.4) circle(0.08)
        (-1.5,0.8) circle(0.08)
        (-1.5,1.6) circle(0.08)
        (-1.5,2.4) circle(0.08)
        (0.5,1.6) circle(0.08)
        (1.5,1.6) circle(0.08)
        (0.25,2.4) circle(0.08)
        (0.75,2.4) circle(0.08)
        (1.25,2.4) circle(0.08)
        (1.75,2.4) circle(0.08)
        (1.1,0.2) circle(0.08);
        \draw(0,0)--(1.1,0.2)
        (0,0)--(-1.5,0.8)
        (0,0)--(-0.5,0.8)
        (0,0)--(0.5,1.6)
        (0,0)--(1.5,1.6)
        (0.5,1.6)--(0.25,2.4)
        (0.5,1.6)--(0.75,2.4)
        (1.5,1.6)--(1.25,2.4)
        (1.5,1.6)--(1.75,2.4)
        (-1.5,0.8)--(-1.5,2.4)
        (-0.5,2.4)--(-0.5,0.8);
        \node at(-1,2.4){$\cdots$};
        \node at(1,1.6){$\cdots$};
        \draw[dashed](0,-0.8) circle(0.8);
        \draw[fill=black](0.565685425,-0.234314575) circle(0.08)
        (-0.565685425,-0.234314575) circle(0.08);
        \draw(0.565685425,-0.234314575) arc(45:135:0.8);
        \node at (0,-2.1){$\mathcal{G}_9$};
    \end{tikzpicture}
    }
    \caption{The families $\mathcal{T}_1(w),\ldots,\mathcal{T}_5(w)$ and  $\mathcal{G}_1,\ldots,\mathcal{G}_{9}$}
    \label{fig:3}
\end{figure}
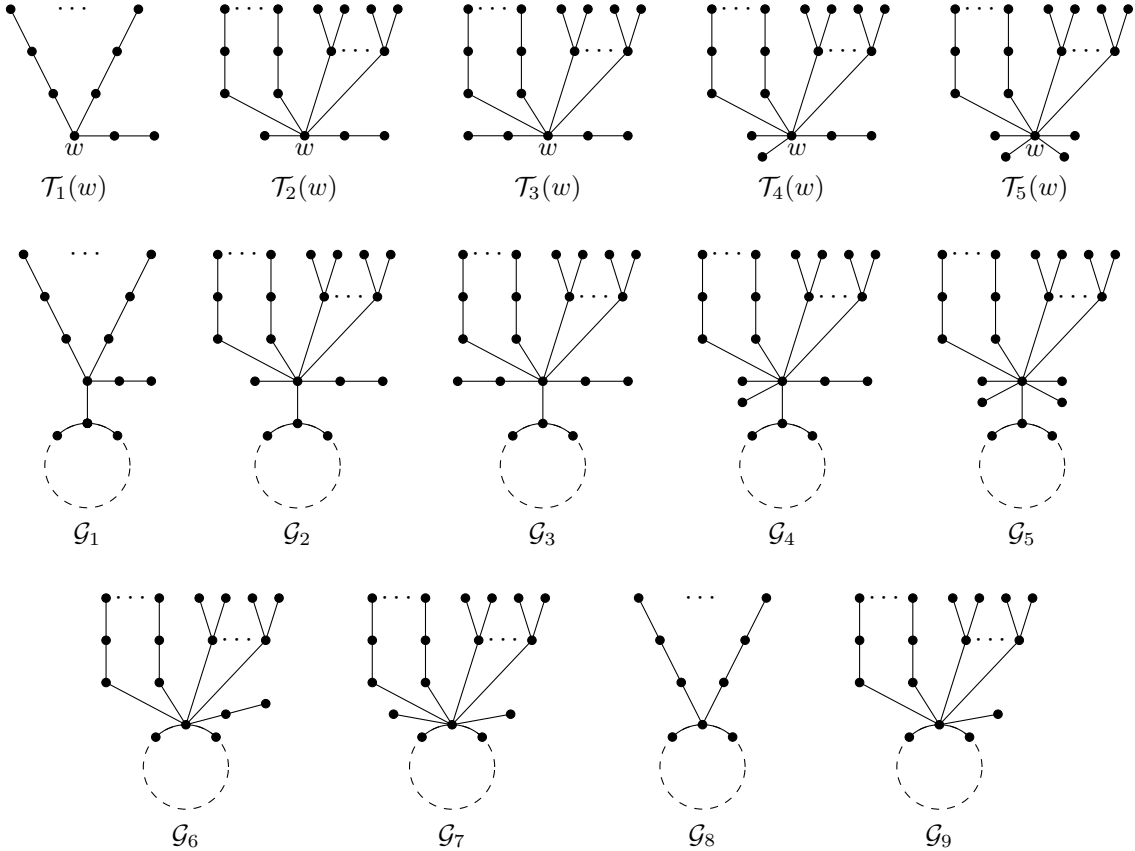

 Now we are ready to present our main results.
\begin{theorem}\label{th1.1}
    Let $n$, $k$ be positive integers such that $n\ge k+4$ and $k\equiv0~(\bmod~3)$. If $G\in\mathcal{P}(n,k)$, then
\begin{equation*}
m(G)\leq
\begin{cases}
\text{$3^{(n-k)/3}+\frac{1}{3}(n-k) +1$,} &\quad\text{if $n\equiv0$ $(\bmod\text{ }3);$}\\
\text{$3^{(n-k-1)/3}+1$,} &\quad\text{if $n\equiv1$ $(\bmod\text{ }3);$}\\
\text{$3^{(n-k-2)/3}$,} &\quad\text{if $n\equiv2$ $(\bmod\text{ }3);$}\\
\end{cases}
\end{equation*}
    with equality if and only if
\begin{equation*}
G\in
\begin{cases}
\text{$\mathcal{G}_1$,} &\quad\text{if $n\equiv0$ $(\bmod\text{ }3);$}\\
\text{$\mathcal{G}_2$,} &\quad\text{if $n\equiv1$ $(\bmod\text{ }3);$}\\
\text{$\bigcup_{i=3}^{7}\mathcal{G}_i$,} &\quad\text{if $n\equiv2$ $(\bmod\text{ }3)$.}\\
\end{cases}
\end{equation*}
\end{theorem}

For convenience, let
\begin{equation*}
	\begin{cases}
		\text{$x_1=k\cdot3^{(n-k-1)/3-1}+\frac{1}{3}(k+1)$},\\
		\text{$x_2=(k+1)\cdot3^{(n-k-1)/3-1}+1$},\\
		\text{$x_3=k\cdot3^{(n-k)/3-1}+\frac{1}{3}(k+1)\left[\frac{1}{3}(n-k)+1\right]+1$},\\
		\text{$x_4=(k+1)\cdot3^{(n-k)/3-1}+\frac{1}{3}(n+k-1)$},\\
		\text{$y_1=\frac{1}{18}(k+2)(k+5)\cdot3^{(n-k)/3}+\frac{1}{9}(n+5)(k-1)$},\\
		\text{$y_2=\frac{1}{18}k(k+5)\cdot3^{(n-k)/3}+\frac{1}{3}(k-1)+\frac{1}{18}(k+2)(k+5)\left[\frac{1}{3}(n-k)+1\right]$},\\
		\text{$y_3=\frac{1}{18}k(k+5)\cdot3^{(n-k-1)/3}+\frac{1}{18}(k+2)(k+5)$},\\
		\text{$y_4=\frac{1}{18}(k+2)(k+5)\cdot3^{(n-k-1)/3}+\frac{1}{3}(k-1)$}.\\
	\end{cases}
\end{equation*}

\begin{theorem}\label{th1.2}
	Let $n$, $k$ be positive integers such that $n\ge k+3$ and $k\equiv2~(\bmod~3)$. If $G\in\mathcal{P}(n,k)$, then
	\begin{equation}\label{eqhhhh1}
		m(G)\leq
		\begin{cases}
			\text{$\max\{x_1,x_2\}$,}\quad &\text{if $n\equiv0$ $(\bmod\text{ }3);$}\\
			\text{$(k+1)\cdot3^{(n-k-2)/3-1}$,}\quad &\text{if $n\equiv1$ $(\bmod\text{ }3);$}\\
			\text{$\max\{x_3,x_4\}$,}\quad &\text{if $n\equiv2$ $(\bmod\text{ }3)$.}\\
		\end{cases}
	\end{equation}
	Moreover,  equality in (\ref{eqhhhh1}) holds if and only if one of the following statements holds:
	\item[(i)] $n\equiv0~(\bmod~3),~G\in\mathcal{H}_i$ and $x_i=\max\{x_1,x_2\}$ with $i\in \{1,2\}$, where $\mathcal{H}_1=\mathcal{G}_2\text{ and }\mathcal{H}_2=\mathcal{G}_9;$
	\item[(ii)]  $n\equiv1~(\bmod~3)$, $G\in\mathcal{G}_6\cup \mathcal{G}_7;$
	\item[(iii)]   $n\equiv2~(\bmod~3),~G\in\mathcal{H}_i$ and $x_i=\max\{x_3,x_4\}$ with $i\in\{3,4\}$, where $\mathcal{H}_3=\mathcal{G}_1,~\mathcal{H}_4=\mathcal{G}_8.$

\end{theorem}
\begin{theorem}\label{th1.3}
Let $n$, $k$ be positive integers such that $n\ge k+3$ and $k\equiv1~(\bmod~3)$. If $G\in\mathcal{P}(n,k)$, then
	\begin{equation}\label{eqhhhh2}
		m(G)\leq
		\begin{cases}
			\text{$\frac{1}{18}(k+2)(k+5)\cdot3^{(n-k-2)/3},$} \quad&\text{if $n\equiv0$ $(\bmod\text{ }3);$}\\
			\text{$\max\{y_1,y_2\},$} \quad&\text{if $n\equiv1$ $(\bmod\text{ }3);$}\\
			\text{$\max\{y_3,y_4\},$} \quad&\text{if $n\equiv2$ $(\bmod\text{ }3).$}\\
		\end{cases}
	\end{equation}
Moreover,  equality in (\ref{eqhhhh2}) holds if and only if one of the following statements holds:
	\item[(i)]   $n\equiv0~(\bmod~3)$, $G\in\mathcal{G}_6\cup\mathcal{G}_7;$
	\item[(ii)]   $n\equiv1~(\bmod~3)$, $G\in\mathcal{M}_i$ and $y_i=\max\{y_1,y_2\}$ with $i\in\{1,2\}$, where $\mathcal{M}_1=\mathcal{G}_8,~\mathcal{M}_2=\mathcal{G}_1;$
	\item[(iii)]   $n\equiv2~(\bmod~3)$, $G\in\mathcal{M}_i$ and $y_i=\max\{y_3,y_4\}$ with $i\in\{3,4\}$, where $\mathcal{M}_3=\mathcal{G}_2,~\mathcal{M}_4=\mathcal{G}_9.$
\end{theorem}

We remark that when $k$ is fixed and $n$ is sufficiently large, $m(G)$ has a unique maximum value $x_2$, $x_4$, $y_1$ or $y_4$ in each corresponding case of Theorem \ref{th1.2} and \ref{th1.3}.

We will prepare some preliminaries in Section 2 and present the proofs of the above theorems in Section 3.
  \section{Preliminaries}

  Let $G$ be a graph.   Given a subset $S \subseteq V(G)$, we denoted by $G[S]$ the subgraph of $G$ induced by $S$. We also write $G[V(G)\setminus S]$ as $G-S$, which is abbreviated as $G-v$ when $S=\{v\}$ is a singleton. Let $H$ be a subgraph of $G$. We will also write $G-V(H)$ as $G-H$.
  The neighborhood of a vertex $v$ in $G$ is denoted by $N_G(v)$. The closed neighborhood of $v$  is defined to be $N_G[v]=N_G(v)\cup \left\{v\right\}$. We denote by $d_G(v)$ the degree of $v$ in $G$, i.e., $d_G(v)=|N_G(v)|$. We abbreviate $N_G(v), N_G[v]$ and $ d_G(v)$ as  $N(v), N[v]$ and $d(v)$ respectively if no confuse arises. For a  positive integer $k$, $kG$ is the disjoint union of $k$ copies of a graph $G$.

 Denote by $MD(G)$ the set of all maximum dissociation sets of a graph $G$. Then $m(G)=|MD(G)|$. Given a vertex $v\in V(G)$,   $MD(G)$ is the disjoint union of the following three sets:
\begin{eqnarray*}
MD(G,v^0)&=&\{F: F\in MD(G), v\in F \text{ and } d_{G[F]}(v)=0\},\\
 MD(G,v^1)&=&\{F: F\in MD(G), v\in F \text{ and } d_{G[F]}(v)=1\},\\
 MD(G,v^-)&=&\{F: F\in MD(G) \text{ and } v\notin F\},
 \end{eqnarray*}
 whose cardinalities will be denoted by $m(G,v^0)$, $m(G,v^1)$, $m(G,v^-)$, respectively. Notice that $S$ is a dissociation set of $G$ if and only if $V(G)\setminus S$ is a 3-path vertex cover of a graph $G$.
By Proposition 1.1 of \cite{ref4}, we have the following lemma.
\begin{lemma}\label{lem2.1}
    \cite{ref4} Let $n$ be a positive integer. Then ${\rm diss}(P_n)=\lceil {2n}/{3}\rceil, {\rm diss}(C_n)=\lfloor {2n}/{3}\rfloor$.
\end{lemma}
Firstly we list some known results on the number of maximum dissociation sets of  trees and forests.

\begin{lemma}\label{lem2.3}
    \cite{ref13} Let $T$ be a tree of order $n\geq3$. Then
\begin{equation*}
 m(T)\leq
\begin{cases}
\text{$3^{n/3-1}+\frac{n}{3}+1$,} &\quad\text{if $n\equiv0$ $(\bmod\text{ }3)$;}\\
\text{$3^{(n-1)/3-1}+1$,} &\quad\text{if $n\equiv1$ $(\bmod\text{ }3)$;}\\
\text{$3^{(n-2)/3-1}$,} &\quad\text{if $n\equiv2$ $(\bmod\text{ }3)$};\\
\end{cases}
\end{equation*}
with equality if and only if
\begin{equation*}
 T\in
\begin{cases}
\text{$\mathcal{T}_1$,} &\quad\text{if $n\equiv0$ $(\bmod\text{ }3)$;}\\
\text{$\mathcal{T}_2$,} &\quad\text{if $n\equiv1$ $(\bmod\text{ }3)$;}\\
\text{$\mathcal{T}_3\cup\mathcal{T}_4\cup\mathcal{T}_5$,} &\quad\text{if $n\equiv2$ $(\bmod\text{ }3)$.}\\
\end{cases}
\end{equation*}
\end{lemma}

\begin{lemma}\label{lemma2.4}
    \cite{ref11} Let $G$ be a forest of order $n\geq6$ with at least two components such that each component   has order at least 3. Then
    \begin{equation*}
 m(G)\leq
\begin{cases}
\text{$3^{n/3}$,} &\quad\text{if $n\equiv0$ $(\bmod\text{ }3)$;}\\
\text{$2\times3^{(n-1)/3-1}$,} &\quad\text{if $n\equiv1$ $(\bmod\text{ }3)$;}\\
\text{$4\times3^{(n-2)/3-2}$,} &\quad\text{if $n\equiv2$ $(\bmod\text{ }3)$}\\
\end{cases}
\end{equation*}
with equality if and only if
\begin{equation*}
 G\cong
\begin{cases}
\text{$\frac{n}{3}P_3$,} &\quad\text{if $n\equiv0$ $(\bmod\text{ }3)$;}\\
\text{$P_4\cup(\frac{n-1}{3}-1)P_3$,} &\quad\text{if $n\equiv1$ $(\bmod\text{ }3)$;}\\
\text{$2P_4\cup(\frac{n-2}{3}-2)P_3$,} &\quad\text{if $n\equiv2$ $(\bmod\text{ }3)$.}\\
\end{cases}
\end{equation*}
\end{lemma}

Now we present some lemmas on the  number of maximum dissociation sets for some special graphs.
Denote by $$\mathcal{H}=\{G_w(P_3,T_1,\ldots,T_l), G_w(P_2,P_2,T_1,\ldots,T_l): T_i= P_4\text{ or }K_{1,3}\text{ for }i=1,\ldots,l\}.$$

$\mathcal{H}$ will be also written as $\mathcal{H}(w)$ if we need to emphasize the branch vertex $w$.
Each graph from $\mathcal{H}$ has one of the diagrams in Figure \ref{fig:5}.

\begin{figure}[H]
	\centering
	\subfigure[$G_w(P_3,T_1,\ldots,T_l)$]{
		\begin{tikzpicture}[scale=0.8]
			\draw[fill=black](0,0) circle(0.08)
		(-0.5,0.8) circle(0.08)
		(-0.5,1.6) circle(0.08)
		(-0.5,2.4) circle(0.08)
		(-1.8,0.8) circle(0.08)
		(-1.8,1.6) circle(0.08)
		(-1.8,2.4) circle(0.08)
		(0.5,1.6) circle(0.08)
		(1.8,1.6) circle(0.08)
		(0.25,2.4) circle(0.08)
		(0.75,2.4) circle(0.08)
		(1.55,2.4) circle(0.08)
		(2.05,2.4) circle(0.08)
		(0.7,0) circle(0.08)
		(1.4,0) circle(0.08);
		\draw(0,0)--(1.4,0)
		(0,0)--(-1.8,0.8)
		(0,0)--(-0.5,0.8)
		(0,0)--(0.5,1.6)
		(0,0)--(1.8,1.6)
		(0.5,1.6)--(0.25,2.4)
		(0.5,1.6)--(0.75,2.4)
		(1.8,1.6)--(1.55,2.4)
		(1.8,1.6)--(2.05,2.4)
		(-1.8,0.8)--(-1.8,2.4)
		(-0.5,2.4)--(-0.5,0.8);
		\node at(-1.15,2.4){$\cdots$};
		\node at(1.15,1.6){$\cdots$};
		\node at(0,-0.3){$w$};
		\end{tikzpicture}
	}\hspace{3cm}
	\subfigure[$G_w(P_2,P_2,T_1,\ldots,T_l)$]{
		\begin{tikzpicture}[scale=0.8]
			\draw[fill=black](0,0) circle(0.08)
			(-0.5,0.8) circle(0.08)
			(-0.5,1.6) circle(0.08)
			(-0.5,2.4) circle(0.08)
			(-1.8,0.8) circle(0.08)
			(-1.8,1.6) circle(0.08)
			(-1.8,2.4) circle(0.08)
			(0.5,1.6) circle(0.08)
			(1.8,1.6) circle(0.08)
			(0.25,2.4) circle(0.08)
			(0.75,2.4) circle(0.08)
			(1.55,2.4) circle(0.08)
			(2.05,2.4) circle(0.08)
			(-0.7,-0) circle(0.08)
			(0.7,0) circle(0.08);
			\draw(0,0)--(0.7,0)
			(0,0)--(-0.7,0)
			(0,0)--(-1.8,0.8)
			(0,0)--(-0.5,0.8)
			(0,0)--(0.5,1.6)
			(0,0)--(1.8,1.6)
			(0.5,1.6)--(0.25,2.4)
			(0.5,1.6)--(0.75,2.4)
			(1.8,1.6)--(1.55,2.4)
			(1.8,1.6)--(2.05,2.4)
			(-1.8,0.8)--(-1.8,2.4)
			(-0.5,2.4)--(-0.5,0.8);
			\node at(-1.15,2.4){$\cdots$};
			\node at(1.15,1.6){$\cdots$};
			\node at(0,-0.3){$w$};
		\end{tikzpicture}
	}
	\caption{Trees in the family $\mathcal{H}$}
	\label{fig:5}
\end{figure}
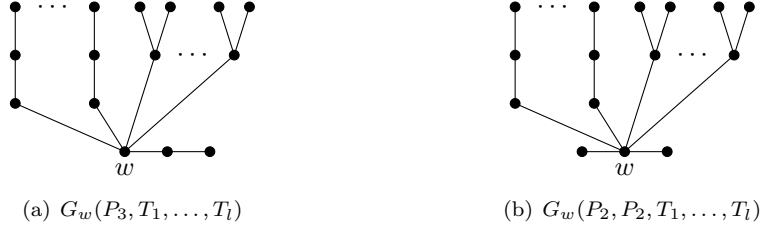

Applying the same proof as Theorem 4.5 in \cite{ref13}, we can get the following corollary.

\begin{corollary}\label{cor3.5}
	Suppose $T\notin\mathcal{H}$ is a tree  of order $n$ with $n\equiv0$ $(\bmod\text{ }3)$.   Then $$m(T)\leq3^{n/3-1}+1.$$
\end{corollary}
Let $T$ be a tree with $V_1\subseteq V(T)$. We denote by
$$\overline{m}(T-V_1)=\left\{\begin{array}{ll}
	m(T-V_1),&\text{if}~~ {\rm diss}(T-V_1)={\rm diss}(T);\\
	0,&\text{otherwise}.
\end{array}
\right.$$
In particular, if $T-V_1$ is an empty graph, we have $\overline{m}(T-V_1)=0$. We abbreviate $\overline{m}(T-V_1)$ as $\overline{m}(T-v)$ if $V_1=\{v\}$ is a singleton.
\begin{lemma}\label{CL3.3}
	Let $T$ be a tree of order $n\ge4$ and $v\in V(T)$.
	\item[(i)] If $T\in\mathcal{H}(w)$, then $\overline{m}(T-v)\le 3^{n/3-1}$, with equality if and only if $v=w$.
	\item[(ii)]If $T\in\mathcal{T}_2(w)$, then $\overline{m}(T-v)\le 3^{(n-1)/3-1}$, with equality if and only if $v=w$.
	\item[(iii)]If $T\in\mathcal{T}_3(w)\cup\mathcal{T}_4(w)\cup\mathcal{T}_5(w)$, then $\overline{m}(T-v)\le 3^{(n-2)/3-1}$, with equality if and only if $v=w$.
\end{lemma}
\begin{proof}
	A straightforward computation may verify the statements; see Tables \ref{table:1}-\ref{EQUg} in Appendix B for the details.
\end{proof}

\begin{lemma}\label{lem2.4}
    Let $G$ be a forest of order $n$. Then $m(G)\leq3^{\lfloor{n}/{3}\rfloor}$  with equality if and only if
\begin{equation*}\label{2.1}
G\cong
\begin{cases}
\text{$\frac{n}{3}P_3$,} &\quad\text{if $n\equiv0$ $(\bmod\text{ }3)$;}\\
\text{$\frac{n-1}{3}P_3\cup K_1$,} &\quad\text{if $n\equiv1$ $(\bmod\text{ }3)$;}\\
\text{$\frac{n-2}{3}P_3\cup 2K_1$ or $\frac{n-2}{3}P_3\cup P_2$,} &\quad\text{if $n\equiv2$ $(\bmod\text{ }3)$.}\\
\end{cases}
\end{equation*}
\end{lemma}

\begin{proof}
   It is straightforward for $1\leq n\leq5$. We consider the case $n\geq6$. Suppose $G=\bigcup_{i=1}^rT_i$ with $T_1,\ldots, T_r$ being the components of $G$. Moreover, we assume $|V_{i}|\le 2$ for $i=1,\ldots,k$ and $|V_{i}|\ge 3$ for $i=k+1,\ldots,r$.

   Note that $V(T_i)$ is contained in all the maximum dissociation sets if and only if $|V(T_i)|\le 2$ for $i=1,\ldots,r$. If $G$ has at most one component with order larger than 2,  then by Lemma \ref{lem2.3}, we have $m(G)<3^{\lfloor n/3\rfloor}$.

  Now  suppose $G$  has at least two components with order larger than 2, i.e., $r-k\ge 2$.
     Let $s=\sum_{i=1}^k|V(T_i)|$. Applying Lemma \ref{lemma2.4} on $G-T_1-\cdots-T_k$, if $n-s\not\equiv 0$ (mod 3) we have $m(G)< 3^{\lfloor(n-s)/3\rfloor}$; otherwise we have $m(G)\le 3^{(n-s)/3}$, with equality if and only if $G-T_1-\cdots-T_k\cong \frac{n-s}{3}P_3$.
    Therefore, we have the following.
     \begin{itemize}
   \item[(i)] If $n\equiv 0$ (mod 3), then
   $m(G)\le 3^{n/3}$ with equality if and only if $G\cong \frac{n}{3}P_3$;
   \item[(ii)] If   $n\equiv 1$ (mod 3), then $m(G)\le 3^{(n-1)/3}$, with equality if and only if $G\cong \frac{n-1}{3}P_3\cup K_1$;
        \item[(iii)] If   $n\equiv 2$ (mod 3), then $m(G)\le 3^{(n-2)/3}$, with equality if and only if $G\cong \frac{n-2}{3}P_3\cup 2K_1$ or $\frac{n-2}{3}P_3\cup P_2$.
   \end{itemize}
   \end{proof}

By Lemma \ref{lem2.1}, we have ${\rm diss}(C_n)={\rm diss}(P_{n-1})$. Tu, Li and Du \cite{ref12}   prove  that $$m(P_n)<0.81\times6^{n/4} \quad\text{and}\quad m(C_n)\leq6^{n/4},$$ with equality if and only if $n=4$. We present the precise values of $m(P_n)$ and $m(C_n)$ as follows.

\begin{lemma}\label{lem2.2}
    Let $n$ be a positive integer. Then
 \begin{equation*}
m(P_n)=
\begin{cases}
\text{$\frac{1}{18}(n+3)(n+6)$,} &\quad\text{if $n\equiv0$ $(\bmod\text{ }3)$;}\\
\text{$\frac{1}{3}(n+2)$,} &\quad\text{if $n\equiv1$ $(\bmod\text{ }3)$;}\\
\text{$1$,} &\quad\text{if $n\equiv2$ $(\bmod\text{ }3)$}\\
\end{cases}
\end{equation*}
and
 \begin{equation*}
m(C_n)=
\begin{cases}
\text{$3$,} &\quad\text{if $n\equiv0$ $(\bmod\text{ }3)$;}\\
\text{$\frac{1}{6}n(n+5)$,} &\quad\text{if $n\equiv1$ $(\bmod\text{ }3)$;}\\
\text{$n$,} &\quad\text{if $n\equiv2$ $(\bmod\text{ }3)$.}\\
\end{cases}
\end{equation*}
\end{lemma}

\begin{proof}
    Suppose $n\equiv 2$ (mod 3), say,  $n=3k+2$. Let $P_n=u_1u_2\cdots u_n$, $P^{(i)}=u_{3i-2}u_{3i-1}u_{3i}\text{ for }i=1,\ldots ,k$ and $P^{(k+1)}=u_{3k+1}u_{3k+2}$. Suppose  $F\in MD(P_{n})$ is a maximum dissociation set of $P_n$.    By Lemma \ref{lem2.1}, we have ${\rm diss}(P_{3k+2})=2k+2$.  Notice that $F$ contains at most two vertices from each $P^{(i)}$ for $i=1,2,\ldots,k+1$. It follows that $F$   contains exactly two vertices from   each $P^{(i)}$ for $i=1,2,\ldots,k+1$.  Therefore, we have $F=\bigcup_{i=0}^{k}\{u_{3i+1},u_{3i+2}\}$ and $m(P_{n})=1$.

 Suppose $n\equiv 1$ (mod 3), say,  $n=3k+1$.
 We prove $m(P_{3k+1})=k+1$ by using induction on $k$. Clearly, the result holds for $k\in\{0,1\}$. Assume $m(P_{3k-2})=k$ and we consider the path $P_{3k+1}$.

 Let $v$ be an end vertex of $P_{3k+1}$ with $N(v)=\{u\}$. By Lemma \ref{lem2.1}, we have
  $${\rm diss}(P_{3k+1})=2k+1, \quad {\rm diss}(P_{3k-1})=2k\quad  \text{ and }\quad {\rm diss}(P_{3k-2})=2k-1.$$ It follows that
   $$m(P_{3k+1},v^-)=0,\quad m(P_{3k+1},v^0)>0\quad \text{ and }\quad m(P_{3k+1},v^1)>0.$$ By the induction hypothesis, we have
  \begin{align*}
  m(P_{3k+1})&=m(P_{3k+1},v^0)+m(P_{3k+1},v^1) \\
   &=m(P_{3k+1}-v-u)+m(P_{3k+1}-N[u]) \\
   &=1+k\\
   &=\frac{1}{3}(n+2).
  \end{align*}

  Suppose $n\equiv 0$ (mod 3), say,  $n=3k$. Similarly as above, we prove $m(P_{3k})=(k+1)(k+2)/2$ by using induction on $k$. Obviously, the result is true for $k=1$. Assume $m(P_{3k-3})=k(k+1)/2$ and we consider the path $P_{3k}$. Let $v$ be an end vertex of $P_{3k}$ with $N(v)=\{u\}$ and $N(u)=\{v,w\}$. By Lemma \ref{lem2.1}, we have
   $$m(P_{3k},v^-)>0, \quad m(P_{3k},v^0)>0 \quad \text{and}\quad m(P_{3k},v^1)>0.$$
   Since $m(P_{3k+2})=1$ and $m(P_{3k+1})=k+1$, by the induction hypothesis, we have
  \begin{align*}
     m(P_{3k})&=m(P_{3k},v^-)+m(P_{3k},v^0)+m(P_{3k},v^1)\\
    &=m(P_{3k}-v)+m(P_{3k}-\{v,u\})+m(P_{3k}-\{v,u,w\})\\
    &=1+k+\frac{1}{2}k(k+1)\\
    &=\frac{1}{18}(n+3)(n+6).
 \end{align*}

Now  we consider $m(C_n)$. Choose a vertex $v\in V(C_{3k})$.   If $n=3k$, by Lemma \ref{lem2.1}, we have $$m(C_{3k}, v^0)=0,\quad m(C_{3k}, v^-)>0  \quad\text{and}\quad  m(C_{3k}, v^1)>0.$$ It follows that
 \begin{align*}
 m(C_n)&=m(C_{3k})=m(C_{3k}, v^-)+m(C_{3k}, v^1)\\
 &=m(C_{3k}-v)+\sum\limits_{u\in N(v)}m(C_{3k}-N[v]\cup N[u])\\
 &=m(P_{3k-1})+2m(P_{3k-4})\\
 &=3.
\end{align*}
If $n=3k+1$,   by Lemma \ref{lem2.1}, we have
 $$m(C_{3k+1}, v^-)>0, \quad m(C_{3k+1}, v^0)>0\quad \text{and}\quad m(C_{3k+1}, v^1)>0.$$ It follows that
\begin{align*}
m(C_n)&=m(C_{3k+1})=m(C_{3k+1}, v^-)+m(C_{3k+1}, v^0)+m(C_{3k+1}, v^1)\\
&=m(P_{3k})+m(P_{3k-2})+2m(P_{3k-3})\\
&=(k+1)(k+2)/2+k+k(k+1)\\
&=\frac{1}{6}n(n+5).
\end{align*}
If $n=3k+2$, by Lemma \ref{lem2.1},  we have
$$m(C_{3k+2}, v^-)>0, \quad m(C_{3k+2}, v^0)>0\quad\text{and}\quad m(C_{3k+2}, v^1)>0, $$which lead  to
\begin{align*}
m(C_n)&=m(C_{3k+2})=m(C_{3k+2}, v^-)+m(C_{3k+2}, v^0)+m(C_{3k+2}, v^1)\\
&=m(P_{3k+1})+m(P_{3k-1})+2m(P_{3k-2})\\
&=k+1+1+2k\\
&=n.
\end{align*}This completes the proof.
\end{proof}
By Lemma \ref{lem2.1} and Lemma \ref{lem2.2}, we have the following corollary.
\begin{corollary}\label{CLAIM3.2}
Let $u$ be a vertex in a cycle $C_k$.

\item[(i)] If $k\equiv0~(mod \text{ }3),$ then $m(C_k,u^-)=1$, $m(C_k,u^1)=2$, $m(C_k,u^0)=0;$
\item[(ii)] If $k\equiv1~(mod\text{ }3),$ then $m(C_k,u^-)=(k+2)(k+5)/18$, $m(C_k,u^1)=(k-1)(k+2)/9$, $m(C_k,u^0)=(k-1)/3;$
\item[(iii)] If $k\equiv2~(mod\text{ }3),$ then $m(C_k,u^-)=(k+1)/3$, $m(C_k,u^1)=2(k-2)/3$, $m(C_k,u^0)=1.$
\end{corollary}
A {\it cut vertex} of a graph is a vertex whose deletion increases the number of components.
\begin{lemma}\label{lem2.5}
  Suppose $G$ is a graph with a cut vertex $u$ from  $V_1\subset V(G)$.   Denote by $G_1=G[V_1]$ and $G_2=G[V(G)\setminus V_1]$. If ${\rm diss}(G_1-u)={\rm diss}(G_1)$ and $G$ contains no edge  between $V_1\setminus\{u\}$ and $V(G)\setminus V_1$, then $${\rm diss}(G)={\rm diss}(G_1)+{\rm diss}(G_2).$$
\end{lemma}

\begin{proof}
     It is clear that ${\rm diss}(G)\leq {\rm diss}(G_1)+{\rm diss}(G_2)$. On the other hand, let $F_1$ and $F_2$ be two maximum dissociation sets of $G_1-u$ and $G_2$, respectively. Then $F_1\cup F_2$ is a dissociation set of $G$. Therefore, $${\rm diss}(G_1)+{\rm diss}(G_2)={\rm diss}(G_1-u)+{\rm diss}(G_2)=|F_1|+|F_2|\leq {\rm diss}(G).$$ Hence, we have ${\rm diss}(G)={\rm diss}(G_1)+{\rm diss}(G_2)$.
\end{proof}

Let $S$ and $T$ be two collections of subsets of a set $V$. We denote by
 $$S\oplus T=\{a\cup b: a\in S, b\in T\}.$$

\begin{lemma}\label{lem2.6}
    Let $G$ be a graph and $V_1\subset V(G)$. If ${\rm diss}(G)={\rm diss}(G[V_1])+{\rm diss}(G- V_1)$, then we have $$MD(G)\subseteq MD(G[V_1])\oplus MD(G-V_1) $$
    and consequently,
    $$m(G)\leq m(G[V_1])\cdot m(G- V_1).$$
\end{lemma}
\begin{proof}
	For any $F\in MD(G)$, let $F_1=F\cap V_1$ and $F_2=F\cap(V(G)\setminus V_1)$. Then$$|F_1|+|F_2|=|F|={\rm diss}(G)={\rm diss}(G[V_1])+{\rm diss}(G-V_1)\geq|F_1|+|F_2|,$$
which implies $F_1\in MD(G[V_1])$ and $F_2\in MD(G-V_1)$. Hence, we have   $$F\in MD(G[V_1])\oplus MD(G-V_1),$$ which leads to $MD(G)\subseteq MD(G[V_1])\oplus MD(G-V_1) $.
\end{proof}

\begin{corollary}\label{co2.9}
Let $G\in\mathcal{P}(n,k)$ such that $G-C_k$ consists of $r$ components  $T_1,\ldots, T_r$.
Then
\begin{eqnarray*}
 {\rm diss}(G)&=&{\rm diss}(T_x)+{\rm diss}(G-T_x),\\
 MD(G)&\subseteq& MD(T_x)\oplus MD(G-T_x),\\
  m(G)&\le& m(T_x)\cdot m(G-T_x)
\end{eqnarray*}
for $x=1,\ldots,r$.
\end{corollary}
\begin{proof}
 Let $u$ be the vertex in $C_k$ with degree larger than 2. By Lemma \ref{lem2.1}, we have ${\rm diss}(C_k)={\rm diss}(C_k-u)$. Applying Lemma \ref{lem2.5} on $G-T_x$, we have
\begin{align*}
	{\rm diss}(G-T_x)&={\rm diss}(C_k)+{\rm diss}(G-T_x-C_k)\\
	&={\rm diss}(C_k-u)+{\rm diss}(G-T_x-C_k)\\
	&={\rm diss}(G-T_x-u).
\end{align*}Now the results follow  from Lemma \ref{lem2.5} and Lemma \ref{lem2.6}.
\end{proof}

\begin{lemma}\label{CLAIM3.3}
	Let  $G\in\mathcal{P}(n,k)$  such that $G-C_k$ consists of $r$ components $T_1,\ldots,T_r$  and $v_i\in V(T_i)$ is adjacent to $u\in V(C_k)$ for $i=1,\ldots,r$. Then
	\begin{align*}
		m(G)&=m(G,u^-)+m(G,u^0)+m(G,u^1),
	\end{align*}
	where
\begin{eqnarray}
m(G,u^-)=m(C_k,u^-)\prod\limits_{i=1}^{r}m(T_i),\nonumber\\ m(G,u^0)=m(C_k,u^0)\prod\limits_{i=1}^{r}\overline{m}(T_i-v_i)\label{eq0420}
\end{eqnarray} and
\begin{equation*}\label{eqh0413}
m(G,u^1)=m(C_k,u^1)\prod\limits_{i=1}^{r}\overline{m}(T_i-v_i)+m(C_k,u^0)\sum\limits_{j=1}^{r}\left[\overline{m}(T_j-N(v_j))\prod\limits_{i\neq j}\overline{m}(T_i-v_i)\right].
\end{equation*}
	\begin{proof}
By the definitions of $m(G,u^-)$, $m(G,u^0)$ and $m(G,u^1)$ we get the first equation directly.
			Note that ${\rm diss}(C_k)={\rm diss}(C_k-u)$. By Lemma \ref{lem2.5}, we have
\begin{equation}\label{eqh0411}
{\rm diss}(G)={\rm diss}(C_k)+{\rm diss}(G-C_k).
 \end{equation}
 For any maximum dissociation sets $F_1\in MD(C_k)$ and $F_2\in MD(G-C_k)$ such that $u\notin F_1$,   we have $F_1\cup F_2\in MD(G,u^-)$, which implies $$MD(C_k,u^-)\oplus MD(G-C_k)\subseteq MD(G,u^-).$$

On the other hand, for any  $F\in MD(G,u^-)$, let $F_1=F\cap V(C_k)$ and $F_2=F\cap (V(G)\setminus V(C_k))$. Then $$|F_1|+|F_2|={\rm diss}(G)={\rm diss}(C_k)+{\rm diss}(G-C_k)\geq|F_1|+|F_2|,$$which implies that $F_1\in MD(C_k,u^-)$ and $F_2\in MD(G-C_k)$. Hence, we get $$MD(G,u^-)\subseteq MD(C_k,u^-)\oplus MD(G-C_k).$$

Therefore, we have $MD(G,u^-)=MD(C_k,u^-)\oplus MD(G-C_k),$   and hence $$m(G,u^-)=|MD(G,u^-)|=|MD(C_k,u^-)||MD(G-C_k)|=m(C_k,u^-)\prod\limits_{i=1}^{r}m(T_i).$$

If there is some $i\in\{1,\ldots,r\}$ such that ${\rm diss}(T_i)\ne {\rm diss}(T_i-v_i)$, i.e., $\m(T_i-v_i)=0$, then by \eqref{eqh0411}, we have $m(G,u^0)=0$.
If $\diss(T_i)= \diss(T_i-v_i)$ for $i=1,\ldots,r$,
denote by $\mathcal{F}=\{F\in MD(G-C_k): v_i\not\in F,~ i=1,\ldots,r\}$. Then applying similar arguments as above we can deduce
$$MD(G,u^0)=MD(C_k,u^0)\oplus \mathcal{F},$$
which leads to
$$m(G,u^0)=|MD(G,u^0)|=|MD(C_k,u^0)|\cdot| \mathcal{F}|=m(C_k,u^0)\prod\limits_{i=1}^{r}\overline{m}(T_i-v_i).$$
In both cases we get \eqref{eq0420}.

We denote by $$\mathcal{F}_0=MD(C_k,u^1)\oplus MD(G-C_k-\{v_1,\ldots,v_r\})$$ when $ {\rm diss}(T_i)= \diss(T_i-v_i)$ for all  $i=1,\ldots,r$; otherwise let $\mathcal{F}_0=\emptyset$.
For $j=1,\ldots,r$, we denote by $$\mathcal{F}_j=MD(C_k,u^0)\oplus MD(T_j,v_j^0)\oplus MD(G-C_k-T_j-\{v_1,\ldots, v_{j-1},v_{j+1},\ldots,v_r\})$$ if $\diss(T_j-N(v_j))=\diss(T_j)$ and $\diss(T_i-v_i)=\diss(T_i)$ for all $i\in\{1,\ldots,r\}\setminus\{j\}$; otherwise let $\mathcal{F}_j=\emptyset$. Now we prove
\begin{equation}\label{eqh0415}
MD(G,u^1)=\cup_{j=0}^r\mathcal{F}_j.
\end{equation}

 If ${\rm diss}(T_i)={\rm diss}(T_i-v_i) $ for $i=1,\ldots,r$, then for all $F_1\in MD(C_k,u^1)$ and $F_2\in MD(G-C_k-\{v_1,\ldots,v_r\})$, by \eqref{eqh0411} we have $F_1\cup F_2\in MD(G,u^1),$  which leads to $\mathcal{F}_0\subseteq MD(G,u^1)$.

Given $j\in\{1,\ldots,r\}$, if  $\diss(T_j)=\diss(T_i-N(v_j))$ and $\diss(T_i)=\diss(T_i-v_i)$ for all $i\in\{1,\ldots,r\}\setminus\{j\}$,	then for all
 $M_0\in MD(C_k,u^0)$, $M_j\in MD(T_j,v_j^0)$ and $M_i\in MD(T_i,v_i^-)$ for $i\in\{1,\ldots,r\}\setminus\{j\}$, we have $\bigcup_{t=0}^rM_t\in MD(G,u^1),$
		  which implies $ \mathcal{F}_j\subseteq MD(G,u^1)$.

On the other hand, 			
let $F\in MD(G,u^1)$. If $N(u)\cap F\subset V(C_k)$, let $F_1=F\cap V(C_k)$ and $F_2=F\cap (V(G)\setminus V(C_k))$. Then $$|F_1|+|F_2|=|F|=\diss(G)=\diss(C_k)+\diss(G-C_k)\geq|F_1|+|F_2|,$$which implies   $F_1\in MD(C_k,u^1)$ and $F_2\in MD(G-C_k-\{v_1,\ldots,v_r\})$. Therefore, we get    $F\in\mathcal{F}_0$.

 If  $N(u)\cap F\subset V(T_j)$ with $j\in\{1,\ldots,r\}$, let $M_0=F\cap V(C_k)$,  $M_i=F\cap V(T_i)$ for $i=1,\ldots,r$. Then$$\sum\limits_{i=0}^r|M_i|=|F|=\diss(G)=\diss(C_k)+\sum\limits_{i=1}^r\diss(T_i)\geq|M_0|+\sum\limits_{i=1}^r|M_i|,$$which implies that $M_0\in MD(C_k,u^0)$, $M_j\in MD(T_j,v_j^0)$ and $M_i\in MD(T_i,v_i^-)$ for $i\in\{1,\ldots,r\}\setminus\{j\}$. Therefore, we get $F\in\mathcal{F}_j$.

  Therefore  for any $F\in MD(G,u^1)$, we have $F\in\bigcup_{j=0}^r\mathcal{F}_j$, i.e., $MD(G,u^1)\subseteq \bigcup_{j=0}^r\mathcal{F}_j$. Hence,  we have \eqref{eqh0415}.

   Notice that $\mathcal{F}_i\cap\mathcal{F}_j=\emptyset$ for all distinct $i,j\in\{0,1,\ldots,r\}$. We have
			\begin{align*}
				m(G,u^1)&=|MD(G,u^1)|=\sum_{j=0}^r|\mathcal{F}_j|\\
				&=m(C_k,u^1)\prod\limits_{i=1}^{r}\overline{m}(T_i-v_i)+m(C_k,u^0)\sum\limits_{j=1}^{r}\left[\overline{m}(T_j-N(v_j))\prod\limits_{i\neq j}\overline{m}(T_i-v_i)\right].
			\end{align*}This completes the proof.
	\end{proof}
\end{lemma}
 \par
Denote by $$f(x)=3^{x-1}+x+1 \quad \text{and}\quad g(x)=3^{x-1}.$$ Let   $r_1,\ldots, r_i$, $q$ be positive integers such that $r_1+\cdots +r_i=q$. Then
\begin{align}\label{2.3}
	(2r_1+1)\cdot f(r_2)\cdots f(r_i)&\geq(2r_1+1)\times2r_2\times\cdots\times2r_i\notag\\
	&\geq2r_1+1+2r_2+\cdots+2r_i\notag\\
	&=2q+1,
\end{align}
with equality if and only if $i=1$.

Next we  present some inequalities that will be   used repeatedly.

\begin{lemma}\label{lem2.8}
Let $t, k_1, k_2,\ldots, k_i$ be positive integers such that $k_1+k_2+\cdots +k_i=t$. If $i\geq2$, then
\begin{align}\label{2.2}
f(k_1)f(k_2)\cdots f(k_i)+2g(k_1)g(k_2)\cdots g(k_i)< f(t)+2g(t).
\end{align}
\end{lemma}

\begin{proof}
We use induction on $t$.  Since $i\ge 2$, we have $t\ge 2$. If $t=2$, then we have $i=2$ and $k_1=k_2=1$. By a direct computation we have (\ref{2.2}).

  Now we assume $t\ge 3$ and  the result holds for  all integers less than $t$. If $i=t$, then
   $k_1=\cdots=k_i=1$. Again, by a direct computation we have (\ref{2.2}).
Next we suppose $i<t$. Then there exists some $j\in\{1,2,\ldots,i\}$ such that $k_j\geq2$. Without loss of generality,  we may assume $k_1\geq2$.  Let $k_1^{'}=k_1-1$. Then $k'_1+k_2+\cdots+k_i=t-1$. By the induction hypothesis, we have
\begin{align*}
\prod\limits_{r=1}^{i}f(k_r)+2\prod\limits_{r=1}^{i}g(k_r)&=f(k_1^{'}+1)f(k_2)\cdots f(k_i)+2g((k_1^{'}+1)g(k_2)\cdots g(k_i)\\
&= \left[3f(k_1^{'})-(2k_1^{'}+1)\right]f(k_2)\cdots f(k_i)+6g(k_1^{'})g(k_2)\cdots g(k_i)\\
&< 3\left[f(t-1)+2g(t-1)\right]-(2k_1^{'}+1)f(k_2)\cdots f(k_i)\\
&=f(t)+2g(t)+2t-1-(2k_1^{'}+1)f(k_2)\cdots f(k_i)\\
&\leq f(t)+2g(t),
\end{align*}
where the last inequality follows from  (\ref{2.3}).
\end{proof}

Using similar arguments, we have the following two lemmas, whose detailed proofs are presented in Appendix A.
\begin{lemma}\label{lem2.9}
    Let $m,l,k_1,\ldots,k_i$ be positive integers such that  $l\geq5$, $m\ge l+3$ and $$k_1+\cdots+k_i=\left\lfloor\frac{m-l}{3}\right\rfloor.$$

    \item[(i)] If   $m-l\equiv1$ $(mod~3)$, then
    \begin{align}\label{2.5}
    	\frac{1}{3}(l+1)\prod\limits_{j=1}^if(k_j)+\prod\limits_{j=1}^ig(k_j)\leq\frac{1}{3}(l+1)\cdot g\left(\frac{1}{3}(m-l-1)+1\right)+1,
    \end{align}with equality if and only if $i= (m-l-1)/3$, i.e., $k_1=\cdots=k_i=1$.

\item[(ii)] If   $m-l\equiv0$ $(mod~3)$ and $i\ge 2$, then
\begin{eqnarray}\label{2.4}
     &\frac{1}{3}(l+1)\prod\limits_{j=1}^if(k_j)+\frac{1}{3}(2l-1)\prod\limits_{j=1}^ig(k_j)+\sum\limits_{r=1}^i\prod\limits_{j\neq r}g(k_j)\notag\\
     &\leq\frac{1}{3}(l+1)\cdot g\left(\frac{1}{3}(m-l)+1\right)+\frac{1}{3}(m+l-1),
\end{eqnarray}
with equality if and only if $i= (m-l)/3$, i.e., $k_1=\cdots=k_i=1$.
\end{lemma}

\begin{lemma}\label{lem2.10}
   Let $m,l,k_1,\ldots,k_i$ be positive integers such that  $l\geq4$, $m\ge l+3$ and $$k_1+\cdots+k_i=\left\lfloor\frac{m-l}{3}\right\rfloor.$$

\item[(i)] If   $m-l\equiv1$ $(mod\text{ }3)$, then
\begin{align}\label{2.7}
    &\frac{1}{18}(l+2)(l+5)\prod\limits_{j=1}^if(k_j)+\frac{1}{3}(l-1)\prod\limits_{j=1}^ig(k_j)\notag\\
    &\leq\frac{1}{18}(l+2)(l+5) \cdot g\left(\frac{1}{3}(m-l-1)+1\right)+\frac{1}{3}(l-1),
\end{align}with equality if and only if $i=(m-l-1)/3$.

\item[(ii)] If  $m-l\equiv0$ $(mod\text{ }3)$  and $i\ge 2$, then
\begin{align}\label{2.6}
	&\frac{1}{18}(l+2)(l+5)\prod\limits_{j=1}^if(k_j)+\frac{1}{9}(l-1)(l+2)\prod\limits_{j=1}^ig(k_j)+\frac{1}{3}(l-1)\left[\prod\limits_{j=1}^ig(k_j)+\sum\limits_{r=1}^i\prod\limits_{j\neq r}g(k_j)\right]\notag\\
	&\le \frac{1}{18}(l+2)(l+5)\cdot g\left(\frac{1}{3}(m-l)+1\right)+\frac{1}{9}(m+5)(l-1),
\end{align}with equality if and only if $i= (m-l)/3.$
\end{lemma}
\begin{lemma}\label{le2.15}
	 Let  $s_1, s_2\geq3$ be  integers.
	 \item (i) If $s_1\equiv1$ (mod 3), $s_2\equiv2$ (mod 3), then
	\begin{equation*}
		\begin{cases}
			\text{$(3^{(s_1-1)/3-1}+1)\cdot3^{(s_2-2)/3-1}<3^{(s_1+s_2)/3-1}+\frac{1}{3}(s_1+s_2)+1$,}\\
			\text{$3^{(s_1-1)/3-1}\cdot3^{(s_2-2)/3-1}<3^{(s_1+s_2)/3-1}$.}
		\end{cases}
	\end{equation*}
	\item (ii) If $s_1$, $s_2\equiv2$ (mod 3), then
	\begin{equation*}
		\begin{cases}
			\text{$3^{(s_1-2)/3-1}\cdot3^{(s_2-2)/3-1}<3^{(s_1+s_2-1)/3-1}+1$,}\\
			\text{$3^{(s_1-2)/3-1}\cdot3^{(s_2-2)/3-1}<3^{(s_1+s_2-1)/3-1}$.}
		\end{cases}
	\end{equation*}
	\item (iii) If $s_1\equiv1$ (mod 3), $s_2\equiv1$ (mod 3) and $s_1$, $s_2$ are not both equal to 4, then
	\begin{equation*}
		\begin{cases}
			\text{$(3^{(s_1-1)/3-1}+1)\cdot(3^{(s_2-1)/3-1}+1)<3^{(s_1+s_2-2)/3-1}$,}\\
			\text{$3^{(s_1-1)/3-1}\cdot3^{(s_2-1)/3-1}<3^{(s_1+s_2-2)/3-1}$.}
		\end{cases}
	\end{equation*}
	\item (iiii) If $s_1=s_2=s_3=4$, then$$\prod\limits_{i=1}^33^{(s_i-1)/3-1}<\prod\limits_{i=1}^3(3^{(s_i-1)/3-1}+1)<3^{(s_i+s_2+s_3)/3-1}<3^{(s_i+s_2+s_3)/3-1}+\frac{1}{3}(s_i+s_2+s_3)+1.$$
\end{lemma}
\begin{proof}
A straightforward computation can verify the above inequalities.
\end{proof}
\section{Proofs of the main results}
In this section, firstly we determine the maximum value of $m(G)$ with $G\in \mathcal{P}(n,k)$ such that all component of $G-C_k$ have fixed orders. Then we present the proofs of the main results.

\subsection{The case that all components of $G-C_k$ have fixed orders}

Firstly we deal with the case that each component of $G-C_k$ has order at leat 3.
\begin{theorem}\label{cl3.4}
	Let $n,k,n_1,\ldots,n_r$ $(r\ge1)$ be positive integers larger than 2 such that $n=k+n_1+\cdots+n_r$. Suppose $G\in\mathcal{P}(n,k)$  such that $G-C_k$ consists of $r$ components $T_1,\ldots,T_r$ with $|V(T_i)|=n_i$ for $i=1,\ldots,r$. Then $m(G)$ attains the maximum if and only if
\begin{equation*}\label{eqh041}
G=\bigcup_{i=1}^rT_i\cup C_k\cup \{uv_1,\ldots,uv_r\}
\end{equation*}
such that $u\in V(C_k)$  and $T_i\in\bigcup_{j=1}^5\mathcal{T}_j(v_i)$, with an additional case that $v_i$ is an arbitrary vertex of $T_i=P_3$ when $k\equiv 0~(mod~ 3)$ and $n_i=3$,    $i= 1,\ldots,r$.
\end{theorem}
\begin{proof}

Let $G=\bigcup_{i=1}^rT_i\cup C_k\cup \{uv_1,\ldots,uv_r\}$ with $u\in V(C_k)$ and $v_i\in V(T_i)$ for $i=1,\ldots, r$.
Suppose $m(G)$ attains the maximum.

For $x\in\{1,\ldots,r\}$, denote by
\begin{equation}\label{eqh0410}
G_x^*=(G-T_x)\cup T_x^*\cup\{uv^*_x\},
 \end{equation}
 where $T_x^*\in\bigcup_{i=1}^5\mathcal{T}_i(v_x^*)$ with $|V(T_x^*)|=n_x$.  A direct computation shows
\begin{equation}\label{eq0403}
\overline{m}(T_x^*-v_x^*)=3^{\lfloor n_x/3\rfloor-1};
\end{equation}
\begin{equation}\label{eq0404}
m(T_x^*)\equiv \left\{\begin{array}{ll}
   3^{n_x/3-1}+ {n_x}/{3}+1,& \text{  if  }n_x\equiv 0 \text{ (mod 3)};\\
3^{(n_x-1)/3-1}+1,& \text{  if  }n_x\equiv 1 \text{ (mod 3)};\\
3^{(n_x-2)/3-1},& \text{  if  }n_x\equiv 2 \text{ (mod 3)}
 \end{array}\right.
 \end{equation}
 and
\begin{equation}\label{eq0405}
\m(T_x^*-N(v_x^*))\equiv \left\{\begin{array}{ll}
   1,& \text{  if  }n_x\equiv 0 \text{ (mod 3)};\\
0,& \text{  otherwise.  }
 \end{array}\right.
 \end{equation}

 Firstly we  claim
	\begin{equation}\label{claim.}
		m(G,u^1)\neq 0.
	\end{equation}
Otherwise, suppose $m(G,u^1)=0$.  By Lemma \ref{CLAIM3.3}, we have
$$m(G,u^1)=m(C_k,u^1)\prod\limits_{i=1}^{r}\overline{m}(T_i-v_i)+m(C_k,u^0)\sum\limits_{j=1}^{r}\left[\overline{m}(T_j-N(v_j))\prod\limits_{i\neq j}\overline{m}(T_i-v_i)\right].$$
Since $m(C_k,u^1)\ne 0$, there exists $x\in\{1,\ldots,r\}$ such that $\overline{m}(T_x-v_x)=0$, which leads to $$m(G,u^0)= m(C_k,u^0)\prod\limits_{i=1}^{r}\overline{m}(T_i-v_i)=0.$$ Thus, applying Lemma \ref{lem2.3} and Lemma \ref{CLAIM3.3}, we have
	\begin{align*}
		m(G)=m(G,u^-)=m(C_k,u^-)\prod\limits_{i=1}^rm(T_i)\le m(C_k,u^-)\prod\limits_{i=1}^rm(T_i^*).
	\end{align*}
	Let $G^{'}=\bigcup_{i=1}^rT_x^*\cup C_k\cup\{uv_1^*,\ldots,uv_r^*\}$. Then by  Lemma \ref{CLAIM3.3}, we have
	\begin{align*}
		m(G^{'})&=m(G^{'},u^-)+m(G^{'},u^0)+m(G^{'},u^1)\\
		&=m(C_k,u^-)\prod\limits_{i=1}^rm(T_i^*)+m(G^{'},u^0)+m(G^{'},u^1).
	\end{align*}
By Lemma \ref{CL3.3} and Lemma \ref{CLAIM3.3},  we have $m(G^{'},u^1)\neq0$. Thus $m(G^{'})>m(G)$, which contradicts the fact that $m(G)$ attains the maximum.  Therefore, we have (\ref{claim.}).
	
For any  $x\in\{1,\ldots r\}$, we distinguish two cases.
	
	{\it Case 1.} $k\equiv0$ (mod 3).	By Lemma \ref{CLAIM3.3} and Corollary \ref{CLAIM3.2}, we have
	\begin{align}
		m(G)&=m(C_k,u^-)\prod_{i=1}^rm(T_i)+m(C_k,u^1)\prod_{i=1}^r\overline{m}(T_i-v_i)\notag \\
&=m(T_x)\cdot m(C_k,u^-)\prod_{i\ne x}m(T_i)+\m(T_x-v_x)\cdot m(C_k,u^1)\prod_{i\ne x}\overline{m}(T_i-v_i)\label{Equation09}
	\end{align}
	and
\begin{align}
		m(G_x^*)=m(T_x^*)\cdot m(C_k,u^-)\prod_{i\ne x}m(T_i)+\m(T^*_x-v_x^*)\cdot m(C_k,u^1)\prod_{i\ne x}\overline{m}(T_i-v_i)\label{eq042}
	\end{align}
Moreover, by (\ref{claim.}) we have $\prod_{i\ne x}\overline{m}(T_i-v_i)\ne 0.$

If $n_x=3$, then $T_x=P_3$, $m(T_x)=3$ and $\m(T_x-v_x)=1$ for all $v_x\in V(T_x)$. It follows that \eqref{Equation09} has the same value no matter how we choose $v_x$. Therefore, we get $T_x=P_3$ with $v_x$ being an arbitrary vertex of $T_x$. Next we suppose $n_x\ge 4$.

	{\it Subcase 1.1.}  $n_x\equiv 0$ (mod 3).
If $T_x\notin\mathcal{T}_1$,
	  by Lemma \ref{lem2.3} and Lemma \ref{lem2.4}, we have $$m(T_x)\le 3^{n_x/3-1}+\frac{n_x}{3}~~\text{and}~~\overline{m}(T_x-v_x)\le 3^{n_x/3-1}.$$
If $T_x\in\mathcal{T}_1(w)$ with $ w\neq v_x$, by a direct computation  and Lemma \ref{CL3.3}, we have
 $$m(T_x)=3^{n_x/3-1}+\frac{n_x}{3}+1~~\text{and}~~\overline{m}(T_x-v_x)< 3^{n_x/3-1}.$$
In both cases we can combine (\ref{eq0403}), (\ref{eq0404}), (\ref{Equation09}) and (\ref{eq042}) to deduce $m(G)<m(G^*_x)$, a contradiction.
 Therefore, we have $ T_x\in\mathcal{T}_1(w) $ with $w=v_x$.

	{\it Subcase 1.2.}  $n_x\equiv 1$ (mod 3).
If $T_{x}\notin\mathcal{T}_2$, then by Lemma \ref{lem2.3}, we have
$$\m(T_x-v_x)\le m(T_{x})\leq3^{(n_x-1)/3-1}.$$
If  $T_x\in\mathcal{T}_2(w)$ with $ w\neq v_x$,  then by a direct computation  and Lemma \ref{CL3.3}, we get $$m(T_x)=3^{(n_x-1)/3-1}+1~~\text{and}~~\overline{m}(T_x-v_x)<3^{(n_x-1)/3-1}.$$
Again, in both cases we can combine (\ref{eq0403}), (\ref{eq0404}),  (\ref{Equation09}) and (\ref{eq042}) to deduce $m(G)<m(G^*_x)$, a contradiction.
 Therefore, we have $ T_x\in\mathcal{T}_2(w) $ with $w=v_x$.
	
	{\it Subcase 1.3.}  $n_x\equiv 2$ (mod 3).  If $T_{x}\notin\mathcal{T}_3\cup\mathcal{T}_4\cup\mathcal{T}_5$, then by Lemma \ref{lem2.3}, we have $$\m(T_x-v_x)\le m(T_{x})<3^{(n_x-2)/3-1}.$$
If $T_{x}\in\mathcal{T}_3(w)\cup\mathcal{T}_4(w)\cup\mathcal{T}_5(w)$ with $ w\neq v_x$,  then by a direct computation and Lemma \ref{CL3.3}, $$m(T_x)=3^{(n_x-2)/3-1}~~\text{and}~~\overline{m}(T_x-v_x)<3^{(n_x-2)/3-1}.$$
Again, in both cases we can combine (\ref{eq0403}), (\ref{eq0404}),  (\ref{Equation09}) and (\ref{eq042}) to deduce $m(G)<m(G^*_x)$, a contradiction.
 Therefore, we have
$T_{x}\in\mathcal{T}_3(w)\cup\mathcal{T}_4(w)\cup\mathcal{T}_5(w)$ with $w =v_x$.
	
	{\it Case 2.} $k\not\equiv0$ (mod 3). By Lemma \ref{CLAIM3.3}, we have
	\begin{align}
		m(G^*_x)&= m(T_x^*)\cdot m(C_k,u^-)\prod\limits_{i\neq x}m(T_i)+\m(T_x^*-v_x^*)\cdot m(C_k,u^1)\prod\limits_{i\neq x}\overline{m}(T_i-v_i)\notag\\
		&\quad+\m(T_x^*-v_x^*)\cdot m(C_k,u^0)\left[\prod\limits_{i\neq x}\overline{m}(T_i-v_i)+\sum\limits_{j\neq x}\Big(\overline{m}(T_j-N(v_j))\prod\limits_{i\neq j,x}\overline{m}(T_i-v_i)\Big)\right]\notag\\
		&\quad+\m(T_x^*-N(v_x^*))\cdot m(C_k,u^0)\prod\limits_{i\neq x}\overline{m}(T_i-v_i)\label{eq0407}
	\end{align}
and
\begin{align}
		m(G)&= m(T_x)\cdot m(C_k,u^-)\prod\limits_{i\neq x}m(T_i)+\m(T_x-v_x)\cdot m(C_k,u^1)\prod\limits_{i\neq x}\overline{m}(T_i-v_i)\notag\\
		&\quad+\m(T_x-v_x)\cdot m(C_k,u^0)\left[\prod\limits_{i\neq x}\overline{m}(T_i-v_i)+\sum\limits_{j\neq x}\Big(\overline{m}(T_j-N(v_j))\prod\limits_{i\neq j,x}\overline{m}(T_i-v_i)\Big)\right]\notag\\
		&\quad+\m(T_x-N(v_x))\cdot m(C_k,u^0)\prod\limits_{i\neq x}\overline{m}(T_i-v_i).\label{eq0408}
	\end{align}
We distinguish three subcases.

	{\it Subcase 2.1.} $n_x\equiv0$ (mod 3).  Suppose $T_{x}\notin\mathcal{H}$. By Corollary \ref{cor3.5} and Lemma \ref{lem2.4}, we have
\begin{equation}\label{eq0406}
m(T_{x})\leq3^{n_x/3-1}+1\quad{\rm and}\quad\overline{m}(T_{x}-v_x)\leq3^{n_x/3-1}.
\end{equation}
Applying  Corollary \ref{co2.9}, we have
\begin{eqnarray*}
MD(G)&\subseteq& MD(T_x)\oplus MD(G-T_x)\\
&=&MD(T_x)\oplus \left[MD(G-T_x,u^-) \cup  MD(G-T_x,u^0) \cup  MD(G-T_x,u^1)\right]\\
&=&\left[MD(T_x)\oplus \left(MD(G-T_x,u^-) \cup  MD(G-T_x,u^0)\right)\right] \cup  \left[MD(T_x)\oplus MD(G-T_x,u^1)\right].
\end{eqnarray*}

Suppose $v_x$ is in a maximum dissociation set of $T_x$, say,  $v_x\in F_1\in MD(T_x)$. Then for all $F_2\in  m(G-T_x,u^1)$, we have $F_1\cup F_2\not\in MD(G)$. Therefore, by (\ref{eq0403}),(\ref{eq0404}), (\ref{eq0405}), (\ref{eq0406}) and Lemma \ref{CLAIM3.3}, we have
	\begin{align*}
		m(G)&\le m(T_x) m(G-T_x,u^-)+m(T_x) m(G-T_x,u^0)+[m(T_x)-1] m(G-T_x,u^1)\\
		&\le (3^{n_x/3-1}+1)m(C_k,u^-)\prod\limits_{i\neq x}m(T_i)+(3^{n_x/3-1}+1)m(C_k,u^0)\prod\limits_{i\neq x}\overline{m}(T_i-v_i)\\
		&~~~+3^{n_x/3-1} \left[m(C_k,u^1)\prod\limits_{i\neq x}\overline{m}(T_i-v_i)+
		m(C_k,u^0)\sum\limits_{j\neq x}\Big(\overline{m}(T_j-N(v_j))\prod\limits_{i\neq j,x}\overline{m}(T_i-v_i)\Big)\right]\\
		& <m(G^*_x),
	\end{align*}a contradiction.

 If $v_x$ is not in any  maximum dissociation sets of $T_x$, by Lemma \ref{lem2.4}, we have  $$m(T_x)=m(T_x-v_x)\leq3^{n_x/3-1}  $$
 and $$\diss(G)=\diss(G-T_x)+\diss(T_x).$$
 By (\ref{eq0403}),(\ref{eq0404}), (\ref{eq0405}), Corollary \ref{co2.9}   and  Lemma \ref{CLAIM3.3}, we have
	\begin{align}
		m(G)&\le m(T_x)m(G-T_x)\notag\\
		&\le  3^{n_x/3-1} m(C_k,u^-)\prod\limits_{i\neq x}m(T_i)+3^{n_x/3-1} m(C_k,u^1)\prod\limits_{i\neq x}\overline{m}(T_i-v_i)\notag\\
		&\quad+3^{n_x/3-1} m(C_k,u^0)\left[\prod\limits_{i\neq x}\overline{m}(T_i-v_i)+\sum\limits_{j\neq x}\Big(\overline{m}(T_j-N(v_j))\prod\limits_{i\neq j,x}\overline{m}(T_i-v_i)\Big)\right]\notag\\
		&<m(G^*_x),\notag
	\end{align}a contradiction.  Therefore, we have $T_x\in\mathcal{H}.$
	
  By computing $m(T_x)$, $\m(T_x-v_x)$, $\m(T_x-N(v_x))$ and using (\ref{eq0403}), (\ref{eq0404}), (\ref{eq0405}) and (\ref{claim.}), we can compare (\ref{eq0407}) and (\ref{eq0408}) case by case to deduce $m(G)<m(G_x^*)$ if $T_x\notin\mathcal{T}_1$ or $T_x\in\mathcal{T}_1(w)$ with $w\ne v_x$. We list all the possibilities of $m(T_x)$, $\m(T_x-v_x)$ and $\m(T_x-N(v_x))$ when $T_x\in \mathcal{H}$ with $v_x\in V(T_x)$ in the Tables 1-3 in Appendix B  and omit the details here.

	{\it Subcase 2.2.} $n_x\equiv1$ (mod 3).
   Suppose $T_x\notin\mathcal{T}_2$. By Lemma \ref{lem2.3}, we get $m(T_x)\le 3^{(n_x-1)/3-1}$. Applying Corollary \ref{co2.9} and Lemma \ref{CLAIM3.3} on $G-T_x$, we have
	\begin{align}
		m(G)&\le m(T_x)\cdot m(G-T_x) \notag\\
		&\leq 3^{(n_x-1)/3-1}m(C_k,u^-)\prod\limits_{i\neq x}m(T_i)+3^{(n_x-1)/3-1} m(C_k,u^1)\prod\limits_{i\neq x}\overline{m}(T_i-v_i)\notag\\
		&\quad+3^{(n_x-1)/3-1} m(C_k,u^0)\sum\limits_{j\neq x}\Big(\overline{m}(T_j-N(v_j))\prod\limits_{i\neq j,x}\overline{m}(T_i-v_i)\Big)\notag\\
		&\quad+3^{(n_x-1)/3-1}m(C_k,u^0)\prod\limits_{i\neq x}\overline{m}(T_i-v_i)<m(G^*_x),\notag
	\end{align}a contradiction. Therefore, $T_x\in\mathcal{T}_2$.

Again, by computing $m(T_x)$, $\m(T_x-v_x)$, $\m(T_x-N(v_x))$ and using (\ref{eq0403}), (\ref{eq0404}), (\ref{eq0405}) and (\ref{claim.}), we can compare (\ref{eq0407}) and (\ref{eq0408}) case by case to deduce $m(G)<m(G_x^*)$ if $T_x\in\mathcal{T}_2(w)$ with $w\ne v_x$;
see  Table 4 in Appendix B for  all the possibilities of $m(T_x)$, $\m(T_x-v_x)$ and $\m(T_x-N(v_x))$.
	
	{\it Subcase 2.3.} $n_x\equiv2$ (mod 3). Since $v_x^*$
is not in any maximum dissociation set of $T^*_x$,   we have$$m(G^*)=m(T^*_x)\cdot m(G^*_x-T_x^*)=3^{(n_x-2)/3-1}m(G^*_x-T_x^*).$$	Suppose $T_x\notin\mathcal{T}_3\cup \mathcal{T}_4\cup \mathcal{T}_5$. Then by Lemma \ref{lem2.3}, we get $m(T_x)<3^{(n_x-2)/3-1}$. Note that $G-T_x\cong G^*_x-T^*_x$. By Corollary \ref{co2.9}, we have $$m(G)\le m(T_x)\cdot m(G-T_x)<3^{(n_x-2)/3-1}m(G-T_x)=m(G^*_x),$$ a contradiction. Thus, $T_x\in\mathcal{T}_3\cup \mathcal{T}_4\cup \mathcal{T}_5$.

Suppose $T_x\in\mathcal{T}_3(w)\cup \mathcal{T}_4(w)\cup \mathcal{T}_5(w)$ with $w\neq v_x$.
Then there is a maximum dissociation set $F_1\in MD(T_x)$ containing $v_x$.

On the other hand, we have
$m(G-T_x,u^1)\ne 0$. Otherwise, assume $m(G-T_x,u^1)=0$. Applying Lemma \ref{CLAIM3.3} on $G-T_x$, either there exists $i\neq x$ such that $\overline{m}(T_i-v_i)=0$ and  $\overline{m}(T_i-N(v_i))=0$, or there exist two distinct $i,j\neq x$ such that $\overline{m}(T_i-v_i)=\overline{m}(T_j-v_j)=0$, which both lead to $m(G,u^1)=0$ and contradicts (\ref{claim.}).

Since $F_1\cup F_2\notin MD(G)$ for all $F_2\in MD(G-T_x,u^1)$,    by Corollary \ref{co2.9}, we have $$m(G)<m(T_x)\cdot m(G-T_x)=3^{(n_x-2)/3-1}m(G-T_x)=m(G^*_x),$$ a contradiction. Therefore, we have $T_x\in\mathcal{T}_3(w)\cup \mathcal{T}_4(w)\cup \mathcal{T}_5(w)$ with $w=v_x$.

Combining the above cases,  for $x=1,\ldots,r$, $T_x$ satisfies
\begin{itemize}
\item[(i)] if $n_x\ge 4$, then $T_x\in\bigcup_{j=1}^5\mathcal{T}_j(v_x)$;
\item[(ii)] if $n_x=3$ and  $k\equiv 0~(\mod 3)$, $v_x$ is an arbitrary vertex of $T_x=P_3$;
\item[(iii)] if $n_x=3$ and  $k\not\equiv 0~(\mod  3)$, $v_x$ is an end vertex of $T_x=P_3$, i.e., $T_x\in \mathcal{T}_1(v_x)$.
\end{itemize}
Notice   that for $T, T'\in \bigcup_{j=1}^5\mathcal{T}_j(w)$ with $|V(T)|=|V(T')|$. We always have
$$m(T)=m(T'),\quad \m(T-w)=\m(T'-w) \quad {\rm and}\quad \m(T-N(w))=\m(T'-N(w)).$$
By Lemma \ref{CLAIM3.3}, $m(G)$ is a constant if it satisfies the conditions (i), (ii) and (iii). Thus we get the conclusion.
\end{proof}

\begin{corollary}\label{cl3.8}	Let $n,k,n_1,\ldots,n_{r-1}$ $(r\ge2)$ be positive integers larger than 2 and $n_r=1$ such that $n=k+n_1+\cdots+n_r$. Suppose $G\in\mathcal{P}(n,k)$  such that $G-C_k$ consists of $r$ components $T_1,\ldots,T_r$ with $|V(T_i)|=n_i$ for $i=1,\ldots,r$. If $k\not\equiv0$ (mod 3), then $m(G)$ attains the maximum if and only if $G=\bigcup_{i=1}^rT_i\cup C_k\cup \{uv_1,\ldots,uv_r\}$ with $u\in V(C_k)$  and $T_i\in\bigcup_{j=1}^5\mathcal{T}_j(v_i)$, with an additional case that $v_i$ is an arbitrary vertex of $T_i=P_3$ when   $n_i=3$,    $i= 1,\ldots,r-1$.
\end{corollary}
\begin{proof}	Let $G=\bigcup_{i=1}^rT_i\cup C_k\cup \{uv_1,\ldots,uv_r\}$ with $u\in V(C_k)$ and $v_i\in V(T_i)$ for $i=1,\ldots r$.
Suppose $m(G)$ attains the maximum. For $x=1,\ldots,r$, we define $G^*_x$ the same as (\ref{eqh0410}).  Notice that $\m(T_r-v_r)=0$ and $\m(T_r-N(v_r))=1$.  Applying Lemma \ref{CLAIM3.3}, we have
\begin{eqnarray}\label{11}
m(G)&=&m(C_k,u^-)\prod\limits_{i=1}^{r-1}m(T_i)+m(C_k,u^0)\prod\limits_{i=1}^{r-1}\overline{m}(T_i-v_i)\nonumber\\
&=&m(T_x)\cdot m(C_k,u^-)\prod\limits_{1\le i\le r-1,i\ne x}m(T_i)+\m(T_x-v_x)\cdot m(C_k,u^0)\prod\limits_{1\le i\le r-1,i\ne x}\overline{m}(T_i-v_i)\nonumber\\
\end{eqnarray}
and
\begin{align*}
m(G^*_x)=m(T_x^*)\cdot m(C_k,u^-)\prod\limits_{1\le i\le r-1,i\ne x}m(T_i)+\m(T_x^*-v_x^*)\cdot m(C_k,u^0)\prod\limits_{1\le i\le r-1,i\ne x}\overline{m}(T_i-v_i).
\end{align*}

We firstly  claim that
\begin{equation*}
	 \overline{m}(T_x-v_x)\ne 0 \quad {\rm for~all }\quad v_x\in V(T_x) ~{\rm and}~  x\in\{1,\ldots,r-1\}.
\end{equation*}
Otherwise, suppose there exists $x\in\{1,\ldots,r-1\}$ such that $\overline{m}(T_x-v_x)=0$ with $v_x\in V(T_x)$. By (\ref{11}) and Lemma \ref{lem2.3}, we have
\begin{align*}
	m(G)=m(C_k,u^-)\prod\limits_{i=1}^{r-1}m(T_i)\le m(C_k,u^-)\prod\limits_{i=1}^{r-1}m(T_i^*).
\end{align*}
Let $G^{'}=\bigcup_{i=1}^{r-1}T_i^*\cup T_r\cup C_k\cup\{uv_1^*,\ldots,uv_{r-1}^*,uv_r\}$. Then applying Lemma \ref{CLAIM3.3}, we have
\begin{align*}	m(G^{'})=m(C_k,u^-)\prod\limits_{i=1}^{r-1}m(T_i^*)+m(C_k,u^0)\prod\limits_{i=1}^{r-1}\overline{m}(T_i^*-v_i^*).
\end{align*} Notice that $\prod_{i=1}^{r-1}\overline{m}(T_i^*-v_i^*)\neq0$. We have $m(G^{'})>m(G)$, a contradiction.

Now applying the same arguments as in Case 1 of the proof of Theorem \ref{cl3.4}, we can deduce that  $T_i$ satisfies the following conditions for all $i= 1,\ldots,r-1$:
 \begin{itemize}
 \item[(i)] if $n_i\ge 4$, then $T_i\in\bigcup_{j=1}^5\mathcal{T}_j(v_i)$;
  \item[(ii)] if $n_i=3$, then $T_i=P_3$ with $v_i$ being an arbitrary vertex.
 \end{itemize}
 Notice that $m(G)$ is a constant if  $T_i$ satisfies (i) and (ii) for all     $i= 1,\ldots,r-1$. Thus we get the conclusion.
\end{proof}

\subsection{Proof of Theorem \ref{th1.1}}
\noindent\textit{Proof.}  Suppose $G\in\mathcal{P}(n,k)$ such that $G-C_k$ consists of $r$ components $T_1,\ldots,T_r$ with $|V(T_i)|=n_i$ for $i=1,\ldots,r$. Let $v_i\in V(T_i)$ be adjacent to $u\in V(C_k)$ for $i=1,\ldots,r$. We distinguish two cases.

{\it Case 1.} At least  one of $ n_1,\ldots,n_r $ is less than 3, say, $n_i\le 2$ for $i=1, \ldots, p$ with $p\ge 1$.
By Corollary \ref{CLAIM3.2},  we have $m(C_k,u^0)=0$. Notice that $\overline{m}(T_1-v_1)=0$. Applying Lemma \ref{CLAIM3.3}, we have $m(G,u^0)=m(G,u^1)=0$.  Therefore, $$m(G)=m(G,u^-)=m(C_k,u^-)\cdot m(G-C_k)=m(G-C_k).$$ Now applying Lemma \ref{lem2.4}, we have
\begin{equation}\label{equa0}
    m(G)=m(G-C_k)\le
    \begin{cases}
        \text{$3^{(n-k)/3}<3^{(n-k)/3}+\frac{1}{3}(n-k)+1$}, &\text{if $n\equiv0$ (mod 3);}\\
         \text{$3^{(n-k-1)/3}<3^{(n-k-1)/3}+1$}, &\text{if $n\equiv1$ (mod 3);}\\
          \text{$3^{(n-k-2)/3}$}, &\text{if $n\equiv2$ (mod 3).}\\
    \end{cases}
\end{equation}
Moreover, if $p=1$, then   equality in the case $n\equiv2$ (mod 3) holds  if and only if $n_1=2$ and $T_i\cong P_3$ for $i=2,\ldots,r$,  i.e., $G\in \mathcal{G}_6$; if $p=2$ and $n_1=n_2=1$,   equality in the case $n\equiv2$ (mod 3) holds if and only if $T_i\cong P_3$ for $i=3,\ldots,r$,  i.e., $G\in \mathcal{G}_7$; for   the other cases, all inequalities in (\ref{equa0}) are strict.

{\it Case 2.} $n_i\geq3$ for all $i=1,\ldots,r$. Without loss of generality, we assume
\begin{equation}\label{eqn1}
n_i\equiv \left\{\begin{array}{ll}
 0\text{ (mod 3)},& \text{  if  }i=1,\ldots,a;\\
 1\text{ (mod 3)},& \text{  if  }i=a+1,\ldots,b;\\
 2\text{ (mod 3)},& \text{  if  }i= b+1,\ldots,r;
 \end{array}\right.
 \end{equation}
 where $a=0$, $a=b$ and $b=r$  means there is no $i$ such that $n_i\equiv 0~(\mod 3)$, $n_i\equiv 1~(\mod 3)$ and  $n_i\equiv 2~(\mod 3)$, respectively.
By Theorem \ref{cl3.4}, $m(G)$ attains the maximum if and only if $G\cong G^*\equiv\bigcup_{i=1}^rT_i\cup C_k\cup \{uv_1,\ldots,uv_r\}$ with $u\in V(C_k)$ and
\begin{equation}\label{eqn2}
T_i\in \left\{\begin{array}{ll}
 \mathcal{T}_1(v_i),& \text{  if  }i=1,\ldots,a;\\
 \mathcal{T}_2(v_i),& \text{  if  }i=a+1,\ldots,b;\\
 \bigcup_{j=3}^5\mathcal{T}_j(v_i),& \text{  if  }i= b+1,\ldots,r,
 \end{array}\right.
 \end{equation} with an additional case that $v_i$ is an arbitrary vertex of $T_i=P_3$ when $n_i=3$ for $i= 1,\ldots,r$.
Moreover, by Lemma \ref{CLAIM3.3} and Corollary \ref{CLAIM3.2},   we have
\begin{align}\label{3.1}
   m(G^*) =&\prod\limits_{i=1}^{ a }(3^{n_i/3-1}+\frac{n_i}{3}+1)\prod\limits_{i=a+1}^{b}(3^{(n_i-1)/3-1}+1)
   \prod\limits_{i=b+1}^{ r}3^{(n_i-2)/3-1}\notag\\
   &+2\prod\limits_{i=1}^{ a }3^{n_i/3-1}\prod\limits_{i=a+1}^{b}3^{(n_i-1)/3-1}\prod\limits_{i=b+1}^{ r}3^{(n_i-2)/3-1},
\end{align}
where the item  $\prod_{i=1}^a(\cdots)$ does not appear if $a=0$,  the item   $\prod_{i=a+1}^b(\cdots)$ does not appear if $b=a$, and the item  $\prod_{i=b+1}^r(\cdots)$ does not appear if $r=b$.

Now we distinguish three cases.

{\it Subcase 2.1.} $n\equiv 0$ (mod 3).  By Lemma \ref{lem2.8} and Lemma \ref{le2.15}, (\ref{3.1})  attains  the maximum  if and only if $a=r=1$ and $n_1=n-k\ge 4$. Therefore,
\begin{align*}
	m(G)\le m(G^*)=3^{(n-k)/{3}}+\frac{1}{3}(n-k)+1,
\end{align*}
with equality if and only if $G\cong   G^{*} \in\mathcal{G}_1$.

 {\it Subcase 2.2.}  $n\equiv1$ (mod 3). Suppose (\ref{3.1})   attains  the maximum.  Then by Lemma \ref{le2.15},  we have $r=b= a +1$. If $a\ne 0$, by Lemma \ref{lem2.8}, we have
\begin{align}
     m(G)\le m(G^*)&=(3^{(n_{r}-1)/3-1}+1)\prod_{i=1}^a(3^{n_i/3-1}+\frac{n_i}{3}+1)
     +2\times3^{(n_{r}-1)/3-1} \prod_{i=1}^a3^{n_i/3-1}\notag\\
     &= 3^{(n_{r}-1)/3-1}\left[\prod_{i=1}^a(3^{n_i/3-1}+\frac{n_i}{3}+1)
     +2 \prod_{i=1}^a3^{n_i/3-1}\right]+\prod_{i=1}^a(3^{n_i/3-1}+\frac{n_i}{3}+1)\notag\\
    &\leq 3^{(n_{r}-1)/{3}-1}\left[3^{(n-k-n_{r})/{3}}+\frac{1}{3} (n-k-n_{r})+1\right]  +\prod_{i=1}^a\left(3^{ {n_i}/{3}-1}+\frac{n_i}{3}+1\right)\notag\\
    &\leq3^{(n-k-1)/{3}-1}+\left[\frac{1}{3} (n-k-n_{r})+1\right]  \times 3^{(n_{r}-1)/3-1}+3^{(n-k-n_r)/3}\notag\\
    &\leq2\times3^{(n-k-1)/3-1}+\left[\frac{1}{3} (n-k-n_{r})+1\right]\times 3^{(n_{r}-1)/3-1}\notag\\
    &<3^{(n-k-1)/3}+1.\notag
\end{align}
If $a=0$, then
\begin{align}
	m(G)\leq m(G^*)=3^{(n-k-1)/3}+1,\notag
\end{align} with equality if and only if $G\cong G^{*} \in\mathcal{G}_2$.

 {\it Subcase 2.3.}  $n\equiv2$ (mod 3).  Suppose (\ref{3.1})  attains  the maximum. Then by Lemma \ref{le2.15},  one of the following holds:
 \begin{itemize}
 \item[(i)] $r=b=a+2$ and $n_{a+1}=n_b=4$;
 \item[(ii)] $b=a, r=a+1$.
 \end{itemize}

Suppose (i) holds. If $a\ne 0$,  then by    Lemma \ref{lem2.8}, we have
\begin{align*}
    m(G)&\leq   m(G^*)=4\times\prod_{i=1}^a\left(3^{n_i/3-1}+\frac{n_i}{3}+1\right)
    +2\times\prod_{i=1}^a3^{n_i/3-1}\\
    &<4\times\left[3^{(n-k-8)/3}+\frac{1}{3}(n-k-8)+1\right]  \\
    &<3^{(n-k-2)/3}.
\end{align*}If $a=0$, then $n-k=8$ and $$m(G)\leq  m(G^*)=6<3^{(n-k-2)/3}.$$

Suppose (ii) holds. If $a\neq0$, then by Lemma \ref{lem2.8}, we have
\begin{align*}
    m(G)&\leq  m(G^*)=3^{(n_{r}-2)/3-1}\prod_{i=1}^a(3^{n_i/3-1}
    +\frac{n_i}{3}+1)+2\times3^{(n_{r}-2)/3-1}\prod_{i=1}^a3^{n_i/3-1}\\
    &\leq3^{(n_{r}-2)/3-1}\times\left[3^{(n-k-n_r)/3}+\frac{1}{3}(n-k-n_r)+1\right] \\
    &<3^{(n-k-2)/3}.
\end{align*}
If $a=0$, then  $$m(G) \leq m(G^* )=3^{(n-k-2)/3},$$ with equlity if and only if $G\cong G^{*} \in\mathcal{G}_3\cup\mathcal{G}_4\cup\mathcal{G}_5$.

Now we are ready to make the conclusion of Theorem \ref{th1.1}. When $n\equiv0$ (mod 3), combining (\ref{equa0}) and Subcase 2.1,  we have
$$m(G)\le 3^{(n-k)/3}+(n-k)/3+1,$$ with equality if and only if $G\in \mathcal{G}_1$.

 When $n\equiv1$ (mod 3), combining (\ref{equa0}) and Subcase 2.2, we have $$m(G)\le 3^{(n-k-1)/3}+1,$$ with equality if and only if $G\in \mathcal{G}_2$.

  When $n\equiv2$ (mod 3), combining (\ref{equa0}) and Subcase 2.3, we have $$m(G)\le 3^{(n-k-2)/3},$$ with equality if and only if $G\in \bigcup_{i=3}^{7}\mathcal{G}_i$.

  This completes the proof of Theorem \ref{th1.1} .\hfill\qedsymbol

\subsection{Proof of Theorem \ref{th1.2}}

Given a nonnegative integer $x\geq0$, let $n_1,\ldots,n_x$ be positive integers. We denote by
\begin{equation*}
	h(x)=
	\begin{cases}
		\text{1,}&\quad\text{if $x=0$};\\
		\text{$3^{n_1/3-1}+1$,}&\quad\text{if $x=1$};\\
		\text{$\prod\limits_{i=1}^x3^{n_i/3-1}+\sum\limits_{j=1}^x\prod\limits_{1\le i\le x, i\ne j}3^{n_i/3-1}$,}&\quad\text{if $x>1$}.\\
	\end{cases}
\end{equation*}
\noindent\textit{Proof.} Suppose $G\in\mathcal{P}(n,k)$ such that $G-C_k$ consists of $r$ components $T_1,\ldots,T_r$ with $|V(T_i)|=n_i$ for $i=1,\ldots,r$. Let $v_i\in V(T_i)$ be adjacent to $u\in V(C_k)$ for $i=1,\ldots,r$.  We distinguish two cases.

{\it Case 1.} At least  one of $ n_1,\ldots,n_r $ is less than 3, say, $n_i\le 2$ for $i=r-p+1, \ldots, r$ with $p\ge 1$.

  If $n_{r-p+1}+\cdots+n_r\ge 2$, then by Lemma \ref{CLAIM3.3}, $m(G,u^0)=m(G,u^1)=0$. Therefore, $$m(G)=m(G,u^-)=m(C_k,u^-)\cdot m(G-C_k)=m(C_k,u^-)\cdot m(G-C_k-T_{r-p+1}-\cdots-T_r).$$
  Applying Corollary \ref{CLAIM3.2} and  Lemma \ref{lem2.4}, we have
\begin{equation}\label{equation1}
   m(G)\le
   \begin{cases}
       \text{$ \frac{1}{3}(k+1)\times3^{(n-k-4)/3}<x_2,$}&\text{if $n\equiv0$ (mod 3);}\\
       \text{$ \frac{1}{3}(k+1)\times3^{(n-k-2)/3},$}&\text{if $n\equiv1$ (mod 3);}\\
       \text{$ \frac{1}{3}(k+1)\times3^{(n-k-3)/3}<x_4,$}&\text{if $n\equiv2$ (mod 3).}\\
   \end{cases}
\end{equation}
Moreover,  equality in (\ref{equation1}) holds if and only if $n\equiv1$ (mod 3) and one of the following holds:
\begin{itemize}
\item[(a)] $p=1$, $n_r=2$, $T_i\cong P_3$ for $i=1,\ldots,r-1$, i.e.,   $G\in\mathcal{G}_6$;
\item[(b)] $p=2$, $n_{r-1}=n_r=1$, $T_i\cong P_3$ for $i=1,\ldots,r-2$,  i.e.,   $G\in\mathcal{G}_7$.
\end{itemize}

 Now we assume $n_{r-p+1}+\cdots+n_r=1$, i.e., $p=1$ and $n_r=1$.   Without loss of generality, we assume
$$n_i\equiv \left\{\begin{array}{ll}
 0\text{ (mod 3)},& \text{  if  }i=1,\ldots,a;\\
 1\text{ (mod 3)},& \text{  if  }i=a+1,\ldots,b;\\
 2\text{ (mod 3)},& \text{  if  }i= b+1,\ldots,r-1.
 \end{array}\right.$$

 By Corollary \ref{cl3.8}, $m(G)$ attains the maximum   only if $G\cong G^*\equiv\bigcup_{i=1}^rT_i\cup C_k\cup \{uv_1,\ldots,uv_r\}$ with $u\in V(C_k)$, $V(T_r)=\{v_r\}$   and  $T_i\in\bigcup_{j=1}^5\mathcal{T}_j(v_i)$, with an additional case that $v_i$ is an arbitrary vertex of $T_i=P_3$ when   $n_i=3$,    $i= 1,\ldots,r-1$. Note that $\overline{m}(T_r-v_r)=0$ and $\overline{m}(T_r-N(v_r))=1$. By  Corollary \ref{CLAIM3.2} and Lemma \ref{CLAIM3.3}, we have
\begin{align}
   m(G^*) &=\frac{1}{3}(k+1)\prod_{i=1}^a(3^{n_i/3-1}+\frac{n_i}{3}+1)\prod_{i=a+1}^b(3^{(n_i-1)/3
   -1}+1)\prod_{i=b+1}^{r-1}3^{(n_i-2)/3-1}\notag\\
   &~~~+\prod_{i=1}^a3^{n_i/3-1}\prod_{i=a+1}^b3^{(n_i-1)/3-1}\prod_{i=b+1}^{r-1}3^{(n_i-2)/3-1}\equiv f.\label{3.4}
\end{align}

We distinguish three subcases.

{\it Subcase 1.1.} $n\equiv0$ (mod 3). Suppose  $f$ attains the maximum.  Then by Lemma \ref{le2.15}, we have $a=r-1$, i.e., $n_i\equiv 0$ (mod 3) for $i=1,\ldots, r-1$ and $n_1+\cdots+n_{r-1}=n-k-1\equiv 0$ (mod 3). Applying Lemma \ref{lem2.9} (i) with $m=n-k-1, l=k$ and $i=r-1$, we have
\begin{align}
    f&=\frac{1}{3}(k+1)\prod\limits_{i=1}^{r-1}(3^{n_i/3-1}+\frac{n_i}{3}+1)
    +\prod\limits_{i=1}^{r-1}3^{n_i/3-1}\notag\\
    &\leq (k+1)\times3^{(n-k-1)/3-1}+1=x_2,\notag
\end{align}with equality if and only if $n_1=\cdots=n_{r-1}=3$, i.e., $G\in\mathcal{G}_9$.

{\it Subcase 1.2.} $n\equiv1$ (mod 3). Similarly,  suppose  $f$ attains the maximum.  Then by Lemma \ref{le2.15}, we have $b=a+1=r-1$, i.e., $n_{r-1}\equiv 1$ (mod 3), $n_i\equiv 0$ (mod 3) for $i=1,\ldots, r-2$ and $n_1+\cdots+n_{r-1}=n-k-1\equiv 1$ (mod 3).

Note that for a positive integer $t\equiv0$ (mod 3), we have $3^{t/3-1}+t/3+1\leq 3^{t/3}$. If $a\neq0$, then
\begin{align*}
m(G)&\leq f
=\frac{1}{3}(k+1)(3^{(n_{r-1}-1)/3-1}+1)\prod_{i=1}^{r-2}(3^{n_i/3-1}
     +\frac{n_i}{3}+1)+3^{(n_{r-1}-1)/3-1}\prod_{i=1}^{r-2}3^{n_i/3-1}\\
     &\le 3^{(n_{r-1}-1)/{3}-1}\left[\frac{1}{3}(k+1)\prod_{i=1}^{r-2}(3^{n_i/3-1}
     +\frac{n_i}{3}+1)+ \prod_{i=1}^{r-2}3^{n_i/3-1}  \right]\\
     &\quad+\frac{1}{3}(k+1)\times3^{(n-k-1-n_{r-1})/3} \tag{by Lemma \ref{lem2.9} (i)}\\
    &\leq3^{(n_{r-1}-1)/3-1}\times\left[\frac{1}{3}(k+1)\times3^{(n-k-1-n_{r-1})/3}+1\right]
    +\frac{1}{3}(k+1)\times3^{(n-k-1-n_{r-1})/3} \\
    &<(k+1)\times3^{(n-k-2)/3-1}.
\end{align*}

If $a=0$, then
\begin{align*}
    m(G)\leq f&= \frac{1}{3}(k+1)(3^{(n-k-2)/3-1}+1)+3^{(n-k-2)/3-1} <(k+1)\times3^{(n-k-2)/3-1}.
\end{align*}

{\it Subcase 1.3.}  $n\equiv2$ (mod 3). Suppose  $f$ attains the maximum.  Then by Lemma \ref{le2.15}, one of the following holds:
 \begin{itemize}
 \item[(i)] $b=a+2= r-1$ and $n_{a+1}=n_b=4$, i.e., $n_{r-2}=n_{r-1}=4$ and $n_i\equiv 0$ (mod 3) for $i=1,\ldots, r-3$;
 \item[(ii)]  $a=b=r-2$, i.e., $n_{r-1}\equiv 2$ (mod 3), $n_i\equiv 0$ (mod 3) for $i=1,\ldots, r-2$ and $n_1+\cdots+n_{r-1}=n-k-1\equiv 2$ (mod 3).
 \end{itemize}

 Suppose (i) holds.  If $a\neq0$, then by (\ref{3.4}) and applying  Lemma \ref{lem2.9} (i) with $m=n-8$ and $l=k$, we have
\begin{align*}
    m(G)&\leq f =\frac{4}{3}(k+1)\prod_{i=1}^{r-3}(3^{n_i/3-1}+\frac{n_i}{3}+1)+\prod_{i=1}^{r-3}3^{n_i/3-1}\\
    &<4\left[\frac{1}{3}(k+1)\times3^{(n-k-9)/3}+1\right]<x_4.
\end{align*}
If $a=0$, then $n-k=9$,  $x_3=(31k+7)/3$ and $x_4=(29k+35)/3$. It follows that
 $$m(G)\leq f=\frac{4}{3}(k+1)+1<  x_4.$$

 Now suppose (ii) holds. If $a\neq0$, then by (\ref{3.4}) and applying  Lemma \ref{lem2.9} (i) with $m=n-n_{r-1}$ and $l=k$, we have
\begin{align*}
      m(G)&\leq f
=\frac{1}{3}(k+1)\times3^{(n_{r-1}-2)/3-1}\prod_{i=1}^{r-2}
     (3^{n_i/3-1}+\frac{n_i}{3}+1)+3^{(n_{r-1}-2)/3-1}\prod_{i=1}^{r-2}3^{n_i/3-1}\\
     &\leq3^{(n_{r-1}-2)/3-1}\times\left[\frac{1}{3}(k+1)\times3^{(n-k-1-n_{r-1})/3}+1\right] <x_4.
\end{align*}
If $a=0$, then $$m(G)\leq f=\frac{1}{3}(k+1)\times3^{(n-k-3)/3-1}+3^{(n-k-3)/3-1}<x_4.$$

{\it Case 2.}  $n_i\geq3$ for $i=1,\ldots,r$. Similarly as in Case 2 in the proof of Theorem \ref{th1.1}, we suppose (\ref{eqn1}) holds.
By Theorem \ref{cl3.4}, $m(G)$ attains the maximum   only if $G\cong G^*\equiv\bigcup_{i=1}^rT_i\cup C_k\cup \{uv_1,\ldots,uv_r\}$ with $u\in V(C_k)$ and $T_i$ satisfying (\ref{eqn2}). Then  by Corollary \ref{CLAIM3.2} and Lemma \ref{CLAIM3.3}, we have
\begin{align}\label{3.5}
  m(G^*)&=\frac{1}{3}(k+1)\prod_{i=1}^a(3^{n_i/3-1}
  +\frac{n_i}{3}+1)\prod_{i=a+1}^b(3^{(n_i-1)/3-1}+1)\prod_{i=b+1}^r3^{(n_i-2)/3-1}\notag\\
   &~~~+\frac{2}{3}(k-2)\prod_{i=1}^a3^{n_i/3-1}\prod_{i=a+1}^b3^{(n_i-1)/3-1}\prod_{i=b+1}^r3^{(n_i-2)/3-1}\notag\\
   &~~~+h(a)\prod_{i=a+1}^b3^{(n_i-1)/3-1}\prod_{i=b+1}^r3^{(n_i-2)/3-1}.
\end{align}
We distinguish three subcases.

{\it Subcase 2.1.} $n\equiv0$ (mod 3).  Similarly, suppose (\ref{3.5}) attains the maximum.  By Lemma \ref{le2.15}, we have $b=a+1=r$, i.e., $n_r\equiv1$ (mod 3), $n_i\equiv0$ (mod 3) for $i=1,\ldots,r-1$ and $n_1+\cdots +n_r=n-k\equiv1$ (mod 3).

 If $a=0$, i.e., $n_r=n-k$, then by (\ref{3.5}), we have
\begin{align}
     m(G)&\leq m(G^*)=\frac{1}{3}(k+1)(3^{(n_r-1)/3-1}+1)+\frac{2}{3}(k-2)\times3^{(n_r-1)/3-1}+3^{(n_r-1)/3-1}\notag\\
     &= k\times3^{(n-k-1)/3-1}+\frac{1}{3}(k+1)=x_1,\notag
\end{align}with equality if and only if $G\in\mathcal{G}_2$.

Now  suppose $a\ge 1$. Then we have $5\le k\le n-7$, which implies $n\ge 12$.
Let $$\varphi(x)=\frac{1}{3}(k+1)\times3^{(n-k-1)/3-1}+\frac{1}{3}(n+k-1-x)\times3^{(x-1)/3-1}
+\frac{1}{3}(k+1)\times3^{(n-k-x)/3}. $$ One may check that $\varphi'(4)<0$, $\varphi'(n-k)>0$ and  $\varphi''(x)>0$ for $x\in[4\text{, }n-k]$. Therefore,   $\varphi(x)$ firstly decreases and then increases when $x\in[4\text{, }n-k]$. Notice that $$\varphi(4)=2(k+1)\cdot3^{(n-k-1)/3-2}+(n+k-5)/3\quad \text{and}\quad \varphi(n-k)=x_1.$$ Let$$\phi(k)=x_1-\varphi(4)=\frac{1}{3}(k-2)\times3^{(n-k-1)/3-1}+2-\frac{n}{3}.$$Since $\phi(k)$ is strictly decreasing when $5\le k\le n-7$, we have $\phi(k)\ge\phi(n-7)>0$, which implies $x_1>\varphi(4)$.

 If $a>1$, then by (\ref{3.5}), we have
\begin{align}
    m(G)&\leq m(G^*)\notag\\ &=\frac{1}{3}(k+1)(3^{(n_r-1)/3-1}+1)\prod\limits_{i=1}^{r-1}(3^{n_i/3-1}+\frac{n_i}{3}+1) +\frac{2}{3}(k-2)\times3^{(n_r-1)/3-1}\prod\limits_{i=1}^{r-1}3^{n_i/3-1}\notag\\
    &~~~+3^{(n_r-1)/3-1}\Big(\prod\limits_{i=1}^{r-1}3^{n_i/3-1}+\sum\limits_{j=1}^{r-1}\prod\limits_{1\le i\le r-1,i\neq j}3^{n_i/3-1}\Big)\notag\\
    &\leq3^{(n_r-1)/3-1}\left[\frac{1}{3}(k+1)\times3^{(n-k-n_r)/3}+\frac{1}{3}(n+k-1-n_r)\right]\notag\\
    &~~~+\frac{1}{3}(k+1)\prod\limits_{i=1}^{r-1}(3^{n_i/3-1}+\frac{n_i}{3}+1)\tag{by Lemma \ref{lem2.9} (ii)}\\
    &\leq\frac{1}{3}(k+1)\times3^{(n-k-1)/3-1}+\frac{1}{3}(n+k-1-n_r)\times3^{(n_r-1)/3-1}\notag +\frac{1}{3}(k+1)\times3^{(n-k-n_r)/3}\notag\\
    &\le\max\{\varphi(4),~\varphi(n-k-6)\}<x_1.\notag
\end{align}
If $a=1$, then by (\ref{3.5}), we have
\begin{align}
     m(G)&\leq m(G^*)\notag\\
     &=\frac{1}{3}(k+1)(3^{n_1/3-1}+\frac{n_1}{3}+1)(3^{(n_r-1)/3-1}+1)\notag\\
     &~~~+\frac{2}{3}(k-2)\times3^{n_1/3-1}\times3^{(n_r-1)/3-1}+(3^{n_1/3-1}+1)\times3^{(n_r-1)/3-1}\notag\\
     &<\frac{1}{3}(k+1)\times3^{n_1/3}\times(3^{(n_r-1)/3-1}+1)+\frac{2}{3}(k-2)\times3^{(n-k-1)/3-2}+3^{(n-k-1)/3-1}\notag\\
     &<\frac{1}{3}(k+1)\times3^{n_1/3}\times2\times3^{(n_r-1)/3-1}+\frac{2}{9}(k-2)\times3^{(n-k-1)/3-1}+3^{(n-k-1)/3-1}\notag\\
     &<x_2.\notag
\end{align}

{\it Subcase 2.2.} $n\equiv1$ (mod 3). Suppose (\ref{3.5}) attains the maximum.  By Lemma \ref{le2.15},   one of the following holds:
 \begin{itemize}
	\item[(i)] $b=a+2= r$ and $n_{a+1}=n_b=4$, i.e., $n_{r-1}=n_{r}=4$ and $n_i\equiv 0$ (mod 3) for $i=1,\ldots, r-2$;
	\item[(ii)]  $a=b=r-1$, i.e., $n_{r}\equiv 2$ (mod 3), $n_i\equiv 0$ (mod 3) for $i=1,\ldots, r-1$ and $n_1+\cdots+n_{r}=n-k\equiv 2$ (mod 3).
\end{itemize}

Suppose (i) holds. If $a>1$, then $5\le k\le n-14$. By (\ref{3.5}), we have
\begin{align}
	m(G)&\leq m(G^*)\notag\\
&=\frac{4}{3}(k+1)\prod\limits_{i=1}^{r-2}(3^{n_i/3-1}+\frac{n_i}{3}+1)+\frac{1}{3}(2k-1)\prod\limits_{i=1}^{r-2}3^{n_i/3-1}+\sum\limits_{j=1}^{r-2}\prod\limits_{1\le i\le r-2,i\neq j}3^{n_i/3-1}\notag\\
	&\leq\frac{1}{3}(k+1)\times3^{(n-k-n_{r-1}-n_r)/3}+\frac{1}{3}(n+k-1-n_{r-1}-n_r)\notag
	 +(k+1)\prod\limits_{i=1}^{r-2}(3^{n_i/3-1}+\frac{n_i}{3}+1)\tag{applying Lemma \ref{lem2.9} (ii)  with $m=n-n_{r-1}-n_r$}\notag\\	&\leq\frac{1}{3}(k+1)\times3^{(n-k-8)/3}+\frac{1}{3}(n+k-9)+(k+1)\times3^{(n-k-8)/3}\equiv\beta_1\notag\\
	&<(k+1)\times3^{(n-k-2)/3-1}\equiv\alpha_1.\notag
\end{align}
The last inequality  holds because $\varphi(k)=\alpha_1-\beta_1=\frac{5}{9}(k+1)\times3^{(n-k-2)/3-1}-\frac{1}{3}(n+k-9)$ is strictly decreasing when $5\le k\le n-8$, which leads to $\alpha_1-\beta_1>\varphi(n-8)>0.$

 If $a=1$, i.e., $n_1=n-k-8$, then by (\ref{3.5}), we have
\begin{align*}
	m(G)\leq m(G^*)&=\frac{4}{3}(k+1)(3^{n_1/3-1}+\frac{n_1}{3}+1)+\frac{2}{3}(k-2)\times3^{n_1/3-1}+3^{n_1/3-1}+1\\
	&<\frac{4}{3}(k+1)\times3^{n_1/3}+\frac{2}{3}(k-2)\times3^{n_1/3}+3^{n_1/3}\\
	&<(k+1)\times3^{(n-k-2)/3-1}.
\end{align*}If $a=0$, i.e., $n =k+8$, then by (\ref{3.5}), we have $$m(G)\leq m(G^*)=2k+1<(k+1)\times3^{(n-k-2)/3-1}.$$

Suppose (ii) holds. If $a>1$, then by (\ref{3.5}), we have
\begin{align*}
    m(G)&\leq m(G^*)\notag\\
    &= \frac{k+1}{3}\times3^{(n_r-2)/3-1}\prod\limits_{i=1}^{r-1}(3^{n_i/3-1}+\frac{n_i}{3}+1)
    +\frac{2}{3}(k-2)\times3^{(n_r-2)/3-1}\prod\limits_{i=1}^{r-1}3^{n_i/3-1}\notag\\
    &~~~+3^{(n_r-2)/3-1}\Big(\prod\limits_{i=1}^{r-1}3^{n_i/3-1}+\sum\limits_{j=1}^{r-1}\prod\limits_{1\le i\le r-1,i\neq j}3^{n_i/3-1}\Big)\\
    &\leq3^{(n_r-2)/3-1}\left[\frac{1}{3}(k+1)\times3^{(n-k-n_r)/3}+\frac{1}{3}(n+k-1-n_r)\right]\tag{by Lemma \ref{lem2.9} (ii)}\\
    &<(k+1)\times3^{(n-k-2)/3-1}.
\end{align*}If $a=1$, i.e., $n_1+n_r=n-k$, then by (\ref{3.5}), we have
\begin{align*}
    m(G)&\leq m(G^*)\\
    &=\frac{1}{3}(k+1)\times3^{(n_r-2)/3-1}\times(3^{n_1/3-1}+\frac{n_1}{3}+1)\\
    &~~~+\frac{2}{3}(k-2)\times3^{(n_r-2)/3-1}\times3^{n_1/3-1}+3^{(n_r-2)/3-1}\times(3^{n_1/3-1}+1)\\
    &<\frac{1}{3}(k+1)\times3^{(n-k-2)/3-1}+\frac{2}{3}(k-2)\times3^{(n-k-2)/3-1}+3^{(n-k-2)/3-1}\\
    &<(k+1)\times3^{(n-k-2)/3-1}.
\end{align*}If $a=0$, then by (\ref{3.5}), we have
 $$m(G)\leq m(G^*)= k\times3^{(n-k-2)/3-1}<(k+1)\times3^{(n-k-2)/3-1}.$$

{\it Subcase 2.3.} $n\equiv2$ (mod 3). Suppose (\ref{3.5}) attains the maximum.  By Lemma \ref{le2.15}, we have  $a=r$, i.e., $n_i\equiv0$ (mod 3) for $i=1,\ldots,r$.

 If $a>1$, then by (\ref{3.5}) and   Lemma \ref{lem2.9} (ii), we have
\begin{align}
    m(G)\leq m(G^*)&=\frac{1}{3}(k+1)\prod\limits_{i=1}^{r}(3^{n_i/3-1}+\frac{n_i}{3}+1)+\frac{1}{3}(2k-1)\prod\limits_{i=1}^{r}3^{n_i/3-1}+\sum\limits_{j=1}^{r}\prod\limits_{i\neq j}3^{n_i/3-1}\notag\\
    &\leq (k+1)\times3^{(n-k)/3-1}+\frac{1}{3}(n+k-1)=x_4,\notag
\end{align}with equality  if and only if $n_1=\cdots= n_r=3$, i.e., $G\cong G^*\in\mathcal{G}_8$.

 If $a=1$, then according to (\ref{3.5}), we have
\begin{align}
     m(G)\leq m(G^*)= k\times3^{(n-k)/3-1}+\frac{1}{3}(k+1)\left[\frac{1}{3}(n-k)+1\right]+1=x_3,\notag
\end{align}with equality if and only if $G\cong G^*\in\mathcal{G}_1$.

Now we are ready to make the conclusion of Theorem \ref{th1.2}. When $n\equiv0$ (mod 3), combining (\ref{equation1}), Subcase 1.1 and Subcase 2.1, we have
 \begin{equation*}\label{eqh0421}
 m(G)\le\max\{x_1,x_2\},
  \end{equation*}
with equality if and only if $G\in\mathcal{H}_i$ and $x_i=\max\{x_1,x_2\}$ with $i\in\{1,2\}$, where $\mathcal{H}_1=\mathcal{G}_2\text{ and }\mathcal{H}_2=\mathcal{G}_9.$

When $n\equiv1$ (mod 3),  combining (\ref{equation1}), Subcase 1.2 and Subcase 2.2, we have
$$m(G)\le \frac{1}{3}(k+1)\times3^{(n-k-2)/3},$$ with equality if and only if $G\in\mathcal{G}_6\cup\mathcal{G}_7$.

 When $n\equiv2$ (mod 3), combining (\ref{equation1}), Subcase 1.3 and Subcase 2.3, we have \begin{equation*}\label{eqh0422}
 m(G)\le\max\{x_3, x_4\},
 \end{equation*}
 with equality   if and only if $G\in\mathcal{H}_i$ and $x_i=\max\{x_3,x_4\}$ with $i\in\{3,4\}$, where $\mathcal{H}_3=\mathcal{G}_1,~\mathcal{H}_4=\mathcal{G}_8.$

 This completes the proof of Theorem \ref{th1.2}. \hfill\qedsymbol

\subsection{Proof of Theorem \ref{th1.3}}

We apply similar arguments as in the proof of   Theorem \ref{th1.2} to prove Theorem \ref{th1.3}.

\noindent\textit{Proof.}
 Suppose $G\in\mathcal{P}(n,k)$ such that $G-C_k$ consists of $r$ components $T_1,\ldots,T_r$ with $|V(T_i)|=n_i$ for $i=1,\ldots,r$. Let $v_i\in V(T_i)$ be adjacent to $u\in V(C_k)$ for $i=1,\ldots,r$.  We distinguish two cases.

 {\it Case 1.} At least  one of $ n_1,\ldots,n_r $  is less than 3, say, $n_i\le 2$ for $i=r-p+1, \ldots, r$ with $p\ge 1$.

 If $n_{r-p+1}+\cdots+n_r\ge 2$, then by Lemma \ref{CLAIM3.3}, $m(G,u^0)=m(G,u^1)=0$. Therefore, $$m(G)=m(G,u^-)=m(C_k,u^-)\cdot m(G-C_k)=m(C_k,u^-)\cdot m(G-C_k-T_{r-p+1}-\cdots-T_r).$$
 Applying Corollary \ref{CLAIM3.2} and  Lemma \ref{lem2.4}, we have
 \begin{equation}\label{equation010}
 	m(G)\le
 	\begin{cases}
 		\text{$ \frac{1}{18}(k+2)(k+5)\times3^{(n-k-2)/3},$}&\text{if $n\equiv0$ (mod 3);}\\
 		\text{$ \frac{1}{18}(k+2)(k+5)\times3^{(n-k-3)/3}<y_1,$}&\text{if $n\equiv1$ (mod 3);}\\
 		\text{$ \frac{1}{18}(k+2)(k+5)\times3^{(n-k-4)/3}<y_4,$}&\text{if $n\equiv2$ (mod 3).}\\
 	\end{cases}
 \end{equation}
 Moreover,  equality in (\ref{equation010}) holds if and only if $n\equiv0$ (mod 3) and one of the following holds:
 \begin{itemize}
 	\item[(a)] $p=1$, $n_r=2$, $T_i\cong P_3$ for $i=1,\ldots,r-1$, i.e.,   $G\in\mathcal{G}_6$;
 	\item[(b)] $p=2$, $n_{r-1}=n_r=1$, $T_i\cong P_3$ for $i=1,\ldots,r-2$,  i.e.,   $G\in\mathcal{G}_7$.
 \end{itemize}

 Now we assume $n_{r-p+1}+\cdots+n_r=1$, i.e., $p=1$ and $n_r=1$.   Without loss of generality, we assume
 $$n_i\equiv \left\{\begin{array}{ll}
 	0\text{ (mod 3)},& \text{  if  }i=1,\ldots,a;\\
 	1\text{ (mod 3)},& \text{  if  }i=a+1,\ldots,b;\\
 	2\text{ (mod 3)},& \text{  if  }i= b+1,\ldots,r-1.
 \end{array}\right.$$

  By Corollary \ref{cl3.8}, $m(G)$ attains the maximum  only if $G\cong G^*\equiv\bigcup_{i=1}^rT_i\cup C_k\cup \{uv_1,\ldots,uv_r\}$ with $u\in V(C_k)$, $V(T_r)=\{v_r\}$   and  $T_i\in\bigcup_{j=1}^5\mathcal{T}_j(v_i)$, with an additional case that $v_i$ is an arbitrary vertex of $T_i=P_3$ when   $n_i=3$,    $i= 1,\ldots,r-1$.  Note that $\overline{m}(T_r-v_r)=0$ and $\overline{m}(T_r-N(v_r))=1$. By  Corollary \ref{CLAIM3.2} and Lemma \ref{CLAIM3.3}, we have
 \begin{align}
 	m(G^{*})&=\frac{1}{18}(k+2)(k+5)\prod\limits_{i=1}^{a}(3^{n_i/3-1}+\frac{n_i}{3}+1)\prod\limits_{i=a+1}^{b}(3^{(n_i-1)/3-1}+1)\prod\limits_{i=b+1}^{r-1}3^{(n_i-2)/3-1}\notag\\
 	&~~~+\frac{1}{3}(k-1)\prod\limits_{i=1}^{a}3^{n_i/3-1}\prod\limits_{i=a+1}^{b}3^{(n_i-1)/3-1}\prod\limits_{i=b+1}^{r-1}3^{(n_i-2)/3-1}.\label{3.44}
 \end{align}

 We distinguish three subcases.

 {\it Subcase 1.1.} $n\equiv0$ (mod 3). Suppose (\ref{3.44}) attains the maximum. By Lemma \ref{le2.15}, we have $b=a+1=r-1$, i.e., $n_{r-1}\equiv 1$ (mod 3), $n_i\equiv 0$ (mod 3) for $i=1,\ldots, r-2$ and $n_1+\cdots+n_{r-1}=n-k-1\equiv 1$ (mod 3).

 If $a\neq0$, then we have
 \begin{align*}
 	m(G) &\leq m(G^*)\\  &=\frac{1}{18}(k+2)(k+5)(3^{(n_{r-1}-1)/3-1}+1)\prod\limits_{i=1}^{r-2}(3^{n_i/3-1}+\frac{n_i}{3}+1)\\
 	&~~~+\frac{1}{3}(k-1)\times3^{(n_{r-1}-1)/3-1}\prod\limits_{i=1}^{r-2}3^{n_i/3-1}\\
 	&\leq\left[\frac{1}{18}(k+2)(k+5)\times3^{(n-k-1-n_{r-1})/3}+\frac{1}{3}(k-1)\right]\times3^{(n_{r-1}-1)/3-1}\\
 	&~~~+\frac{1}{18}(k+2)(k+5)\prod\limits_{i=1}^{r-2}(3^{n_i/3-1}+\frac{n_i}{3}+1)\tag{by Lemma \ref{lem2.10} $(i)$}\\
 	&<\frac{1}{18}(k+2)(k+5)\times3^{(n-k-2)/3}.
 \end{align*}
 If $a=0$, i.e., $r=2$ and $n_{r-1}=n_1=n-k-1$, then
 \begin{align*}
 	m(G)&\leq m(G^*)=\frac{1}{18}(k+2)(k+5)(3^{(n_1-1)/3-1}+1)+\frac{1}{3}(k-1)\times3^{(n_1-1)/3-1}\\
 	&<\frac{1}{18}(k+2)(k+5)\times3^{(n-k-2)/3}.
 \end{align*}

 {\it Subcase 1.2.} $n\equiv1$ (mod 3). Similarly, suppose (\ref{3.44}) attains the maximum. By Lemma \ref{le2.15},  one of the following holds:
 \begin{itemize}
 	\item[(i)] $b=a+2= r-1$ and $n_{a+1}=n_b=4$, i.e., $n_{r-2}=n_{r-1}=4$ and $n_i\equiv 0$ (mod 3) for $i=1,\ldots, r-3$.
 	\item[(ii)]  $a=b=r-2$, i.e., $n_{r-1}\equiv 2$ (mod 3), $n_i\equiv 0$ (mod 3) for $i=1,\ldots, r-2$ and $n_1+\cdots+n_{r-1}=n-k-1\equiv 2$ (mod 3).
 \end{itemize}

 Suppose (i) holds. If $a\neq0$, then by \eqref{3.44}, we have
 \begin{align*}
 	m(G)&\leq m(G^*)=\frac{4}{18}(k+2)(k+5)\prod\limits_{i=1}^{r-3}(3^{n_i/3-1}+\frac{n_i}{3}+1)+\frac{1}{3}(k-1)\prod\limits_{i=1}^{r-3}3^{n_i/3-1}\\
 	&<4\left[\frac{1}{18}(k+2)(k+5)\times3^{(n-k-9)/3}+\frac{1}{3}(k-1)\right]\tag{by Lemma \ref{lem2.10} (i)}\\
 	&<y_1.
 \end{align*}
 If $a=0$, i.e., $n=k+9$, then $y_1=(29k^2+215k+242)/18,~y_2=(31k^2+169k+34)/18$. Thus,
 \begin{align*}
 	m(G)\leq m(G^*)=\frac{2}{9}(k+2)(k+5)+\frac{1}{3}(k-1)< y_1.
 \end{align*}

 Suppose (ii) holds. If $a\neq0$, then we have
 \begin{align*}
 	m(G)&\leq m(G^*)\\&=\frac{1}{18}(k+2)(k+5)\times3^{(n_{r-1}-2)/3-1}
 \prod\limits_{i=1}^{r-2}(3^{n_i/3-1}+\frac{n_i}{3}+1)\\
 	&~~~+\frac{1}{3}(k-1)\times3^{(n_{r-1}-2)/3-1}\prod\limits_{i=1}^{r-2}3^{n_i/3-1}\\
 	&\leq\left[\frac{1}{18}(k+2)(k+5)\times3^{(n-k-1-n_{r-1})/3}+\frac{1}{3}(k-1)\right]\times3^{(n_{r-1}-2)/3-1}\tag{by Lemma \ref{lem2.10} (i)}\\
 	&<y_1.
 \end{align*}
 If $a=0$, i.e., $r=2$ and $n_{r-1}=n_1=n-k-1$, then
 \begin{align*}
 	m(G)\leq m(G^*)&= \frac{1}{18}(k+2)(k+5)\times3^{(n_1-2)/3-1}+\frac{1}{3}(k-1)\times3^{(n_1-2)/3-1}<y_1.
 \end{align*}

 {\it Subcase 1.3.} $n\equiv2$ (mod 3). Suppose (\ref{3.44}) attains the maximum. By Lemma \ref{le2.15}, we have  $a=r-1$, i.e., $n_i\equiv0$ (mod 3) for $i=1,\ldots,r-1$. By \eqref{3.44} and Lemma \ref{lem2.10} (i), we have
 \begin{align}
 	m(G)&\leq m(G^*)=\frac{1}{18}(k+2)(k+5)\prod\limits_{i=1}^{r-1}(3^{n_i/3-1}+\frac{n_i}{3}+1)+\frac{1}{3}(k-1)\prod\limits_{i=1}^{r-1}3^{n_i/3-1}\notag\\
 	&\leq\frac{1}{18}(k+2)(k+5)\times3^{(n-k-1)/3}+\frac{1}{3}(k-1)=y_4,\notag
 \end{align} with equality if and only if $n_1=\cdots =n_{r-1}=3$, i.e., $G\cong G^*\in\mathcal{G}_9$.

 {\it Case 2.}  $n_i\geq3$ for $i=1,\ldots,r$. Similarly as in Case 2 of the proof of Theorem \ref{th1.1}, we suppose (\ref{eqn1}) holds.
 By Theorem \ref{cl3.4}, $m(G)$ attains the maximum  only if $G\cong G^*\equiv\bigcup_{i=1}^rT_i\cup C_k\cup \{uv_1,\ldots,uv_r\}$ with $u\in V(C_k)$ and $T_i$ satisfying (\ref{eqn2}).
 Moreover, by Lemma \ref{CLAIM3.3} and Corollary \ref{CLAIM3.2},   we have
 \begin{align}\label{3.55}
 	m(G^{*})&=\frac{1}{18}(k+2)(k+5)\prod\limits_{i=1}^{a}(3^{n_i/3-1}+\frac{n_i}{3}+1)\prod\limits_{i=a+1}^{b}(3^{(n_i-1)/3-1}+1)\prod\limits_{i=b+1}^{r}3^{(n_i-2)/3-1}\notag\\
 	&~~~+\frac{1}{9}(k-1)(k+2)\prod\limits_{i=1}^{a}3^{n_i/3-1}\prod\limits_{i=a+1}^{b}3^{(n_i-1)/3-1}\prod\limits_{i=b+1}^{r}3^{(n_i-2)/3-1}\notag\\
 	&~~~+\frac{1}{3}(k-1)h(a)\prod\limits_{i=a+1}^{b}3^{(n_i-1)/3-1}\prod\limits_{i=b+1}^{r}3^{(n_i-2)/3-1}.
 \end{align}

 {\it Subcase 2.1.} $n\equiv0$ (mod 3). Similarly as in Case 1, suppose (\ref{3.55}) attains the maximum. Then by Lemma \ref{le2.15},   one of the following holds:
 \begin{itemize}
 	\item[(i)] $b=a+2= r$ and $n_{a+1}=n_b=4$, i.e., $n_{r-1}=n_{r}=4$ and $n_i\equiv 0$ (mod 3) for $i=1,\ldots, r-2$.
 	\item[(ii)]  $a=b=r-1$, i.e., $n_{r}\equiv 2$ (mod 3), $n_i\equiv 0$ (mod 3) for $i=1,\ldots, r-1$ and $n_1+\cdots+n_{r}=n-k\equiv 2$ (mod 3).
 \end{itemize}

 Suppose (i) holds. If $a>1$, then $4\le k\le n-14$. By (\ref{3.55}),  we have
 \begin{align}
 	m(G)&\leq m(G^*)\notag \\
 &=\frac{4}{18}(k+2)(k+5)\prod\limits_{i=1}^{r-2}(3^{n_i/3-1}+\frac{n_i}{3}+1)+\frac{1}{9}(k-1)(k+2)\prod\limits_{i=1}^{r-2}3^{n_i/3-1}\notag\\
 	&~~~+\frac{1}{3}(k-1)\Big(\prod\limits_{i=1}^{r-2}3^{n_i/3-1}+\sum_{j=1}^{r-2}\prod\limits_{1\le i\le r-2,i\neq j}3^{n_i/3-1}\Big)\notag\\
 	&\leq\frac{1}{18}(k+2)(k+5)\times3^{(n-k-8)/3}+\frac{1}{9}(k-1)(n-3)
 	+\frac{1}{6}(k+2)(k+5)\prod\limits_{i=1}^{r-2}(3^{n_i/3-1}+\frac{n_i}{3}+1)\tag{applying Lemma \ref{lem2.10} $(ii)$ with $m=n-8$ and $l=k$}\notag\\
 	&\leq\frac{4}{18}(k+2)(k+5)\times3^{(n-k-8)/3}+\frac{1}{9}(k-1)(n-3)\equiv\beta_2\notag\\
 	&<\frac{1}{18}(k+2)(k+5)\times3^{(n-k-2)/3}\equiv\alpha_2.\notag
 \end{align}The last equality  holds because $\varphi(k)=\alpha_2-\beta_2=\frac{5}{18}(k+2)(k+5)\times3^{(n-k-2)/3-2}-\frac{1}{9}(k-1)(n-3)$ is strictly decreasing when $4\le k\le n-8$, which leads to $\alpha_2-\beta_2>\varphi(n-8)>0$.

 If $a=1$, i.e., $r=3$ and $n_1=n-k-8$, then we have
 \begin{align*}
 	m(G)&\leq m(G^*)\\
 &=\frac{4}{18}(k+2)(k+5)(3^{n_1/3-1}+\frac{n_1}{3}+1)\\
 	&~~~+\frac{1}{9}(k-1)(k+2)\times3^{n_1/3-1}+\frac{1}{3}(k-1)\Big(3^{n_1/3-1}+1\Big)\\
 	&<\left[\frac{4}{18}(k+2)(k+5)+\frac{1}{9}(k-1)(k+2)+\frac{1}{3}(k-1)\right]\times3^{(n-k-8)/3}\\
 	&<\frac{1}{18}(k+2)(k+5)\times3^{(n-k-2)/3}.
 \end{align*}If $a=0$, then $n=k+8$ and
 \begin{align*}
 	m(G)&\leq m(G^*)=\frac{4}{18}(k+2)(k+5)+\frac{1}{9}(k-1)(k+2)+\frac{1}{3}(k-1)\\
 	&<\frac{1}{18}(k+2)(k+5)\times3^{(n-k-2)/3}.
 \end{align*}

 Suppose (ii) holds. If $a>1$, then
 \begin{align*}
 	m(G)&\leq m(G^*)\\
 &= \frac{1}{18}(k+2)(k+5)\times3^{(n_r-2)/3-1}\prod\limits_{i=1}^{r-1}(3^{n_i/3-1}+\frac{n_i}{3}+1)\\
 	&~~~+\frac{1}{9}(k-1)(k+2)\times3^{(n_r-2)/3-1}\prod\limits_{i=1}^{r-1}3^{n_i/3-1}\\
 	&~~~+\frac{1}{3}(k-1)\times3^{(n_r-2)/3-1}\Big(\prod\limits_{i=1}^{r-1}3^{n_i/3-1}+\sum_{j=1}^{r-1}\prod\limits_{1\le i\le r-2,i\neq j}3^{n_i/3-1}\Big)\\
 	&\leq3^{(n_r-2)/3-1}\left[\frac{1}{18}(k+2)(k+5)\times3^{(n-k-n_r)/3}+\frac{1}{9}(k-1)(n+5-n_r)\right]\tag{applying Lemma \ref{lem2.10} $(ii)$ with $m=n-n_r$ and $l=k$}\\
 	&<\frac{1}{18}(k+2)(k+5)\times3^{(n-k-2)/3}.
 \end{align*}If $a=1$, i.e., $r=2$ and $n_1+n_r=n-k$, then we have
 \begin{align*}
 	m(G)&\leq m(G^*)\\&=\frac{1}{18}(k+2)(k+5)\times3^{(n_r-2)/3-1}\Big(3^{n_1/3-1}+\frac{n_1}{3}+1\Big)\\
 	&~~~+\frac{1}{9}(k-1)(k+2)\times3^{(n_r-2)/3-1}\times3^{n_1/3-1}+\frac{1}{3}(k-1)\times3^{(n_r-2)/3-1}\Big(3^{n_1/3-1}+1\Big)\\
 	&<\frac{1}{18}(k+2)(k+5)\times3^{(n-k-2)/3-1}+\frac{1}{9}(k-1)(k+2)\times3^{(n-k-2)/3-2}
 	 +\frac{1}{3}(k-1)\times3^{(n-k-2)/3-1}\\
 	&<\frac{1}{18}(k+2)(k+5)\times3^{(n-k-2)/3}.
 \end{align*}If $a=0$, then by (\ref{3.55}) we have
 \begin{align*}
 	m(G)\leq m(G^*)=\frac{1}{6}k(k+5)\times3^{(n-k-2)/3-1}<\frac{1}{18}(k+2)(k+5)\times3^{(n-k-2)/3}.
 \end{align*}

 {\it Subcase 2.2.} $n\equiv1$ (mod 3). Similarly, suppose (\ref{3.55}) attains the maximum. By Lemma \ref{le2.15}, we have  $a=r$, i.e., $n_i\equiv0$ (mod 3) for $i=1,\ldots,r$.

 If $a\neq1$, by Lemma \ref{lem2.10} (ii) and (\ref{3.55}), we have
 \begin{align}
 	m(G)&\leq m(G^*)\\
 &=\frac{1}{18}(k+2)(k+5)\prod\limits_{i=1}^{r}(3^{n_i/3-1}+\frac{n_i}{3}+1)\notag\\
 	&~~~+\frac{1}{9}(k-1)(k+2)\prod\limits_{i=1}^{r}3^{n_i/3-1}+\frac{1}{3}(k-1)\Big(\prod\limits_{i=1}^{r}3^{n_i/3-1}+\sum_{j=1}^{r}\prod\limits_{i\neq j}3^{n_i/3-1}\Big)\notag\\
 	&\leq\frac{1}{18}(k+2)(k+5)\times3^{(n-k)/3}+\frac{1}{9}(n+5)(k-1)=y_1,\notag
 \end{align}with equality  if and only if $n_1=\cdots =n_r=3$, i.e., $G\cong G^*\in\mathcal{G}_8$.

 If $a=1$, then by  (\ref{3.55}), we have
 \begin{align}
 	m(G)&\leq m(G^*)\notag\\
 	&=\frac{1}{6}k(k+5)\times3^{(n-k)/3-1}
 +\frac{1}{3}(k-1)+\frac{1}{18}(k+2)(k+5)\left[\frac{1}{3}(n-k)+1\right] =y_2,\notag
 \end{align}with equality if and only if $G\cong G^*\in\mathcal{G}_1$.

 {\it Subcase 2.3.} $n\equiv2$ (mod 3). Again,  suppose (\ref{3.55}) attains the maximum. By Lemma \ref{le2.15}, we have $b=a+1=r$, i.e., $n_r\equiv1$ (mod 3), $n_i\equiv0$ (mod 3) for $i=1,\ldots,r-1$.

If $a=0$, then we have
 \begin{align*}
 	m(G)\leq m(G^*)&= \frac{1}{6}k(k+5)\times3^{(n-k-1)/3-1}+\frac{1}{18}(k+2)(k+5)=y_3,
 \end{align*}with equality if and only if $G\in\mathcal{G}_2$.

Now  suppose $a\ge 1$. Then we have $4\le k\le n-7$, which leads to $n\ge 11$.
Let
$$
\phi(x)=\frac{1}{18}(k+2)(k+5)\times\left[3^{(n-k-1)/3-1}+3^{(n-k-x)/3}\right]
+\frac{1}{9}(k-1)(n+5-x)\times3^{(x-1)/3-1}.$$
 One may check that $\phi'(4)<0$, $\phi'(n-k)>0$ and  $\phi''(x)>0$ for $x\in[4\text{, }n-k]$. Therefore,   $\phi(x)$ firstly decreases and then increases when $x\in[4\text{, }n-k]$.

 Notice that $\phi(n-k)=y_3$, $\phi(4)=(k+2)(k+5)\cdot3^{(n-k-1)/3-3}+(k-1)(n+1)/9$, and
 \begin{align}\label{y4y5}
 	y_4>\phi(4)~\text{ if and only if }~(k+2)(k+5)\times3^{(n-k-1)/3}>6(k-1)(n-2).
 \end{align}Since $\varphi(x)=(x+5)\times3^{(n-x-1)/3}$ is strictly decreasing when $4\le x\le n-7$, we get $$\varphi(x)\ge\varphi(n-7)=9(n-2)>6(n-2).$$ Therefore, if $4\le k\le n-7$, we have $$(k+2)(k+5)\times3^{(n-k-1)/3}>6(k-1)(n-2).$$ By (\ref{y4y5}) we have $\phi(4)<y_4$.

 If $a>1$, then we have
 \begin{align}
 	m(G)&\leq m(G^*)\\
 &= \frac{1}{18}(k+2)(k+5)(3^{(n_r-1)/3-1}+1)\prod\limits_{i=1}^{r-1}(3^{n_i/3-1}+\frac{n_i}{3}+1)\notag\\
 	&~~~+\frac{1}{9}(k-1)(k+2)\times3^{(n_r-1)/3-1}\prod\limits_{i=1}^{r-1}3^{n_i/3-1}\notag\\
 	&~~~+\frac{1}{3}(k-1)\times3^{(n_r-1)/3-1}\Big(\prod\limits_{i=1}^{r-1}3^{n_i/3-1}+\sum_{j=1}^{r-1}\prod\limits_{1\leq i\le r-1,i\neq j}3^{n_i/3-1}\Big)\notag\\
 	&\leq\left[\frac{1}{18}(k+2)(k+5)\times3^{(n-k-n_r)/3}+\frac{1}{9}(k-1)(n+5-n_r)\right]\times3^{(n_r-1)/3-1}\notag\\
 	&~~~+\frac{1}{18}(k+2)(k+5)\prod\limits_{i=1}^{r-1}(3^{n_i/3-1}+\frac{n_i}{3}+1)\tag{by Lemma \ref{lem2.10} (ii)}\\
 	&\leq\frac{1}{18}(k+2)(k+5)\times3^{(n-k-1)/3-1}+\frac{1}{9}(k-1)(n+5-n_r)\times3^{(n_r-1)/3-1}\notag\\
 	&~~~+\frac{1}{18}(k+2)(k+5)\times3^{(n-k-n_r)/3}\notag\\
 	&\leq\max\{\phi(4),\phi(n-k-6)\}<y_4.\notag
 \end{align}
 If $a=1$, i.e., $r=2$ and $n_1+n_r=n-k$, then we have
 \begin{align}
 	m(G)&\leq m(G^*)\notag\\
 &= \frac{1}{18}(k+2)(k+5)(3^{(n_r-1)/3-1}+1)(3^{n_1/3-1}+\frac{n_1}{3}+1)\notag\\
 	&~~~+\frac{1}{9}(k-1)(k+2)\times3^{(n_r-1)/3-1}\times3^{n_1/3-1}+\frac{1}{3}(k-1)\times3^{(n_r-1)/3-1}\times(3^{n_1/3-1}+1)\notag\\
 	&<\left[\frac{1}{9}(k+2)(k+5)+\frac{1}{27}(k-1)(k+2)+\frac{1}{3}(k-1)\right]\times3^{(n-k-1)/3-1}\notag\\
 	&<\frac{1}{18}(k+2)(k+5)\times3^{(n-k-1)/3}+\frac{1}{3}(k-1)=y_4.\notag
 \end{align}

Now we are ready to make the conclusion of Theorem \ref{th1.3}.
When $n\equiv0$ (mod 3), combining (\ref{equation010}), Subcase 1.1 and Subcase 2.1, we have $$m(G)\le \frac{1}{18}(k+2)(k+5)\times3^{(n-k-2)/3},$$ with equality if and only if $G\in\mathcal{G}_6\cup\mathcal{G}_7$.

 When $n\equiv1$ (mod 3), combining (\ref{equation010}), Subcase 1.2 and  Subcase 2.2, we have \begin{equation*}\label{eqh0424}
 m(G)\le\max\{y_1,y_2\},
   \end{equation*}
 with  equality   if and only if $G\in\mathcal{M}_i$ and $y_i=\max\{y_1,y_2\}$ with $i\in\{1,2\}$, where $\mathcal{M}_1=\mathcal{G}_8,~\mathcal{M}_2=\mathcal{G}_1.$

 When $n\equiv2$ (mod 3), combining (\ref{equation010}),  Subcase 1.3 and  Subcase 2.3,  we have
  \begin{equation*}\label{eqh0425}
  m(G)\le\max\{y_3,y_4\},
   \end{equation*}
  with equality   if and only if $G\in\mathcal{M}_i$ and  $y_i=\max\{y_3,y_4\}$ with $i\in\{3,4\}$, where $\mathcal{M}_3=\mathcal{G}_2,~\mathcal{M}_4=\mathcal{G}_9.$

   This completes the proof of Theorem \ref{th1.3}. \hfill\qedsymbol\\

   \par

{\it Remark.} If $n=k+1$, $\mathcal{P}(n,k)$ consists of a unique graph. If $n=k+2$, $\mathcal{P}(n,k)$ consists of two graphs $G_1$ and $G_2$ with $m(G_1)=m(G_2)$; see  Figure \ref{fig:35}.     If $n=k+3$, $\mathcal{P}(n,k)$ consists of four graphs $G_3$, $G_4$, $G_5$ and $G_6$; see Figure \ref{fig:35}. A direct computation shows that    $m(G_3)=m(G_4)$ attains the maximum for the case $k\equiv 0$ (\mod 3), and  $m(G_3)$ attains the maximum for the case $k\not\equiv 0$ (\mod 3).
\begin{figure}[H]
	\centering
	\subfigure{
		\begin{tikzpicture}[scale=0.7]
			\draw[fill=black](0,0.8) circle(0.08)
			(0,1.6) circle(0.08)
			(0.6,1.6) circle(0.08);
			\draw[fill=black](0.565685425,0.565685425) circle(0.08)
			(-0.565685425,0.565685425) circle(0.08);
			\draw(0,0.8)--(0,1.6)
			(0,1.6)--(0.6,1.6);
			\draw[dashed](0,0) circle(0.8);
			\draw(0.565685425,0.565685425) arc(45:135:0.8);
			\node at (0,-1.3){$G_1$};
		\end{tikzpicture}
	}\hspace{0.5cm}
	\subfigure{
		\begin{tikzpicture}[scale=0.7]
			\draw[fill=black](0,0.8) circle(0.08)
			(-0.6,1.6) circle(0.08)
			(0.6,1.6) circle(0.08);
			\draw[fill=black](0.565685425,0.565685425) circle(0.08)
			(-0.565685425,0.565685425) circle(0.08);
			\draw(0,0.8)--(-0.6,1.6)
			(0.6,1.6)--(0,0.8);
			\draw[dashed](0,0) circle(0.8);
			\draw(0.565685425,0.565685425) arc(45:135:0.8);
			\node at (0,-1.3){$G_2$};
		\end{tikzpicture}
	}\hspace{0.5cm}
	\subfigure{
		\begin{tikzpicture}[scale=0.7]
			\draw[fill=black](0,0.8) circle(0.08)
			(0,1.6) circle(0.08)
			(0.6,1.6) circle(0.08)
			(1.2,1.6) circle(0.08);
			\draw[fill=black](0.565685425,0.565685425) circle(0.08)
			(-0.565685425,0.565685425) circle(0.08);
			\draw(0,0.8)--(0,1.6)
		     (0,1.6)--(1.2,1.6);
			\draw[dashed](0,0) circle(0.8);
			\draw(0.565685425,0.565685425) arc(45:135:0.8);
			\node at (0,-1.3){$G_3$};
		\end{tikzpicture}
	}\hspace{0.5cm}
	\subfigure{
		\begin{tikzpicture}[scale=0.7]
			\draw[fill=black](0,0.8) circle(0.08)
			(0,1.6) circle(0.08)
			(0.6,1.6) circle(0.08)
			(-0.6,1.6) circle(0.08);
			\draw[fill=black](0.565685425,0.565685425) circle(0.08)
			(-0.565685425,0.565685425) circle(0.08);
			\draw(0,0.8)--(0,1.6)
			(0.6,1.6)--(-0.6,1.6);
			\draw[dashed](0,0) circle(0.8);
			\draw(0.565685425,0.565685425) arc(45:135:0.8);
			\node at (0,-1.3){$G_4$};
		\end{tikzpicture}
	}\hspace{0.5cm}
	\subfigure{
		\begin{tikzpicture}[scale=0.7]
			\draw[fill=black](0,0.8) circle(0.08)
			(-0.6,1.6) circle(0.08)
			(0.6,1.6) circle(0.08)
			(0,1.6) circle(0.08);
			\draw[fill=black](0.565685425,0.565685425) circle(0.08)
			(-0.565685425,0.565685425) circle(0.08);
			\draw(0,0.8)--(-0.6,1.6)
			(0,1.6)--(0,0.8)
			(0.6,1.6)--(0,1.6);
			\draw[dashed](0,0) circle(0.8);
			\draw(0.565685425,0.565685425) arc(45:135:0.8);
			\node at (0,-1.3){$G_5$};
		\end{tikzpicture}
	}\hspace{0.5cm}
	\subfigure{
		\begin{tikzpicture}[scale=0.7]
			\draw[fill=black](0,0.8) circle(0.08)
			(-0.6,1.6) circle(0.08)
			(0.6,1.6) circle(0.08)
			(0,1.6) circle(0.08);
			\draw[fill=black](0.565685425,0.565685425) circle(0.08)
			(-0.565685425,0.565685425) circle(0.08);
			\draw(0,0.8)--(-0.6,1.6)
			(0,1.6)--(0,0.8)
			(0.6,1.6)--(0,0.8);
			\draw[dashed](0,0) circle(0.8);
			\draw(0.565685425,0.565685425) arc(45:135:0.8);
			\node at (0,-1.3){$G_6$};
		\end{tikzpicture}
	}
		\caption{The potted graphs of order $n\le k+3$ }
		\label{fig:35}
	\end{figure}
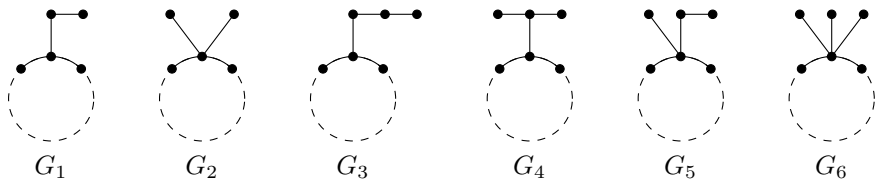

\section{Conclusion}

We determine the maximum number of maximum dissociation sets in potted graphs with a given cycle as well as the extremal graphs attaining this maximum number. It would be interesting to consider the same problem on unicyclic graphs, which we leave for future works.

\section*{Acknowledgement}
This work was supported by the National Natural Science Foundation of China (No. 12171323),  Guangdong Basic and Applied Basic Research
Foundation (No. 2022A1515011995) and the Science and  Technology Foundation of Shenzhen City (No. JCYJ20210324095813036).

}
\begin{appendix}
\section*{\bf Appendix A: the proofs of Lemma \ref{lem2.9} and Lemma \ref{lem2.10}}

{\it Proof of Lemma \ref{lem2.9}.}~~
$(i)$ Let $t= (m-l-1)/3$. We use induction on $t$. The result is obvious for $t=2$. Now we assume $t\ge 3$ and the result is true for $t-1$. If $i=t$, we have $i=t$ and $k_1=\cdots=k_i=1$.   By direct computation we have (\ref{2.5}).

Next we suppose $i<t$. Then there exists some $j\in\{1,2,\ldots,i\}$ such that $k_j\geq2$. Without loss of generality,  we may assume $k_1\geq2$.  Let $k_1^{'}=k_1-1$. Then $k'_1+k_2+\cdots+k_i=t-1$. By the induction hypothesis, we have
\begin{align*}
	\frac{1}{3}(l+1)\prod\limits_{j=1}^if(k_j)+\prod\limits_{j=1}^ig(k_j)&=\frac{1}{3}(l+1)\left[3f(k_1^{'})-(2k_1^{'}+1)\right]\prod\limits_{j=2}^if(k_j)+3g(k_1^{'})\prod\limits_{j=2}^ig(k_j)\\
	&\leq3\left[\frac{1}{3}(l+1)\cdot  g(t)+1\right]-\frac{1}{3}(l+1)(2k_1^{'}+1)\prod\limits_{j=2}^if(k_j)\\
	&\leq\frac{1}{3}(l+1)\cdot g(t+1)+3-\frac{1}{3}(l+1)(2t-1)\tag{by (\ref{2.3})}\\
	&<\frac{1}{3}(l+1)\cdot g(t+1)+1.
\end{align*}

    $(ii)$ Let $t=(m-l)/3$. We use induction on $t$. A direct computation verifies the case $t=2$. Now we assume $t\ge 3$ and the result is true for $t-1$. If $i=t$, we have $k_1=\cdots=k_i=1$. Again, by a direct computation we have the equality in (\ref{2.4}).

Next we suppose $i<t$. Then there exists some $j\in\{1,2,\ldots,i\}$ such that $k_j\geq2$. Without loss of generality,  we may assume $k_1\geq2$.  Let $k_1^{'}=k_1-1$. Then $k'_1+k_2+\cdots+k_i=t-1$. By the induction hypothesis, we have
    \begin{align*}
        &\frac{1}{3}(l+1)\prod\limits_{j=1}^if(k_j)+\frac{1}{3}(2l-1)\prod\limits_{j=1}^ig(k_j)+\sum\limits_{r=1}^i\prod\limits_{j\neq r}g(k_j)\\
        &=\frac{1}{3}(l+1)\times\left[3f(k_1^{'})-(2k_1^{'}+1)\right]\prod\limits_{j=2}^if(k_j)+\frac{1}{3}(2l-1)\times3g(k_1^{'})\prod\limits_{j=2}^ig(k_j)\\
        &\quad~~+3\left[g(k_1^{'})\sum\limits_{r=2}^i\prod\limits_{j\neq r,1}g(k_j)+\prod\limits_{j\neq1}g(k_j)\right]-2\prod\limits_{j\neq1}g(k_j)\\
        &\leq3\left[\frac{1}{3}(l+1)\cdot g(t)+\frac{2}{3}(l-2)+t\right]-\frac{1}{3}(l+1)(2k_1^{'}+1)\prod\limits_{j=2}^if(k_j)-2\prod\limits_{j\neq 1}g(k_j)\\
        &\leq\frac{1}{3}(l+1)\cdot g(t+1)+2(l-2)+3t-\frac{1}{3}(l+1)(2t-1)-2\tag{by (\ref{2.3})}\\
        &=\frac{1}{3}(l+1)\cdot g(t+1)+\frac{2}{3}(l-2)+t+1+\frac{4}{3}(l-2)+(2t-1)\left[1-\frac{1}{3}(l+1)\right]-2\\
        &<\frac{1}{3}(l+1)\cdot g(t+1)+\frac{2}{3}(l-2)+t+1\\
        &=\frac{1}{3}(l+1)\cdot g\left(\frac{1}{3}(m-l)+1\right)+\frac{1}{3}(m+l-1).
    \end{align*}\hfill\qedsymbol
   \par

 {\it Proof of Lemma \ref{lem2.10}.}
 $(i)$ Let $t= (m-l-1)/3$. We use induction on $t$. A direct computation verifies the case  $t=2$. Now we assume $t\ge 3$ and the result is true for $t-1$. If $i=t$, we have  $k_1=\cdots=k_i=1$.   By direct computation we have the equality  in (\ref{2.7}).

 Next we suppose $i<t$. Then there exists some $j\in\{1,2,\ldots,i\}$ such that $k_j\geq2$. Without loss of generality,  we may assume $k_1\geq2$.  Let $k_1^{'}=k_1-1$. Then $k'_1+k_2+\cdots+k_i=t-1$. By the induction hypothesis, we have
 \begin{align*}
 	&\frac{1}{18}(l+2)(l+5)\prod\limits_{j=1}^if(k_j)+\frac{1}{3}(l-1)\prod\limits_{j=1}^ig(k_j)\\
 	&=\frac{1}{18}(l+2)(l+5)\left[3f(k_1^{'})-(2k^{'}_1+1)\right]\prod\limits_{j=2}^if(k_j)+\frac{1}{3}(l-1)\times3g(k_1^{'})\prod\limits_{j=2}^ig(k_j)\\
 	&\leq3\left[\frac{1}{18}(l+2)(l+5)\cdot g(t)+\frac{1}{3}(l-1)\right]-\frac{1}{18}(l+2)(l+5)(2k_1^{'}+1)\prod\limits_{j=2}^if(k_j)\\
 	&\leq\frac{1}{18}(l+2)(l+5)\cdot g(t+1)+l-1-\frac{1}{18}(l+2)(l+5)\\
 	&<\frac{1}{18}(l+2)(l+5)\cdot g(t+1)+\frac{1}{3}(l-1).
 \end{align*}

    $(ii)$ Let $t=(m-l)/3$. We use induction on $t$. A direct computation verifies the case $t=2$. Now we assume $t\ge 3$ and the result is true for $t-1$. If $i=t$, we have $k_1=\cdots=k_i=1$. By a direct computation we have the equality in (\ref{2.6}).

Next we suppose $i<t$. Then there exists some $j\in\{1,2,\ldots,i\}$ such that $k_j\geq2$. Without loss of generality,  we may assume $k_1\geq2$.  Let $k_1^{'}=k_1-1$. Then $k'_1+k_2+\cdots+k_i=t-1$. By the induction hypothesis, we have
    \begin{align*}
        &\frac{1}{18}(l+2)(l+5)\prod\limits_{j=1}^if(k_j)+\frac{1}{9}(l-1)(l+2)\prod\limits_{j=1}^ig(k_j)+\frac{1}{3}(l-1)\left[\prod\limits_{j=1}^ig(k_j)+\sum\limits_{r=1}^i\prod\limits_{j\neq r}g(k_j)\right]\\
        &=\frac{1}{18}(l+2)(l+5)\left[3f(k_1^{'})-(2k_1^{'}+1)\right]\prod\limits_{j=2}^if(k_j)+\frac{1}{9}(l-1)(l+2)\times3g(k_1^{'})\prod\limits_{j=2}^ig(k_j)\\
        &\quad~~+\frac{1}{3}(l-1)\left[3g(k_1^{'})\prod\limits_{j=2}^ig(k_j)+3\Big(g(k_1^{'})\sum\limits_{r=2}^i\prod\limits_{j\neq r,1}g(k_j)+\prod\limits_{j=2}^ig(k_j)\Big)-2\prod\limits_{j=2}^ig(k_j)\right]\\
           &\leq3\left[\frac{1}{18}(l+2)(l+5)\cdot g(t)+\frac{1}{9}(l-1)(l+2)+\frac{1}{3}(l-1)t\right]\\
           &\quad~~-\frac{1}{18}(l+2)(l+5)(2k_1^{'}+1)\prod\limits_{j=2}^if(k_j)
           -\frac{2}{3}(l-1)\prod\limits_{j=2}^ig(k_j)\\
           &\leq\frac{1}{18}(l+2)(l+5)\cdot g(t+1)+\frac{1}{3}(l-1)(l+2)+(l-1)t-\frac{1}{18}(l+2)(l+5)(2t-1)-\frac{2}{3}(l-1)\tag{by (\ref{2.3})}\\
           &=\frac{1}{18}(l+2)(l+5)\cdot g(t+1)+\frac{1}{9}(l-1)(l+2)+\frac{1}{3}(l-1)(t+1)\\
           &\quad~~+\frac{2}{9}(l-1)(l+2)+(2t-1)\left[\frac{1}{3}(l-1)-\frac{1}{18}(l+2)(l+5)\right]-\frac{2}{3}(l-1)\\
           &<\frac{1}{18}(l+2)(l+5)\cdot g(t+1)+\frac{1}{9}(l-1)(l+2)+\frac{1}{3}(l-1)(t+1)\\
           &=\frac{1}{18}(l+2)(l+5)\cdot g\left(\frac{1}{3}(m-l)+1\right)+\frac{1}{9}(m+5)(l-1).
    \end{align*}\hfill\qedsymbol
\section*{Appendix B: The values of $\overline{m}(T-v)$ and $\overline{m}(T-N(v))$ for some graphs}

In this section, we list all the possibilities of $\overline{m}(T-v)$ and $\overline{m}(T-N(v))$ for a tree $T\in \mathcal{H}\cup \left(\cup_{i=2}^5\mathcal{T}_i\right)$. We partially label  the vertices of $T$ as follows, in which we divide the trees in $\mathcal{H}$ into three types: (a), (b), (c).
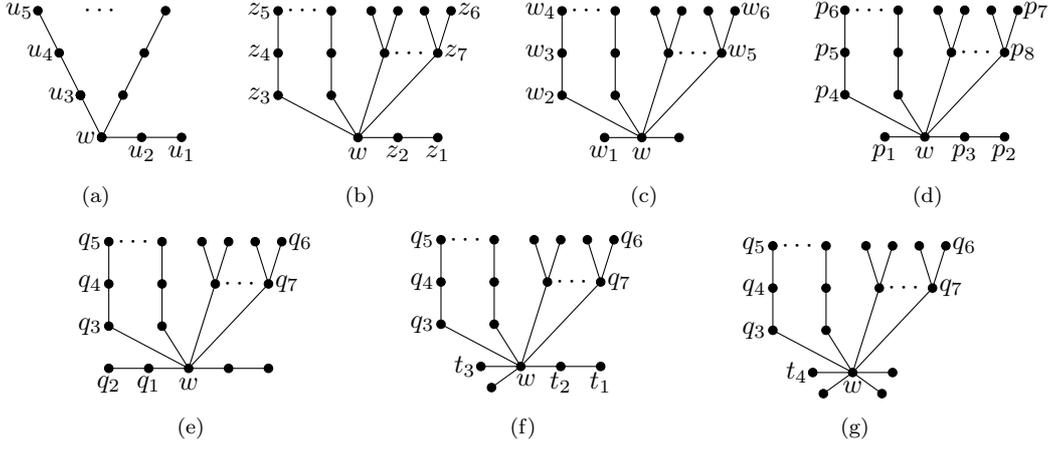
\begin{figure}[H]
	\centering
	\subfigure[]{
		\begin{tikzpicture}[scale=0.7]
			\draw[fill=black](0.8,3.2) circle(0.08)
			(0,1.6) circle(0.08)
			(0.4,2.4) circle(0.08)
			(-0.4,2.4) circle(0.08)
			(0,1.6) circle(0.08)
			(0.75,1.6) circle(0.08)
			(1.5,1.6) circle(0.08)
			(-0.8,3.2) circle(0.08)
			(1.2,4) circle(0.08)
			(-1.2,4) circle(0.08)
			(-0.4,2.4) circle(0.08);
			\draw(0,1.6)--(1.2,4)
			(0,1.6)--(1.5,1.6)
			(0,1.6)--(-1.2,4);
			\node at(0,4){$\cdots$};
			\node at (-0.3,1.6){$w$};
			\node at (0.75,1.3){$u_2 $};
			\node at (1.5,1.3){$u_1 $};
			\node at (-0.75,2.4){$u_3$};
			\node at (-1.15,3.2){$u_4 $};
			\node at (-1.55,4){$u_5 $};
		\end{tikzpicture}
	}\hspace{0.1cm}
	\subfigure[]{
		\begin{tikzpicture}[scale=0.7]
			\draw[fill=black](0,0) circle(0.08)
			(-0.5,0.8) circle(0.08)
			(-0.5,1.6) circle(0.08)
			(-0.5,2.4) circle(0.08)
			(-1.5,0.8) circle(0.08)
			(-1.5,1.6) circle(0.08)
			(-1.5,2.4) circle(0.08)
			(0.5,1.6) circle(0.08)
			(1.5,1.6) circle(0.08)
			(0.25,2.4) circle(0.08)
			(0.75,2.4) circle(0.08)
			(1.25,2.4) circle(0.08)
			(1.75,2.4) circle(0.08)
			(0.75,0) circle(0.08)
			(1.5,0) circle(0.08);
			\draw(0,0)--(1.5,0)
			(0,0)--(-1.5,0.8)
			(0,0)--(-0.5,0.8)
			(0,0)--(0.5,1.6)
			(0,0)--(1.5,1.6)
			(0.5,1.6)--(0.25,2.4)
			(0.5,1.6)--(0.75,2.4)
			(1.5,1.6)--(1.25,2.4)
			(1.5,1.6)--(1.75,2.4)
			(-1.5,0.8)--(-1.5,2.4)
			(-0.5,2.4)--(-0.5,0.8);
			\node at(-1,2.4){$\cdots$};
			\node at(1,1.6){$\cdots$};
			\node at (0,-0.3){$w$};
			\node at (0.75,-0.3){$z_2$};
			\node at (1.5,-0.3){$z_1$};
			\node at (-1.85,0.8){$z_3$};
			\node at (-1.85,1.6){$z_4$};
			\node at (-1.85,2.4){$z_5$};
			\node at (1.85,1.6){$z_7$};
			\node at (2.1,2.4){$z_6$};
		\end{tikzpicture}
	}\hspace{0.1cm}
	\subfigure[]{\begin{tikzpicture}[scale=0.7]
			\draw[fill=black](0,0) circle(0.08)
			(-0.5,0.8) circle(0.08)
			(-0.5,1.6) circle(0.08)
			(-0.5,2.4) circle(0.08)
			(-1.5,0.8) circle(0.08)
			(-1.5,1.6) circle(0.08)
			(-1.5,2.4) circle(0.08)
			(0.5,1.6) circle(0.08)
			(1.5,1.6) circle(0.08)
			(0.25,2.4) circle(0.08)
			(0.75,2.4) circle(0.08)
			(1.25,2.4) circle(0.08)
			(1.75,2.4) circle(0.08)
			(-0.7,0) circle(0.08)
			(0.7,0) circle(0.08);
			\draw(0,0)--(0.7,0)
			(0,0)--(-0.7,0)
			(0,0)--(-1.5,0.8)
			(0,0)--(-0.5,0.8)
			(0,0)--(0.5,1.6)
			(0,0)--(1.5,1.6)
			(0.5,1.6)--(0.25,2.4)
			(0.5,1.6)--(0.75,2.4)
			(1.5,1.6)--(1.25,2.4)
			(1.5,1.6)--(1.75,2.4)
			(-1.5,0.8)--(-1.5,2.4)
			(-0.5,2.4)--(-0.5,0.8);
			\node at(-1,2.4){$\cdots$};
			\node at(1,1.6){$\cdots$};
			\node at(0,-0.3){$w$};
			\node at(-0.7,-0.3){$w_1$};
			\node at(-1.9,0.8){$w_2$};
			\node at(-1.9,1.6){$w_3$};
			\node at(-1.9,2.4){$w_4$};
			\node at(1.9,1.6){$w_5$};
			\node at(2.15,2.4){$w_6$};
	\end{tikzpicture}}\hspace{0.1cm}
\subfigure[]{
	\begin{tikzpicture}[scale=0.7]
		\draw[fill=black](0,0) circle(0.08)
		(-0.5,0.8) circle(0.08)
		(-0.5,1.6) circle(0.08)
		(-0.5,2.4) circle(0.08)
		(-1.5,0.8) circle(0.08)
		(-1.5,1.6) circle(0.08)
		(-1.5,2.4) circle(0.08)
		(0.5,1.6) circle(0.08)
		(1.5,1.6) circle(0.08)
		(0.25,2.4) circle(0.08)
		(0.75,2.4) circle(0.08)
		(1.25,2.4) circle(0.08)
		(1.75,2.4) circle(0.08)
		(-0.75,0) circle(0.08)
		(0.75,0) circle(0.08)
		(1.5,0) circle(0.08);
		\draw(0,0)--(1.5,0)
		(0,0)--(-0.75,0)
		(0,0)--(-1.5,0.8)
		(0,0)--(-0.5,0.8)
		(0,0)--(0.5,1.6)
		(0,0)--(1.5,1.6)
		(0.5,1.6)--(0.25,2.4)
		(0.5,1.6)--(0.75,2.4)
		(1.5,1.6)--(1.25,2.4)
		(1.5,1.6)--(1.75,2.4)
		(-1.5,0.8)--(-1.5,2.4)
		(-0.5,2.4)--(-0.5,0.8);
		\node at(-1,2.4){$\cdots$};
		\node at(1,1.6){$\cdots$};
		\node at (0,-0.3){$w$};
		\node at (-0.75,-0.3){$p_1$};
		\node at (0.75,-0.3){$p_3$};
		\node at (1.5,-0.3){$p_2$};
		\node at (-1.82,0.8){$p_4$};
		\node at (-1.82,1.6){$p_5$};
		\node at (-1.82,2.4){$p_6$};
		\node at (2.1,2.4){$p_7$};
		\node at (1.85,1.6){$p_8$};
	\end{tikzpicture}
}\hspace{0.1cm}
	\subfigure[]{
		\begin{tikzpicture}[scale=0.7]
			\draw[fill=black](0,0) circle(0.08)
			(-0.5,0.8) circle(0.08)
			(-0.5,1.6) circle(0.08)
			(-0.5,2.4) circle(0.08)
			(-1.5,0.8) circle(0.08)
			(-1.5,1.6) circle(0.08)
			(-1.5,2.4) circle(0.08)
			(0.5,1.6) circle(0.08)
			(1.5,1.6) circle(0.08)
			(0.25,2.4) circle(0.08)
			(0.75,2.4) circle(0.08)
			(1.25,2.4) circle(0.08)
			(1.75,2.4) circle(0.08)
			(-0.75,0) circle(0.08)
			(0.75,0) circle(0.08)
			(1.5,0) circle(0.08)
			(-1.5,0) circle(0.08);
			\draw(-1.5,0)--(1.5,0)
			(0,0)--(-1.5,0.8)
			(0,0)--(-0.5,0.8)
			(0,0)--(0.5,1.6)
			(0,0)--(1.5,1.6)
			(0.5,1.6)--(0.25,2.4)
			(0.5,1.6)--(0.75,2.4)
			(1.5,1.6)--(1.25,2.4)
			(1.5,1.6)--(1.75,2.4)
			(-1.5,0.8)--(-1.5,2.4)
			(-0.5,2.4)--(-0.5,0.8);
			\node at(-1,2.4){$\cdots$};
			\node at(1,1.6){$\cdots$};
			\node at (0,-0.3){$w$};
			\node at (-0.75,-0.3){$q_1$};
			\node at (-1.5,-0.3){$q_2$};
			\node at (-1.85,0.8){$q_3$};
			\node at (-1.85,1.6){$q_4$};
			\node at (-1.85,2.4){$q_5$};
			\node at (1.85,1.6){$q_7$};
			\node at (2.1,2.4){$q_6$};
		\end{tikzpicture}
	}\hspace{0.7cm}
	\subfigure[]{
		\begin{tikzpicture}[scale=0.7]
			\draw[fill=black](0,0) circle(0.08)
			(-0.5,0.8) circle(0.08)
			(-0.5,1.6) circle(0.08)
			(-0.5,2.4) circle(0.08)
			(-1.5,0.8) circle(0.08)
			(-1.5,1.6) circle(0.08)
			(-1.5,2.4) circle(0.08)
			(0.5,1.6) circle(0.08)
			(1.5,1.6) circle(0.08)
			(0.25,2.4) circle(0.08)
			(0.75,2.4) circle(0.08)
			(1.25,2.4) circle(0.08)
			(1.75,2.4) circle(0.08)
			(-0.55,-0.4) circle(0.08)
			(0.75,0) circle(0.08)
			(1.5,0) circle(0.08)
			(-0.75,0) circle(0.08);
			\draw(0,0)--(1.5,0)
			(0,0)--(-0.55,-0.4)
			(0,0)--(-1.5,0.8)
			(0,0)--(-0.5,0.8)
			(0,0)--(0.5,1.6)
			(0,0)--(1.5,1.6)
			(0.5,1.6)--(0.25,2.4)
			(0.5,1.6)--(0.75,2.4)
			(1.5,1.6)--(1.25,2.4)
			(1.5,1.6)--(1.75,2.4)
			(-1.5,0.8)--(-1.5,2.4)
			(-0.5,2.4)--(-0.5,0.8)
			(-0.75,0)--(0,0);
			\node at(-1,2.4){$\cdots$};
			\node at(1,1.6){$\cdots$};
			\node at (0.1,-0.25){$w$};
			\node at (-1.85,0.8){$q_3$};
			\node at (-1.85,1.6){$q_4$};
			\node at (-1.85,2.4){$q_5$};
			\node at (1.85,1.6){$q_7$};
			\node at (2.1,2.4){$q_6$};
			\node at (0.75,-0.3){$t_2$};
			\node at (1.5,-0.3){$t_1$};
			\node at (-1.05,0){$t_3$};
		\end{tikzpicture}
	}\hspace{0.7cm}
	\subfigure[]{
		\begin{tikzpicture}[scale=0.7]
			\draw[fill=black](0,0) circle(0.08)
			(-0.5,0.8) circle(0.08)
			(-0.5,1.6) circle(0.08)
			(-0.5,2.4) circle(0.08)
			(-1.5,0.8) circle(0.08)
			(-1.5,1.6) circle(0.08)
			(-1.5,2.4) circle(0.08)
			(0.5,1.6) circle(0.08)
			(1.5,1.6) circle(0.08)
			(0.25,2.4) circle(0.08)
			(0.75,2.4) circle(0.08)
			(1.25,2.4) circle(0.08)
			(1.75,2.4) circle(0.08)
			(-0.55,-0.4) circle(0.08)
			(0.55,-0.4) circle(0.08)
			(0.75,0) circle(0.08)
			(-0.75,0) circle(0.08);
			\draw(0,0)--(0.55,-0.4)
			(0,0)--(-0.55,-0.4)
			(0,0)--(-1.5,0.8)
			(0,0)--(-0.5,0.8)
			(0,0)--(0.5,1.6)
			(0,0)--(1.5,1.6)
			(0.5,1.6)--(0.25,2.4)
			(0.5,1.6)--(0.75,2.4)
			(1.5,1.6)--(1.25,2.4)
			(1.5,1.6)--(1.75,2.4)
			(-1.5,0.8)--(-1.5,2.4)
			(-0.5,2.4)--(-0.5,0.8)
			(-0.75,0)--(0.75,0);
			\node at(-1,2.4){$\cdots$};
			\node at(1,1.6){$\cdots$};
			\node at (0,-0.27){$w$};
			\node at (-1.85,0.8){$q_3$};
			\node at (-1.85,1.6){$q_4$};
			\node at (-1.85,2.4){$q_5$};
			\node at (1.85,1.6){$q_7$};
			\node at (2.1,2.4){$q_6$};
			\node at (-1.05,0){$t_4$};
		\end{tikzpicture}
	}
	\caption{Seven labeled trees $T$ of order $n$}
	\label{fig:enter-label}
\end{figure}

By symmetry of the vertices,  we have the following possibilities for the values  of $\overline{m}(T-v)$ and $\overline{m}(T-N(v))$ in Figure \ref{fig:enter-label}.\vspace{-0.4cm}
\begin{table}[H]
	\renewcommand\arraystretch{1.65}
	\begin{center}
		\caption{The values of $\overline{m}(T-v)$ and $\overline{m}(T-N(v))$ for Figure \ref{fig:enter-label} (a).}\vspace{0.25cm}
		\label{table:1}
		\begin{tabular}{|c|c|c|c|c|c|c|}
			\hline   \text{$v$}&$w$ & \text{$u_1 $} & \text{$u_2 $} & \text{$u_3 $}& \text{$u_4 $}& \text{$u_5$}\\
			\hline   $\overline{m}(T-v)$&$3^{n/3-1}$ & 1 & $n/3$ & $3^{n/3-2}+n/3$ & $3^{n/3-2}+1$ & $3^{n/3-2}$\\
			\hline   $\overline{m}(T-N(v))$&1 & $n/3$ & 0 &  $3^{n/3-2}$ & 0 & $3^{n/3-2}+1$\\
			\hline
		\end{tabular}
	\end{center}\vspace{-0.7cm}
\end{table}
\begin{table}[H]
	\renewcommand\arraystretch{1.65}
	\begin{center}
		\caption{The values of $\overline{m}(T-v)$ and $\overline{m}(T-N(v))$ for Figure \ref{fig:enter-label} (b).}\vspace{0.25cm}
		\label{EQUA00}
		\begin{tabular}{|c|c|c|c|c|c|c|c|c|}
			\hline   \text{$v=$}&\text{$w$} & \text{$z_1$} & \text{$z_2$} & \text{$z_3$}& \text{$z_4$}& \text{$z_5$} &\text{$z_6$} & \text{$z_7$} \\
			\hline   $\overline{m}(T-v)$&$3^{\frac{n}{3}-1}$ & 1 & $<\frac{n}{3}$ & $3^{\frac{n}{3}-2}+2$ & $3^{\frac{n}{3}-2}+1$ & $3^{\frac{n}{3}-2}$ & $3^{\frac{n}{3}-2}$ & $\leq3^{\frac{n}{3}-2}+\frac{n}{3}$ \\
			\hline   $\overline{m}(T-N(v))$&1 & $<\frac{n}{3}$ & 0 & $3^{\frac{n}{3}-2}$ & 0 & $3^{\frac{n}{3}-2}+1$ & $\leq3^{\frac{n}{3}-2}+\frac{n}{3}$ & 0 \\
			\hline
		\end{tabular}
	\end{center}\vspace{-0.7cm}
\end{table}
\begin{table}[H]
	\renewcommand\arraystretch{1.65}
	\begin{center}
		\caption{The values of $\overline{m}(T-v)$ and $\overline{m}(T-N(v))$ for Figure \ref{fig:enter-label} (c).}\vspace{0.25cm}
		\label{table:0}
		\begin{tabular}{|c|c|c|c|c|c|c|c|}
			\hline   \text{$v$}&$w$ & \text{$w_1 $} & \text{$w_2 $} & \text{$w_3 $}& \text{$w_4 $}& \text{$w_5$}& \text{$w_6$}\\
			\hline   $\overline{m}(T-v)$&$3^{n/3-1}$ & 1 & $3^{n/3-2}+2$ & $3^{n/3-2}$ & $3^{n/3-2}$ & $3^{n/3-2}+2$&$3^{n/3-2}$\\
			\hline   $\overline{m}(T-N(v))$&0 & $3^{n/3-1}$ & $3^{n/3-2}$ & 0 & $3^{n/3-2}$ & 0&$3^{n/3-2}+2$\\
			\hline
		\end{tabular}
	\end{center}\vspace{-0.7cm}
\end{table}
\begin{table}[H]
	\renewcommand\arraystretch{1.65}
	\begin{center}
		\caption{The values of $\overline{m}(T-v)$ and $\overline{m}(T-N(v))$ for Figure \ref{fig:enter-label} (d).}\vspace{0.25cm}
		\label{EQUA0}
		\begin{tabular}{|c|c|c|c|c|c|c|c|c|c|}
			\hline   \text{$v=$}&\text{$w$} & \text{$p_1$} & \text{$p_2$} & \text{$p_3$}& \text{$p_4$}& \text{$p_5$} &\text{$p_6$} & \text{$p_7$} & \text{$p_8$}\\
			\hline   $\overline{m}(T-v)$&$3^{\frac{n-1}{3}-1}$ & 0 & 0 & 1 & $3^{\frac{n-1}{3}-2}+1$ & $3^{\frac{n-1}{3}-2}$ & $3^{\frac{n-1}{3}-2}$ & $3^{\frac{n-1}{3}-2}$ & $3^{\frac{n-1}{3}-2}+1$\\
			\hline   $\overline{m}(T-N(v))$&0 & $3^{\frac{n-1}{3}-1}$ & 1 & 0 &  $3^{\frac{n-1}{3}-2}$ & 0 & $3^{\frac{n-1}{3}-2}$ & $3^{\frac{n-1}{3}-2}+1$ &0\\
			\hline
		\end{tabular}
	\end{center}\vspace{-0.7cm}
\end{table}
\begin{table}[H]
	\renewcommand\arraystretch{1.65}
	\begin{center}
		\caption{The values of $\overline{m}(T-v)$ and $\overline{m}(T-N(v))$ for Figure \ref{fig:enter-label} (e).}\vspace{0.25cm}
		\label{EQU}
		\begin{tabular}{|c|c|c|c|c|c|c|c|c|}
			\hline   \text{$v=$}&\text{$w$} & \text{$q_1$} & \text{$q_2$} & \text{$q_3$}& \text{$q_4$}& \text{$q_5$} &\text{$q_6$} & \text{$q_7$}\\
			\hline   $\overline{m}(T-v)$&$3^{\frac{n-2}{3}-1}$ & 0 & 0 & $3^{\frac{n-2}{3}-2}$ & $3^{\frac{n-2}{3}-2}$ & $3^{\frac{n-2}{3}-2}$ & $3^{\frac{n-2}{3}-2}$ & $3^{\frac{n-2}{3}-2}$ \\
			\hline   $\overline{m}(T-N(v))$&0 & 0 & 0 & $3^{\frac{n-2}{3}-2}$ &  0 & $3^{\frac{n-2}{3}-2}$ & $3^{\frac{n-1}{3}-2}$ & 0\\
			\hline
		\end{tabular}
	\end{center}\vspace{-0.7cm}
\end{table}
\begin{table}[H]
	\renewcommand\arraystretch{1.65}
	\begin{center}
		\caption{The values of $\overline{m}(T-v)$ and $\overline{m}(T-N(v))$ for Figure \ref{fig:enter-label} (f).}\vspace{0.25cm}
		\label{EQUf}
		\begin{tabular}{|c|c|c|c|c|c|c|c|c|c|}
			\hline   \text{$v=$}&\text{$w$} & \text{$t_1$} & \text{$t_2$} &\text{$t_3$} & \text{$q_3$}& \text{$q_4$}& \text{$q_5$} &\text{$q_6$} & \text{$q_7$}\\
			\hline   $\overline{m}(T-v)$&$3^{\frac{n-2}{3}-1}$ & 0 & 0 &0 & $3^{\frac{n-2}{3}-2}$ & $3^{\frac{n-2}{3}-2}$ & $3^{\frac{n-2}{3}-2}$ & $3^{\frac{n-2}{3}-2}$ & $3^{\frac{n-2}{3}-2}$ \\
			\hline   $\overline{m}(T-N(v))$&0 & 0 & 0 &$3^{\frac{n-2}{3}-1}$ & $3^{\frac{n-2}{3}-2}$ &  0 & $3^{\frac{n-2}{3}-2}$ & $3^{\frac{n-1}{3}-2}$ & 0\\
			\hline
		\end{tabular}
	\end{center}\vspace{-0.7cm}
\end{table}
\begin{table}[H]
	\renewcommand\arraystretch{1.65}
	\begin{center}
		\caption{The values of $\overline{m}(T-v)$ and $\overline{m}(T-N(v))$ for Figure \ref{fig:enter-label} (g).}\vspace{0.25cm}
		\label{EQUg}
		\begin{tabular}{|c|c|c|c|c|c|c|c|}
			\hline   \text{$v=$}&\text{$w$} & \text{$t_4$} & \text{$q_3$}& \text{$q_4$}& \text{$q_5$} &\text{$q_6$} & \text{$q_7$}\\
			\hline   $\overline{m}(T-v)$&$3^{\frac{n-2}{3}-1}$ & 0 & $3^{\frac{n-2}{3}-2}$ & $3^{\frac{n-2}{3}-2}$ & $3^{\frac{n-2}{3}-2}$ & $3^{\frac{n-2}{3}-2}$ & $3^{\frac{n-2}{3}-2}$ \\
			\hline   $\overline{m}(T-N(v))$&0 & $3^{\frac{n-2}{3}-1}$ & $3^{\frac{n-2}{3}-2}$ &  0 & $3^{\frac{n-2}{3}-2}$ & $3^{\frac{n-1}{3}-2}$ & 0\\
			\hline
		\end{tabular}
	\end{center}\vspace{-0.7cm}
\end{table}
\end{appendix}

\end{document}